\documentclass[a4paper]{article}

\usepackage[latin1]{inputenc} 
\usepackage{graphicx}
\usepackage[pdftex,bookmarks]{hyperref}
\usepackage{amsmath}    
\usepackage{amssymb}
\usepackage{amsmath,amsthm}
\usepackage{textcomp}
\usepackage{color}
\usepackage[all]{xy}
\usepackage{geometry}

\newtheorem{teo}{Theorem}[section]
\newtheorem{defi}{Definition}[section]
\newtheorem{lemma}{Lemma}[section]

\newtheorem{cor}{Corollary}[section]
\newtheorem{prop}{Proposition}[section]

\newtheorem{rem} {Remark}[section]

\DeclareMathOperator{\im}{im}

\DeclareMathOperator{\dvol}{dvol}

\DeclareMathOperator{\supp}{supp}

\DeclareMathOperator{\id}{Id}
\DeclareMathOperator{\vol}{vol}
\DeclareMathOperator{\reg}{reg}
\DeclareMathOperator{\sing}{sing}

\title{\huge \bf Compact convergence, deformation of the $L^2$-$\overline{\partial}$-complex and canonical $K$-homology classes}
\author{Francesco Bei  \bigskip \\
Dipartimento di Matematica, Sapienza Universit\`a di Roma,\\ E-mail addresses: \ bei@mat.uniroma1.it \     francescobei27@gmail.com }
\date{}

\begin{document}

\maketitle

\begin{abstract}
Let $(X,\gamma)$ be a compact, irreducible Hermitian complex space of complex dimension $m$ and with $\dim(\sing(X))=0$. Let $(F,\tau)\rightarrow X$ be a Hermitian holomorphic vector bundle over $X$ and let us denote by $\overline{\eth}_{F,m,\mathrm{abs}}$ the rolled-up operator of the maximal $L^2$-$\overline{\partial}$-complex of $F$-valued $(m,\bullet)$-forms. Let $\pi:M\rightarrow X$ be a resolution of singularities, $g$ a metric on $M$, $E:=\pi^*F$ and $\rho:=\pi^*\tau$. In this paper, under quite general assumptions on $\tau$, we prove the following equality of analytic $K$-homology classes $[\overline{\eth}_{F,m,\mathrm{abs}}]=\pi_*[\overline{\eth}_{E,m}]$, with $\overline{\eth}_{E,m}$ the rolled-up operator of the $L^2$-$\overline{\partial}$-complex of $E$-valued $(m,\bullet)$-forms on $M$. Our proof is based on functional analytic techniques developed in \cite{KuSh} and provides an explicit  homotopy between the even unbounded Fredholm modules induced by $\overline{\eth}_{F,m,\mathrm{abs}}$ and $\overline{\eth}_{E,m}$.
\end{abstract}



\tableofcontents

\section*{Introduction}
Let $(X,\gamma)$ be a compact and irreducible Hermitian complex space of complex dimension $m$. The existence and geometric interpretation of analytic $K$-homology classes induced by the Hodge-Dolbeault operator have been investigated in various papers, see, e.g., \cite{BeiPiazza}, \cite{BePi2}, \cite{DTYU}, \cite{Haskell} and  \cite{Lott-Todd}. Concerning the existence the main obstacle is due to the lack of a satisfactory picture for the $L^2$ theory of the Hodge-Dolbeault operator $\overline{\partial}+\overline{\partial}^t$. To the best of our knowledge there are only  few cases in which $\overline{\partial}+\overline{\partial}^t$ is known to have self-adjoint extensions with an entirely discrete spectrum. The first is when $m=2$ (with no assumptions on $\sing(X)$) and one considers $\overline{\eth}_{\mathrm{rel}}$, that is the rolled-up  operator of the minimal  $L^2$-$\overline{\partial}$-complex of $(0,\bullet)$-forms, see \cite[Th. 3.1]{FraB}. The second case arises when $\dim(\sing(X))=0$ (with no assumptions on $m$) and one considers $\overline{\eth}_{\mathrm{abs}/\mathrm{rel}}$, that is the rolled-up operator of either the minimal or the maximal  $L^2$-$\overline{\partial}$-complex of $(0,\bullet)$-forms, see \cite[Th. 1.2]{OvRu} and \cite[Th. 5.2]{FBei}. In all of the previous cases it is known that the operator $\overline{\eth}_{\mathrm{abs}/\mathrm{rel}}$ gives rise to an unbounded even Fredholm module and thus to a class $[\overline{\eth}_{\mathrm{abs/rel}}]\in KK_0(C(X),\mathbb{C})$, see \cite[Prop. 3.6, 3.7 and 3.8]{BeiPiazza}. Moreover for the class $[\overline{\eth}_{\mathrm{rel}}]$ the following interesting geometric interpretation is given in \cite[Th. 4.1]{BeiPiazza} and \cite[Prop. 5.1]{Lott-Todd}: given an arbitrarily fixed resolution of singularities $\pi:M\rightarrow X$ and a Hermitian metric $g$ on $M$ it holds \begin{equation}
\label{kk}
\pi_*[\overline{\eth}_M]=[\overline{\eth}_{\mathrm{rel}}]
\end{equation}
 in $KK_0(C(X),\mathbb{C})$ with $\overline{\eth}_M$ the rolled-up operator of the  $L^2$-$\overline{\partial}$-complex of $(0,\bullet)$-forms on $(M,g)$. In this paper, we have a twofold aim: we prove an equality similar to \eqref{kk}, but for a different operator; and for the proof, we develop a completely different approach compared with the one used in \cite{BeiPiazza} and \cite{Lott-Todd}.
More precisely let $(X,\gamma)$ be a compact and irreducible Hermitian complex space with $\dim(\sing(X))=0$ endowed with a Hermitian holomorphic vector bundle $(F,\tau)\rightarrow X$. Let $$\overline{\eth}_{E,m,\mathrm{abs}}:L^2\Omega^{m,\bullet}(\reg(X),F,\gamma,\tau)\rightarrow L^2\Omega^{m,\bullet}(\reg(X),F,\gamma,\tau)$$ be the rolled-up operator of the maximal $L^2$-$\overline{\partial}$-complex of  $F$-valued $(m,\bullet)$-forms. Let $\pi:M\rightarrow X$ be an arbitrarily fixed resolution of $X$, let $g$ be a Hermitian metric on $M$ and let $E=\pi^*F$ and $\rho:=\pi^*\tau$.
Our main result, Th.\ref{tata}, shows that under quite general assumptions on $\tau$ we have the following equality in $KK_0(C(X),\mathbb{C})$: 
\begin{equation}
\label{gattogatto}
\pi_*[\overline{\eth}_{E,m}]=[\overline{\eth}_{F,m,\mathrm{abs}}]
\end{equation}
with $\overline{\eth}_{E,m}:L^2\Omega^{m,\bullet}(M,E,g,\rho)\rightarrow L^2\Omega^{m,\bullet}(M,E,g,\rho)$  the rolled-up operator of the $L^2$-$\overline{\partial}$-complex of  $E$-valued $(m,\bullet)$-forms on $M$ and $[\overline{\eth}_{E,m}]$ the corresponding class in $KK_0(C(M),\mathbb{C})$.
Certainly, the reader familiar with the topic will notice immediately that \eqref{gattogatto} can be proved quickly by adopting the same strategy used in \cite{BeiPiazza} and \cite{Lott-Todd} and that goes back to \cite{Haskell}, which is based on the short exact sequence  $0\rightarrow KK_0(C(\sing(X)),\mathbb{C})\rightarrow KK_0(C(X),\mathbb{C})\rightarrow KK_0(C_0(\reg(X)),\mathbb{C})\rightarrow 0$ and,  crucially, employs the fact that $\dim(\sing(X))=0$.  Nevertheless there are at least two reasons that we believe make our paper interesting. First, our result holds in a more general framework than \cite{BeiPiazza} and \cite{Lott-Todd} since we allow the twist with a Hermitian holomorphic vector bundle. Second, our proof is entirely different. Indeed we prove the equality \eqref{gattogatto} in a direct way by constructing an explicit  homotopy between the unbounded even Fredholm modules induced by $\overline{\eth}_{E,m}$ and $\overline{\eth}_{F,m,\mathrm{rel}}$. Our construction relies on functional analytic techniques developed in \cite{KuSh} to tackle spectral convergence problems. Let us now give more details by describing how the paper is organised. The first section contains the background material, such as the basic properties of Hermitian metrics and the corresponding $L^2$-metrics, the functional analytic framework that will be used in the rest of the paper, and the main definitions and properties of analytic $K$-homology. The second section contains the main technical results of this paper. In the first part we consider a compact complex manifold $M$ of complex dimension $m$ endowed with a holomorphic Hermitian vector bundle $(E,\rho)\rightarrow M$.  We equip $M$ with $g_t$, $t\in [0,1]$, a family of Hermitian metrics on $M$ that degenerates to a positive semidefinite Hermitian pseudometric $h:=g_0$ as $t\rightarrow 0$ (we refer to the 2nd section for a precise formulation). We then consider the operators $$\overline{\partial}^{g_1,g_s}_{E,m,q,\max}:L^2\Omega^{m,q}(M,E,g_1,\rho)\rightarrow L^2\Omega^{m,q+1}(M,E,g_s,\rho)$$ and $$\overline{\partial}^{g_s,h}_{E,m,q,\max}:L^2\Omega^{m,q}(M,E,g_s,\rho)\rightarrow L^2\Omega^{m,q+1}(A,E|_A,h|_A,\rho|_A)$$ with $A$ the open and dense subset of $M$ where $h$ is positive definite. Under quite general assumptions on the family $g_t$  we show that both $\overline{\partial}^{g_1,g_s}_{E,m,q,\max}$ and $\overline{\partial}^{g_s,h}_{E,m,q,\max}$ have a well-defined and compact Green operator $$G_{\overline{\partial}^{g_1,g_s}_{E,m,q,\max}}:L^2\Omega^{m,q+1}(M,E,g_s,\rho)\rightarrow L^2\Omega^{m,q}(M,E,g_1,\rho)$$ and  $$G_{\overline{\partial}^{g_s,h}_{E,m,q,\max}}:L^2\Omega^{m,q+1}(A,E|_A,h|_A,\rho|_A)\rightarrow L^2\Omega^{m,q}(M,E,g_s,\rho)$$ and that when $s\rightarrow 0$  $$G_{\overline{\partial}^{g_1,g_s}_{E,m,q,\max}}\rightarrow G_{\overline{\partial}^{g_1,h}_{E,m,q,\max}}\quad \mathrm{and}\quad G_{\overline{\partial}^{g_s,h}_{E,m,q,\max}}\rightarrow G_{\overline{\partial}^{h,h}_{E,m,q,\max}}$$ both in the compact sense, see Def.  \ref{compact}. The above convergence results are then used to prove that the resolvent of $\overline{\eth}_{E,m}$ can be continuously deformed with respect to the operator norm to the resolvent of a self-adjoint operator unitarily equivalent to $\overline{\eth}_{F,m,\mathrm{abs}}$. We point out that the results of this section hold without any assumption on $\sing(X)$. The constraint $\dim(\sing(X))=0$ arises only in the third and last section as a sufficient condition to have well-defined even unbounded Fredholm modules, see Rmk \ref{rem}. Finally in the last section we use the above convergence and deformation results to give a direct proof of \eqref{gattogatto}. We conclude this introduction with a last comment that enlightens another possibly interesting feature of our approach. The strategy adopted in this paper could potentially lead to a better version of \eqref{gattogatto}. More precisely if one can prove that $G_{\overline{\partial}_{E,m,q}^{g_s,g_s}}$ is compactly convergent to $G_{\overline{\partial}_{E,m,q,\max}^{h,h}}$ as $s\rightarrow 0$ then  \eqref{gattogatto} would hold true without assumptions on $\sing(X)$, see Rmk \ref{rem}. Unfortunately, we do not have such a strong convergence result at our disposal.\\

\noindent\textbf{Acknowledgements}. The author is partially supported by INDAM-GNSAGA 
''Gruppo Nazionale per le Strutture Algebriche, Geometriche e le loro Applicazioni'' of the Istituto Nazionale di Alta Matematica ''Francesco Severi''.

\section{Background material}
\subsection{Hermitian metrics and $L^2$-metrics}
Let $(M,J)$ be a complex manifold of complex dimension $m$ and let $g$ and $h$ be Hermitian metrics on $M$. Let $F\in C^{\infty}(M,\mathrm{End}(TM))$ be the endomorphism of the tangent bundle such that $h(\cdot,\cdot)=g(F\cdot,\cdot)$ and let $F_{\mathbb{C}}\in C^{\infty}(M,\mathrm{End}(TM\otimes \mathbb{C}))$ be the $\mathbb{C}$-linear endomorphism induced by $F$ on the complexified tangent bundle. Since $F$ commutes with $J$ it follows that both $T^{1,0}M$ and $T^{0,1}M$ are preserved by $F_{\mathbb{C}}$. We denote the corresponding restrictions with $F^{1,0}_{\mathbb{C}}:=F_{\mathbb{C}}|_{T^{1,0}M}$ and $F^{0,1}_{\mathbb{C}}:=F_{\mathbb{C}}|_{T^{0,1}M}$. 
Let now $g^*$ and $h^*$ be the Hermitian  metrics induced by $g$ and $h$ on $T^*M$, respectively. We have  $h^*(\cdot,\cdot)=g^*((F^{-1})^t\cdot,\cdot)$ with $(F^{-1})^t$ the transpose of $F^{-1}$, that is the endomorphism of $T^*M$ induced by $F^{-1}$. Let $G\in C^{\infty}(M,\mathrm{End}(T^*M))$ be defined as $G:=(F^{-1})^t$ and let $G_{\mathbb{C}}$, $G_{\mathbb{C}}^{1,0}$ and $G_{\mathbb{C}}^{0,1}$ be the $\mathbb{C}$-linear endomorphisms induced by $G$ and acting on $T^*M\otimes \mathbb{C}$,  $T^{1,0,*}M$ and $T^{0,1,*}M$, respectively. Let us now denote by $g_{\mathbb{C}}$ and $h_{\mathbb{C}}$ the Hermitian metrics on $TM\otimes \mathbb{C}$ induced by $g$ and $h$, respectively. Let $h^*_{\mathbb{C}}$, $h^*_{a,b}$, $g^*_{\mathbb{C}}$ and $g_{a,b}^*$ be the Hermitian metrics on $T^*M\otimes \mathbb{C}$ and $\Lambda^{a,b}(M)$ induced by $h_{\mathbb{C}}$ and $g_{\mathbb{C}}$, respectively. Clearly $h^*_{a,b}=h^*_{a,0}\otimes h^*_{0,b}$, $g^*_{a,b}=g^*_{a,0}\otimes g^*_{0,b}$ and $h^*_{a,0}(\cdot,\cdot)=g_{a,0}^*(G_{\mathbb{C}}^{a,0}\cdot,\cdot)$, $h^*_{0,b}(\cdot,\cdot)=g_{0,b}^*(G_{\mathbb{C}}^{0,b}\cdot,\cdot)$, $h^*_{a,b}(\cdot, \cdot)=g^*_{a,b}(G_{\mathbb{C}}^{a,0}\otimes G_{\mathbb{C}}^{0,b}\cdot,\cdot)$ where $G_{\mathbb{C}}^{0,b}\in C^{\infty}(M,\mathrm{End}(\Lambda^{0,b}(M)))$ and $G_{\mathbb{C}}^{a,0}\in C^{\infty}(M,\mathrm{End}(\Lambda^{a,0}(M)))$ are the endomorphisms induced in the natural way by $G^{0,1}_{\mathbb{C}}$ and $G^{1,0}_{\mathbb{C}}$, respectively. Now we consider a holomorphic vector bundle $E\rightarrow M$ endowed with a Hermitian metric $\rho$. Let us denote by $g^{*}_{a,b,\rho}$ and $h^{*}_{a,b,\rho}$  the Hermitian metrics induced by $g^*_{a,b}$, $h_{a,b}^*$ and $\rho$ on $\Lambda^{a,b}(M)\otimes E$, respectively. Let $S^{a,b}\in C^{\infty}(M,\mathrm{End}(\Lambda^{a,b}(M)\otimes E))$ given by $G^{a,0}_{\mathbb{C}}\otimes G^{0,b}_{\mathbb{C}}\otimes \id$. Then, we have $$h_{a,b,\rho}^{*}(\cdot,\cdot)=g^{*}_{a,b,\rho}(S^{a,b}\cdot,\cdot).$$ Let $\Omega_c^{p,q}(M,E)$ be the space of smooth $E$-valued $(p,q)$-forms with compact support, let $L^2\Omega^{p,q}(M,E,g,\rho)$ be the Hilbert space of $E$-valued $L^2$-$(p,q)$-forms over $M$ with respect to $g$ and $\rho$ and with self-explanatory notation let us also consider $L^2\Omega^{p,q}(M,E,h,\rho)$. Given any $\phi,\psi\in \Omega_c^{a,b}(M,E)$, we can describe the $L^2$-product induced by $h$ in terms of the $L^2$-product induced by $g$ as follows:  
\begin{equation}
\label{ghgh}
\langle\psi,\phi\rangle_{L^2\Omega^{a,b}(M,E,h,\rho)}=\int_Mh^*_{a,b,\rho}(\psi,\phi)\dvol_{h}=\int_Mg^*_{a,b,\rho}(S^{a,b}\psi,\phi)\sqrt{\det(F)}\dvol_{g}.
\end{equation}
Let us now point out some consequences of \eqref{ghgh}. If $\psi\in \Omega^{0,b}_c(M,E)$ we have  
$$\begin{aligned}
\langle\psi,\psi\rangle_{L^2\Omega^{0,b}(M,E,h,\rho)}&=\int_Mh^*_{0,b,\rho}(\psi,\psi)\dvol_h\\
&=\int_Mg^*_{0,b,\rho}(S^{0,b}\psi,\psi)\sqrt{\det(F)}\dvol_g\\
&=\int_Mg^*_{0,b,\rho}((G^{0,b}_{\mathbb{C}}\otimes\id)\psi,\psi)\sqrt{\det(F)}\dvol_g\\
&\leq \int_M|G^{0,b}_{\mathbb{C}}|_{g^*_{0,b}}g^*_{0,b,\rho}(\psi,\psi)\sqrt{\det(F)}\dvol_g
\end{aligned}$$
where $|G^{0,b}_{\mathbb{C}}|_{g^*_{0,b}}:M\rightarrow \mathbb{R}$ is the function that assigns to each $p\in M$ the pointwise operator norm of $G^{0,b}_{\mathbb{C},p}:\Lambda^{0,b}_p(M)\rightarrow \Lambda^{0,b}_p(M)$ with respect to $g_{0,b}^*$, that is 
\begin{equation}
\label{pointnorm}
|G^{0,b}_{\mathbb{C}}|_{g^*_{0,b}}(p)=\sup_{0\neq v\in \Lambda^{0,b}_p(M)}\sqrt{\frac{g^*_{0,b}(G^{0,b}_{\mathbb{C}}v,G^{0,b}_{\mathbb{C}}v)}{g^*_{0,b}(v,v)}}.
\end{equation}
In particular, if $|G^{0,b}_{\mathbb{C}}|_{g^*_{0,b}}\sqrt{\det(F)}\in L^{\infty}(M)$, we obtain 
$$
\begin{aligned}
\langle\psi,\psi\rangle_{L^2\Omega^{0,b}(M,E,h,\rho)}&\leq \int_M|G^{0,b}_{\mathbb{C}}|_{g^*_{0,b}}g^*_{0,b,\rho}(\psi,\psi)\sqrt{\det(F)}\dvol_g\\
&\leq \||G^{0,b}_{\mathbb{C}}|_{g^*_{0,b}}\sqrt{\det(F)}\|_{L^{\infty}(M)}\int_Mg^*_{0,b,\rho}(\psi,\psi)\dvol_g\\
&=\||G^{0,b}_{\mathbb{C}}|_{g^*_{0,b}}\sqrt{\det(F)}\|_{L^{\infty}(M)}\langle\psi,\psi\rangle_{L^2\Omega^{0,b}(M,E,g,\rho)}.
\end{aligned}$$
When $\psi\in \Omega^{m,b}_c(M)$ we have 
\begin{align}
\label{ibra3} 
\langle\psi,\psi\rangle_{L^2\Omega^{m,b}(M,E,h,\rho)}&=\int_M h_{m,b,\rho}^*(\psi,\psi)\dvol_h\\
& \nonumber = \int_M g_{m,b,\rho}^*(S^{m,b}\psi,\psi)\sqrt{\det(F)}\dvol_g\\
& \nonumber =\int_M g_{m,b,\rho}^*((\det(G^{1,0}_{\mathbb{C}})\otimes G^{0,b}_{\mathbb{C}}\otimes \id)\psi,\psi)\sqrt{\det(F)}\dvol_g\\
& \nonumber = \int_M g_{m,b,\rho}^*(( \id\otimes G^{0,b}_{\mathbb{C}}\otimes\id)\psi,\psi)\dvol_g\\
& \nonumber \leq \int_M |G^{0,b}_{\mathbb{C}}|_{g^*_{0,b}} g^*_{m,b,\rho}(\psi,\psi)\dvol_g.
\end{align}

Thus, whenever $|G^{0,b}_{\mathbb{C}}|_{g^*_{0,b}}\in L^{\infty}(M)$, we have

$$\langle\psi,\psi\rangle_{L^2\Omega^{m,b}(M,E,h,\rho)}\leq \||G^{0,b}_{\mathbb{C}}|_{g^*_{0,b}}\|_{L^{\infty}(M)}\langle\psi,\psi\rangle_{L^2\Omega^{m,b}(M,E,g,\rho)}
$$
whereas if  there exists a positive constant $c$ such that $g^*_{0,b}(G^{0,b}_{\mathbb{C}}\cdot,\cdot)\geq cg^*_{0,b}(\cdot,\cdot)$  then  $$g_{m,b,\rho}^*( (\id\otimes G^{0,b}_{\mathbb{C}}\otimes \id)\psi,\psi)\geq cg^*_{m,b,\rho}( \psi,\psi)$$ and therefore
\begin{equation}
\label{bibi2}
\langle\psi,\psi\rangle_{L^2\Omega^{m,b}(M,E,h,\rho)}\geq c\langle\psi,\psi\rangle_{L^2\Omega^{m,b}(M,E,g,\rho)}.
\end{equation}
Let us denote by $\overline{\partial}_{E,p,q}:\Omega^{p,q}(M,E)\rightarrow \Omega^{p,q+1}(M,E)$  the Dolbeault operator acting on $E$-valued $(p,q)$-forms. When we look at $\overline{\partial}_{E,p,q}:L^2\Omega^{p,q}(M,E,g,\rho)\rightarrow L^2\Omega^{p,q+1}(M,E,h,\rho)$ as an unbounded and densely defined operator with domain $\Omega_c^{p,q}(M,E)$ we denote by $$\overline{\partial}_{p,q,\max/\min}^{g,h}:L^2\Omega^{p,q}(M,E,g,\rho)\rightarrow L^2\Omega^{p,q+1}(M,E,h,\rho)$$ respectively its maximal and minimal extension. The former is the closed extension defined in the distributional sense: $\omega\in \mathcal{D}(\overline{\partial}_{E,p,q,\max}^{g,h})$ if $\omega\in L^2\Omega^{p,q}(M,E,g,\rho)$ and $\overline{\partial}_{E,p,q}\omega$, applied in the distributional sense, lies in $L^2\Omega^{p,q+1}(M,E,h,\rho)$.  The latter is defined as the graph closure of $\Omega^{p,q}_c(M,E)$ in $L^2\Omega^{p,q}(M,E,g,\rho)$ with respect to the graph norm of $\overline{\partial}_{p,q}:\Omega^{p,q}_c(M,E)\rightarrow L^2\Omega^{p,q+1}(M,E,h,\rho)$. With $\overline{\partial}_{E,p,q}^{g,h,t}:\Omega^{p,q+1}_c(M)\rightarrow \Omega^{p,q}_c(M)$ we denote the formal adjoint of $\overline{\partial}_{E,p,q}$ with respect to the Hermitian metrics $g$ on $\Omega_c^{p,q}(M,E)$ and $h$ on $\Omega_c^{p,q+1}(M,E)$ and with $\overline{\partial}_{E,p,q,\max/\min}^{g,h,t}:L^2\Omega^{p,q+1}(M,E,h,\rho)\rightarrow L^2\Omega^{p,q}(M,E,g,\rho)$ we denote the corresponding maximal and minimal extensions. Note that 
\begin{equation}
\label{mixedad}
\overline{\partial}_{E,p,q,\max}^{g,h,t}=(\overline{\partial}_{E,p,q,\min}^{g,h})^*\quad \mathrm{and}\quad \overline{\partial}_{E,p,q,\min}^{g,h,t}=(\overline{\partial}_{E,p,q,\max}^{g,h,t})^*.
\end{equation}
We conclude this section with the following:
\begin{prop}
\label{Pp}
Let $(M,J)$ be a complex manifold of complex dimension $m$ and let $(E,\rho)\rightarrow M$ be a Hermitian holomorphic vector bundle over $M$. Let $g_1$, $g_2$, $h_1$ and $h_2$ be four Hermitian metrics on $M$ such that $h_1\leq c_1g_1$ and $h_2\leq c_2g_2$, with $c_1$ and $c_2$ positive constants. Finally let us consider the operators
\begin{align}
& \nonumber \overline{\partial}_{E,m,q,\max}^{g_1,g_2}:L^2\Omega^{m,q}(M,E,g_1,\rho)\rightarrow L^2\Omega^{m,q+1}(M,E,g_2,\rho),\\
& \nonumber \overline{\partial}_{E,m,q,\max}^{h_1,h_2}:L^2\Omega^{m,q}(M,E,h_1,\rho)\rightarrow L^2\Omega^{m,q+1}(M,E,h_2,\rho).
\end{align}
If $\omega\in \mathcal{D}(\overline{\partial}_{E,m,q,\max}^{h_1,h_2})$ then $\omega\in \mathcal{D}(\overline{\partial}_{E,m,q,\max}^{g_1,g_2})$ and $\overline{\partial}_{E,m,q,\max}^{h_1,h_2}\omega=$ $\overline{\partial}_{E,m,q,\max}^{g_1,g_2}\omega.$ 
Moreover the induced inclusion $$\mathcal{D}(\overline{\partial}_{E,m,q,\max}^{h_1,h_2})\hookrightarrow \mathcal{D}(\overline{\partial}_{E,m,q,\max}^{g_1,g_2})$$ is continuous with respect to the corresponding graph norms. 
\end{prop}
\begin{proof}
Since we assumed $h_1\leq c_1g_1$ and $h_2\leq c_2g_2$ it follows from \eqref{ibra3} that \eqref{bibi2} holds true for both $g_1$, $h_1$ and $g_2$, $h_2$. Thus the identity $\mathrm{Id}:\Omega_c^{m,q}(M,E)\rightarrow \Omega_c^{m,q}(M,E)$ induces continuous inclusions 
\begin{equation}
\label{cixz}
L^2\Omega^{m,q}(M,E,h_1,\rho)\hookrightarrow L^2\Omega^{m,q}(M,E,g_1,\rho)\quad \mathrm{and}\quad L^2\Omega^{m,q}(M,E,h_2,\rho)\hookrightarrow L^2\Omega^{m,q}(M,E,g_2,\rho)
\end{equation} for each $q=0,...,m$. Let now $\omega\in \mathcal{D}(\overline{\partial}_{E,m,q,\max}^{h_1,h_2})$ and let $\eta=\overline{\partial}_{E,m,q,\max}^{h_1,h_2}\omega$. This is equivalent to saying that $$(-1)^{m+q+1}\int_M\omega\wedge \overline{\partial}_{E^*,0,m-q-1}\phi=\int \eta\wedge \phi$$ for any $\phi\in \Omega_c^{0,m-q-1}(M,E^*)$. Thanks to \eqref{cixz} the above equality implies that $\omega\in \mathcal{D}(\overline{\partial}_{E,m,q,\max}^{g_1,g_2})$ and that $\eta=\overline{\partial}_{E,m,q,\max}^{g_1,g_2}\omega$. Finally, again by \eqref{cixz}, we can conclude that the induced inclusion $\mathcal{D}(\overline{\partial}_{E,m,q,\max}^{h_1,h_2})\hookrightarrow \mathcal{D}(\overline{\partial}_{E,m,q,\max}^{g_1,g_2})$ is continuous with respect to the corresponding graph norms. 
\end{proof}

\subsection{Functional analytic prerequisites}

We briefly recall  some functional analytic tools that will be used later. All the material is taken from \cite{KuSh}. We refer to it for an in-depth treatment. Let $\{H_n\}_{n\in \mathbb{N}}$ be a sequence of infinite dimensional separable complex Hilbert spaces. Let $H$ be another infinite dimensional separable complex Hilbert space.  Let us denote by $\langle \cdot ,\ \cdot \rangle_{H_n}$, $\|\cdot\|_{H_n}$, $\langle \cdot ,\ \cdot \rangle_{H}$ and $\|\cdot\|_{H}$  the corresponding inner products and norms. Let $\mathcal{C}\subseteq H$ be a dense subset. Assume that for every $n\in \mathbb{N}$ there exists a linear map $\Phi_n:\mathcal{C}\rightarrow H_n$. We will say that $H_n$ converges to $H$ as $n\rightarrow \infty$ if and only if
\begin{equation}
\label{limitHilbert}
\lim_{n\rightarrow\infty}\|\Phi_nu\|_{H_n}=\|u\|_{H}
\end{equation}
for any $u\in \mathcal{C}$.\\

\noindent \textbf{Assumption:} In the following definitions and propositions we will always assume that the sequence $\{H_n\}_{n\in \mathbb{N}}$ {\em converges} to $H$.

\begin{defi}
\label{strong}
Let $u\in H$ and let $\{u_n\}_{n\in \mathbb{N}}$ be a sequence such that $u_n\in H_n$ for each $n\in \mathbb{N}$. We say that $u_n$ strongly converges to $u$ as $n\rightarrow \infty$ if there exists a net $\{v_{\beta}\}_{\beta\in \mathcal{B}}\subset \mathcal{C}$ tending to $u$ in $H$ such that
\begin{equation}
\label{stronglimit}
\lim_{\beta} \limsup_{n\rightarrow\infty}\|\Phi_n v_{\beta}-u_n\|_{H_n}=0
\end{equation}
\end{defi}
\begin{defi}
\label{weak}
Let $u\in H$ and let $\{u_n\}_{n\in \mathbb{N}}$ be a sequence such that $u_n\in H_n$ for each $n\in \mathbb{N}$. We say that $u_n$ weakly converges to $u$ as $n\rightarrow \infty$ if
\begin{equation}
\label{weaklimit}
\lim_{n\rightarrow \infty} \langle u_n,w_n\rangle_{H_n}=\langle u,w\rangle_{H}
\end{equation}
for any  $w\in H$ and any sequence $\{w_n\}_{n\in \mathbb{N}}$, $w_n\in H_n$, strongly convergent to $w$.
\end{defi}

\begin{prop}
\label{bounded}
Let $\{u_n\}_{n\in \mathbb{N}}$ be a sequence such that $u_n\in H_n$ for each $n\in \mathbb{N}$. Assume that there exists a positive real number $c$ such that $\|u_n\|_{H_n}\leq c$ for every $n\in \mathbb{N}$. There then exists a subsequence  $\{u_m\}_{m\in \mathbb{N}}\subset \{u_n\}_{n\in \mathbb{N}}$, $u_m\in H_m$, weakly convergent to some element $u\in H$.
\end{prop}

\begin{proof}
See \cite{KuSh} Lemma 2.2.
\end{proof}

\begin{prop}
\label{wibounded}
Let $\{u_n\}_{n\in \mathbb{N}}$,  $u_n\in H_n$, be a sequence  weakly convergent to some element $u\in H$. There then exists a positive real number, $\ell$, such that 
\begin{equation}
\label{lerume}
\sup_{n\in \mathbb{N}} \|u_n\|_{H_n}\leq \ell \quad\quad\quad\ and\ \quad\quad\quad \|u\|_{H}\leq \liminf_{n\rightarrow \infty} \|u_n\|_{H_n}.
\end{equation}
\end{prop}
\begin{proof}
See \cite{KuSh} Lemma 2.3.
\end{proof}

We now have the following remark. Consider the case of a constant sequence of infinite dimensional separable complex Hilbert spaces $\{H_n\}_{n\in \mathbb{N}}$, that is for each $n\in \mathbb{N}$ $H_n=H$, $\mathcal{C}=H$ and $\Phi_n:\mathcal{C}\rightarrow H_n$ is nothing but the identity $\id:H\rightarrow H$. Then Def. \ref{strong} and Def. \ref{weak} coincide with ordinary notions of convergence in $H$ and weak convergence in $H$. Indeed let $\{v_n\}\subset H$ be a sequence converging to some $v\in H$. Then by taking the constant net $\{v_{\beta}\}_{\beta\in B}\subset H$, $v_{\beta}:=v$ as a net in $H$ converging to $v$ we have $$\lim_{\beta}\limsup_{n\rightarrow \infty}\|\Phi_nv_{\beta}-v_n\|_{H_n}=\limsup_{n\rightarrow \infty}\|v-v_n\|_H=0.$$ Therefore $v_n\rightarrow v$ strongly in the sense of Def. \ref{strong}. Conversely let us assume that for some net $\{v_{\beta}\}_{\beta\in B}\subset H$ tending to $v$ in $H$ we have $$\lim_{\beta}\limsup_{n\rightarrow \infty}\|\Phi_nv_{\beta}-v_n\|_{H_n}=0.$$
Given any $\beta\in B$ we have $\|v-v_n\|_H\leq \|v-v_{\beta}\|_{H}+\|v_{\beta}-v_n\|_{H}$. Therefore for every $\beta\in B$ $$\limsup_{n\rightarrow \infty}\|v-v_n\|_H\leq \|v-v_{\beta}\|_{H}+\limsup_{n\rightarrow \infty}\|v_{\beta}-v_n\|_{H}$$ and finally 
$$\limsup_{n\rightarrow \infty}\|v-v_n\|_H\leq \lim_{\beta} \|v-v_{\beta}\|_{H}+\lim_{\beta}\limsup_{n\rightarrow \infty}\|v_{\beta}-v_n\|_{H}=0.$$ Therefore $v_n\rightarrow v$ in $H$ and thus we showed that Def. \ref{strong}  coincides with the ordinary notion of convergence in $H$. This, in turn, implies immediately that Def. \ref{weak} coincides with the standard definition of weak convergence in $H$. We now have  the following
\begin{defi}
\label{compact}
A sequence of bounded operators $B_n:H_n\rightarrow H_n$ compactly converges to a bounded operator $B:H\rightarrow H$ if $B_n(u_n)\rightarrow B(u)$ strongly as $n\rightarrow \infty$ for any sequence $\{u_n\}_{n\in \mathbb{N}}$, $u_n\in H_n$, weakly convergent to $u\in H$.
\end{defi}

Given a Hilbert space $H$ and a bounded operator $T:H\rightarrow H$ we denote by $\|T\|_{\mathrm{op}}$ the operator norm of $T$. We recall the following fact:

\begin{prop}
\label{usefulprop}
Let $H$ be a separable Hilbert space and let $B$ and $\{B_n\}_{n\in \mathbb{N}}$ be bounded operators acting on $H$. Assume that for each weakly convergent sequence $\{v_n\}_{n\in \mathbb{N}}\subset H$, $v_n\rightharpoonup v$ as $n\rightarrow \infty$ to some $v\in H$, we have $\|B_nv_n-Bv\|_H\rightarrow 0$ as $n\rightarrow \infty.$ Then $\|B_n-B\|_{\mathrm{op}}\rightarrow 0$ as $n\rightarrow \infty$  and $B$ is compact.
\end{prop}

\begin{proof}
See \cite[Lemma 2.8]{KuSh}.
\end{proof}

Finally, we conclude this section  by recalling some well-known facts about Green operators. Let $H_1$ and $H_2$ be  separable Hilbert spaces whose Hilbert products are denoted by $\langle\ ,\ \rangle_{H_1}$ and $\langle\ ,\ \rangle_{H_2}$. Let $T:H_1\rightarrow H_2$ be an unbounded, densely defined and closed operator  with domain $\mathcal{D}(T)$. Assume that $\im(T)$ is closed. Let  $T^*:H_2\rightarrow H_1$ be the adjoint of $T$. Then $\im(T^*)$ is closed as well and we have the following orthogonal decompositions: $H_1=\ker(T)\oplus \im(T^*)$ and $H_2=\ker(T^*)\oplus \im(T)$. The  Green operator of $T$, $$G_T:H_2\rightarrow H_1$$ is then the operator defined by the following assignments: if $u\in \ker(T^*)$ then $G_T(u)=0$,  if $u\in \im(T)$ then  $G_T(u)=v$ where $v$ is the unique element in $\mathcal{D}(T)\cap \im(T^*)$ such that $T(v)=u$. We have that $G_T:H_2\rightarrow H_1$ is  a bounded operator. Moreover, if $H_1=H_2$ and $T$ is self-adjoint then $G_T$ is also self-adjoint. If $H_1=H_2$ and $T$  is self-adjoint and non-negative, that is $\langle Tu,u\rangle_{H_1}\geq 0$ for each $u\in \mathcal{D}(T)$, then $G_T$ is self-adjoint and non-negative as well. 
Finally we recall the following property which is straightforward to check
\begin{prop}
\label{useful}
Let $T:H_1\rightarrow H_2$ be as above. The following two properties are then equivalent:
\begin{enumerate}
\item $G_T:H_2\rightarrow H_1$ is a compact operator;
\item the inclusion $\mathcal{D}(T)\cap \im(T^*)\hookrightarrow H_1$, where  $\mathcal{D}(T)\cap \im(T^*)$ is endowed with the graph norm of $T$, is a compact operator.
\end{enumerate}
\end{prop}

\subsection{Analytic K-homology classes}
We now recall  the definition of $KK_0(C(X),\mathbb{C})$. We invite the interested reader to consult \cite{Nigel} for a thorough exposition. Let $Z$ be a second countable compact space and let $C(Z)$ be the corresponding $C^*$-algebra of continuous complex-valued functions. An even Fredholm module over $C(Z)$ is a triplet $(H,\rho, F)$ satisfying the following properties:
\begin{enumerate}
\item $H$ is a separable Hilbert space,
\item $\rho$ is a  representation $\rho:C(Z)\rightarrow \mathcal{B}(H)$ of $C(Z)$ as bounded operators on $H$
\item $F$ is an operator on $H$ such that for all $f\in C(Z)$:
$$(F^2-\id)\circ \rho(f),\ (F-F^*)\circ \rho(f)\ \text{and}\  [F,\rho(f)]\ \text{lie in}\ \mathcal{K}(H)$$
where $\mathcal{K}(H)\subset \mathcal{B}(H)$ is the space of compact operators. 
\item The Hilbert space $H$ is equipped with a $\mathbb{Z}_2$-grading $H=H^+\oplus H^- $ in such a
way that for each $f\in C(Z)$, the operator $\rho(f)$ is even-graded, while the operator $F$ is
odd-graded.
\end{enumerate}
Let $(H_1, \rho_1, F_1)$ and $(H_2, \rho_2, F_2)$ be two even Fredholm modules over $C(Z)$. A {\em unitary equivalence} between them is a  grading-zero unitary isomorphism $u:H_1\rightarrow H_2$ which intertwines the representations $\rho_1$ and $\rho_2$ and the operators $F_1$ and  $F_2$.\\
Given two even Fredholm modules $(H,\rho,F_0)$ and $(H,\rho,F_1)$ over $C(Z)$, an {\em operator homotopy} between them is a family of  Fredholm modules $(H,\rho,F_t)$ parameterized by $t\in [0, 1]$ in such a way that the representation $\rho$, the Hilbert space $H$ and its grading structures remain constant but the operator $F_t$ varies with $t$ and the function $[0,1]\rightarrow \mathcal{B}(H)$, $t\mapsto F_t$ is operator norm continuous. In this case we will say that $(H,\rho,F_0)$ and $(H,\rho,F_1)$ are operator homotopic.\\
The notion  of direct sum for even Fredholm modules is defined naturally: one takes the direct sum of the Hilbert spaces,  representations and operators $F$. The zero even Fredholm module has zero Hilbert space, zero representation, and zero operator.\\
Now we can recall Kasparov's definition of $K$-homology. The $K$-homology group $KK_0(C(Z),\mathbb{C})$ is the abelian group with one generator $[x]$ for each
unitary equivalence class of  even Fredholm modules over $C(Z)$ and with the following
relations:
\begin{itemize}
\item if $x_0$ and $x_1$ are operator homotopic even Fredholm modules then $[x_0]=[x_1]$ in $KK_0(C(Z),\mathbb{C})$,
\item if $x_0$ and $x_1$ are any two even Fredholm modules then $[x_0+x_1]=[x_0]+[x_1]$ in $KK_0(C(Z),\mathbb{C})$.
\end{itemize}
Now we go on by recalling the notion of  {\em even unbounded Fredholm module} over the $C^*$-algebra $C(Z)$. This is a triplet
$(H,\upsilon, D)$ such that:
\begin{enumerate}
\item $H$ is a separable Hilbert space endowed with a unitary $*$-representation $\upsilon:C(Z)\rightarrow \mathcal{B}(H)$; $D$ is an unbounded, densely defined and self-
adjoint linear operator on $H$;
\item there is a dense $*$-subalgebra $\mathcal{A}\subset C(Z)$ such that  for all $a\in \mathcal{A}$ the domain of $D$ is invariant
by $\upsilon(a)$ and $[D,\upsilon(a)]$ extends to a bounded operator on $H$;
\item $\upsilon(a)(1+D^2)^{-1}$ is a compact operator on $H$ for any $a\in \mathcal{A}$;
\item $H$ is equipped with a grading $\tau=\tau^*$, $\tau^2=\mathrm{Id}$, such that $\tau\circ \upsilon=\upsilon\circ \tau$ and $\tau\circ D = -\tau\circ D$. In other words $\tau$ commutes with $\upsilon$ and anti-commutes with $D$.
\end{enumerate} 
We now recall  the following important result, see \cite[Prop. 2.2]{Baaj}:
\begin{prop}
\label{unbounded}
Let $(H,\upsilon,D)$ be an even unbounded Fredholm module over $C(Z)$. Then $$(H,\upsilon,D\circ (\id+D^2)^{-1/2})$$ is an even Fredholm module over $C(Z)$.
\end{prop}
\noindent In what follows, given an even unbounded Fredholm module as above, by the notation $[D]$ we will mean the class induced by the even Fredholm module $(H, \upsilon, D\circ (\id+D^2)^{-1/2})$ in $KK_0(C(Z),\mathbb{C})$.

\begin{prop}
\label{sameclass}
Let $(H,\upsilon,D_t)$ with $t\in [0,1]$ be a family of even unbounded Fredholm modules over $C(Z)$ with respect to a fixed dense $*$-subalgebra $\mathcal{A}\subset C(Z)$.
Assume that: 
\begin{enumerate}
\item For each $a\in \mathcal{A}$ the map $[0,1]\rightarrow B(H)$, $t\mapsto [D_t,\upsilon(a)]$ is continuous with respect to the strong operator topology;
\item The map $[0,1]\rightarrow B(H)$, $t\mapsto (i+D_t)^{-1}$ is continuous with respect to the operator norm.
\end{enumerate}
Then  the following equality: $$[D_0]=[D_1]$$ holds  in $KK_0(C(Z),\mathbb{C})$.
\end{prop}

\begin{proof}
This follows by \cite[Rmk 2.5 (iv)]{Baaj} and \cite[\S 6]{GeK}. See also \cite{Hilsum}. 
\end{proof}

We conclude this section with the following: 

\begin{lemma}
\label{lemmata}
Let $H$ be a separable Hilbert space and let $D_{t}$, $t\in [0,1]$, be a family of unbounded, densely defined and self-adjoint operators with closed range such that
\begin{enumerate}
\item $\|G_{D_t}- G_{D_0}\|_{\mathrm{op}}\rightarrow 0$ as $t\rightarrow 0$;
\item $\|\pi_{K,t}- \pi_{K,0}\|_{\mathrm{op}}\rightarrow 0$ as $t\rightarrow 0$, with $\pi_{K,t}:H \rightarrow \ker(D_t)$ denoting the orthogonal projection on $\ker(D_t)$ for each $t\in [0,1]$.
\end{enumerate}
Then $$\lim_{t\rightarrow 0}\|(D_t+i)^{-1}-(D_{0}+i)^{-1}\|_{\mathrm{op}}=0.$$
\end{lemma}

\begin{proof}
First, we need to recall the following formulas: let $D:H\rightarrow H$ be an arbitrarily fixed unbounded, densely defined and self-adjoint operator with closed range. Let us denote the resolvent of $D$, $(D+i)^{-1}:H\rightarrow H$, with $R_{D}$. Then 
\begin{equation}
\label{utile}
R_{D}|_{\im(D)}=G_{D}\circ (G_{D}+i)^{-1}|_{\im(D)}
\end{equation}
Let us show the above claim. First we point out that $\im(D)$ is preserved by the action of $G_D$ and $R_D$. Let now $\beta\in \im(D)$ be arbitrarily fixed and let $\alpha:=G_D\beta$. Then $D\alpha+i\alpha=\beta+i\alpha$ and consequently $$G_D(\beta)=\alpha=R_D(D\alpha+i\alpha)=R_D(\beta+i\alpha)=R_D(\beta+iG_D\beta)=R_D((\id+iG_D)\beta)$$ that is $G_D|_{\im(D)}=R_D\circ (\id+iG_D)|_{\im(D)}$. Since $G_D:H\rightarrow H$ is self-adjoint we know that $\id+iG_D:H\rightarrow H$ is invertible. Note that $(\id+iG_{D})(\im(D))=\im(D)$. Therefore $(\id+iG_D)|_{\im(D)}:\im(D)\rightarrow \im(D)$ is invertible and thus $((\id+iG_D)|_{\im(D)})^{-1}=(\id+iG_D)^{-1}|_{\im(D)}$. In this way we can conclude that $$R_{D}|_{\im(D)}=G_{D}\circ (\id+iG_{D})^{-1}|_{\im(D)}.$$
Let $\pi_{\im,t}:H\rightarrow \im(D_t)$ be the orthogonal projection on $\im(D_t)$ for each $t\in [0,1]$. Since $R_{D_t}|_{\ker(D_t)}=-i\id$ we have 
$$
\begin{aligned}
&\|(D_t+i)^{-1}-(D_{0}+i)^{-1}\|_{\mathrm{op}}=\|(D_t+i)^{-1}\circ(\pi_{K,t}+\pi_{\im,t})-(D_{0}+i)^{-1}\circ(\pi_{K,0}+\pi_{\im,0})\|_{\mathrm{op}}\\
&\leq \|(D_t+i)^{-1}\circ \pi_{\im,t}-(D_{0}+i)^{-1}\circ \pi_{\im,0}\|_{\mathrm{op}}+\|(D_t+i)^{-1}\circ \pi_{K,t}-(D_{0}+i)^{-1}\circ \pi_{K,0}\|_{\mathrm{op}}\\
&\leq \|G_{D_t}\circ(\id+iG_{D_t})^{-1}\circ \pi_{\im,t}-G_{D_0}\circ(\id+iG_{D_0})^{-1}\circ \pi_{\im,0}\|_{\mathrm{op}}+\| \pi_{K,t}-\pi_{K,0}\|_{\mathrm{op}}.
\end{aligned}
$$
By assumption we know that $\|\pi_{K,t}-\pi_{K,0}\|_{\mathrm{op}}\rightarrow 0$ as $t\rightarrow 0$. This clearly implies that $\|\pi_{\im,t}-\pi_{\im,0}\|_{\mathrm{op}}\rightarrow 0$ as $t\rightarrow 0$ since $\pi_{\im,t}=\id-\pi_{K,t}$. We also know that $\|G_{D_t}-G_{D_0}\|_{\mathrm{op}}\rightarrow 0$ as $t\rightarrow 0$ and this tells us that $\|(G_{D_t}-i)^{-1}-(G_{D_0}-i)^{-1}\|_{\mathrm{op}}\rightarrow 0$ as $t\rightarrow 0$, see \cite[Th. VIII.18]{RSI}. Obviously, this in turn implies that $$\|(\id+iG_{D_t})^{-1}-(\id+iG_{D_0})^{-1}\|_{\mathrm{op}}\rightarrow 0$$ as $t\rightarrow 0$. We can thus conclude that both $$\lim_{t\rightarrow 0}\|G_{D_t}\circ(\id+iG_{D_t})^{-1}\circ \pi_{\im,t}-G_{D_0}\circ(\id+iG_{D_0})^{-1}\circ \pi_{\im,0}\|_{\mathrm{op}}=0\quad\quad \mathrm{and}\quad\quad \lim_{t\rightarrow 0}\| \pi_{K,t}-\pi_{K,0}\|_{\mathrm{op}}=0$$ and therefore $$\lim_{t\rightarrow 0}\|(D_t+i)^{-1}-(D_{0}+i)^{-1}\|_{\mathrm{op}}=0$$ as desired.
\end{proof}

\section{Deformation of the $L^2$-$\overline{\partial}$-complex}
This section is divided into two subsections and contains the main technical results of this paper.

\subsection{Compact convergence of the Green operators}
Let $(M,J)$ be a compact complex manifold of complex dimension $m$ endowed with a Hermitian pseudometric $h$. We recall that a  Hermitian pseudometric on $M$ is a positive semidefinite Hermitian product on $M$ strictly positive over an open and dense subset $A\subset M$. Let $(E,\rho)\rightarrow M$ be a Hermitian holomorphic vector bundle over $M$. We make the following assumptions:
\begin{itemize}
\item $(A,g|_A)$ is {\em parabolic} with respect to some  Riemannian metric $g$ on $M$. 
\end{itemize}
Note that since $M$ is compact and parabolicity is a stable property through quasi-isometries, we can conclude that if $(A,g|_A)$ is parabolic with respect to some Riemannian metric $g$ on $M$ then $(A,g|_A)$ is parabolic with respect to any Riemannian metric $g$ on $M$. Moreover since $(A,g|A)$ is parabolic  we know that $M\setminus A$ has zero Lebesgue measure, see \cite[Th. 3.4, Prop 3.1]{Troya}.

\begin{itemize}
\item The $L^2$-$\overline{\partial}$ cohomology group $$H^{m,q+1}_{2,\overline{\partial}_{\max}}(A,E|_A,h|_A,\rho|_A):=\ker(\overline{\partial}_{E,m,q+1,\max}^{h,h})/\mathrm{im}(\overline{\partial}_{E,m,q,\max}^{h,h})$$ is finite dimensional. 
\end{itemize}

Note that since $H^{m,q+1}_{2,\overline{\partial}_{\max}}(A,E|_A,h|_A,\rho|_A)$ is finite dimensional the image of the operator $$\overline{\partial}_{E,m,q,\max}^{h,h}: L^2\Omega^{m,q}(A,E|_A,h|_A,\rho|_A)\rightarrow L^2\Omega^{m,q+1}(A,E|_A,h|_A,\rho|_A)$$ is closed.   Let $g$ be an arbitrarily fixed Hermitian metric on $M$. As a first step we recall the following result, see \cite[Prop. 3.2]{FBei}.

\begin{prop}
\label{sameop}
In the above setting the following three operators coincide:
\begin{align}
& \nonumber \overline{\partial}_{E,p,q,\max/\min}^{g,g}:L^2\Omega^{p,q}(A,E|_A,h|_A,\rho|_A)\rightarrow L^2\Omega^{p,q+1}(A,E|_A,h|_A,\rho|_A),\\
\label{ucse}
& \overline{\partial}_{E,p,q}^{g,g}:L^2\Omega^{p,q}(M,E,g,\rho)\rightarrow L^2\Omega^{p,q+1}(M,E,g,\rho),
\end{align}
where \eqref{ucse} is the unique closed extension of  $\overline{\partial}_{E,p,q}:\Omega^{p,q}(M,E)\rightarrow \Omega^{p,q+1}(M,E)$ viewed as an unbounded and densely defined operator acting between $L^2\Omega^{p,q}(M,E,g,\rho)$ and  $L^2\Omega^{p,q+1}(M,E,g,\rho).$
\end{prop}

We have now the following:

\begin{prop}
\label{Gianni}
In the above setting the operator 
\begin{equation}
\label{opmix}
\overline{\partial}_{E,m,q,\max}^{g,h}:L^2\Omega^{m,q}(A,E|_A,h|_A,\rho|_A)\rightarrow L^2\Omega^{m,q+1}(A,E|_A,h|_A,\rho|_A)
\end{equation}
 has  closed range and the corresponding Green operator 
\begin{equation}
\label{Gopmix}
G_{\overline{\partial}_{E,m,q,\max}^{g,h}}:L^2\Omega^{m,q+1}(A,E|_A,h|_A,\rho|_A)\rightarrow L^2\Omega^{m,q}(A,E|_A,g|_A,\rho|_A)
\end{equation}
 is compact.
\end{prop}

\begin{proof}
First we note that $\im(\overline{\partial}_{E,m,q,\max}^{g,h})$ is a closed subspace of $L^2\Omega^{m,q+1}(A,E|_A,h|_A,\rho|_A)$. Indeed Prop. \ref{Pp} tells us that $\im(\overline{\partial}_{E,m,q,\max}^{h,h})\subset \im(\overline{\partial}_{E,m,q,\max}^{g,h})$ and by the fact that $H^{m,q+1}_{2,\overline{\partial}_{\max}}(M,h)$ is finite dimensional we deduce that the quotient $\ker(\overline{\partial}_{E,m,q,\max}^{h,h})/\im(\overline{\partial}_{E,m,q,\max}^{g,h})$ is finite dimensional too which in turn implies that $\im(\overline{\partial}_{E,m,q,\max}^{g,h})$ is closed. Hence $G_{\overline{\partial}_{E,m,q,\max}^{g,h}}:L^2\Omega^{m,q+1}(A,E|_A,h|_A,\rho)\rightarrow L^2\Omega^{m,q}(A,E|_A,g|_A,\rho|_A)$ exists. Let us consider $\overline{\partial}_{E,m,q,\min}^{g,h,t}:L^2\Omega^{m,q+1}(A,E|_A,h|_A,\rho|_A)\rightarrow L^2\Omega^{m,q}(A,E|_A,g|_A,\rho|_A)$. Since $\im(\overline{\partial}_{E,m,q,\max}^{g,h})\subset L^2\Omega^{m,q+1}(A,E|_A,h|_A,\rho|_A)$ is closed by  \eqref{mixedad} we know that $\im(\overline{\partial}_{E,m,q,\min}^{g,h,t})\subset L^2\Omega^{m,q}(A,E|_A,g|_A,\rho|_A)$ is closed too. Let us then define $$B:=\mathcal{D}(\overline{\partial}_{E,m,q,\max}^{g,h})\cap \im(\overline{\partial}_{E,m,q,\min}^{g,h,t}).$$ If we endow $\mathcal{D}(\overline{\partial}_{E,m,q,\max}^{g,h})$ with  the corresponding graph product, then  $B$ becomes a closed subspace of $\mathcal{D}(\overline{\partial}_{E,m,q,\max}^{g,h})$ and we have the following orthogonal decomposition $$\mathcal{D}(\overline{\partial}_{E,m,q,\max}^{g,h})=\ker(\overline{\partial}_{E,m,q,\max}^{g,h})\oplus B.$$ According to Prop. \ref{useful} the compactness of \eqref{Gopmix} amounts to showing that  $B\hookrightarrow L^2\Omega^{m,q}(A,E|_A,g|_A,\rho|_A)$ is a compact inclusion, with $B$ endowed with the corresponding graph norm as above. To this aim let us now consider the operator defined in \eqref{ucse}.
Classical elliptic theory on compact manifolds tells us that $\im(\overline{\partial}_{E,m,q}^{g,g})\subset L^2\Omega^{m,q+1}(M,E,g,\rho)$ is closed and  the corresponding Green operator
 \begin{equation}
\label{Gg}
G_{\overline{\partial}_{E,m,q}^{g,g}}:L^2\Omega^{m,q+1}(M,E,g,\rho)\rightarrow L^2\Omega^{m,q}(M,E,g,\rho)
\end{equation}  is compact. As recalled above the compactness of \eqref{Gg} is equivalent to saying that the natural inclusion 
\begin{equation}
\label{ci}
\mathcal{D}(\overline{\partial}_{E,m,q}^{g,g})\cap \im((\overline{\partial}_{E,m,q}^{g,g})^*) \hookrightarrow L^2\Omega^{m,q}(M,E,g,\rho)
\end{equation}
 is a compact operator with $\mathcal{D}(\overline{\partial}_{E,m,q}^{g,g})\cap \im((\overline{\partial}_{E,m,q}^{g,g})^*)$ endowed with the corresponding graph product. Let $A:=\mathcal{D}(\overline{\partial}_{E,m,q}^{g,g})\cap \im((\overline{\partial}_{E,m,q}^{g,g})^*)$ and consider the corresponding orthogonal decomposition of $\mathcal{D}(\overline{\partial}_{E,m,q}^{g,g})$ with respect to the graph product $$\mathcal{D}(\overline{\partial}_{E,m,q}^{g,g})=\ker(\overline{\partial}_{E,m,q}^{g,g})\oplus A.$$ 
Note that, thanks to Prop. \ref{Pp} and Prop. \ref{sameop}, we know that $\mathcal{D}(\overline{\partial}_{E,m,q,\max}^{g,h})\subset \mathcal{D}(\overline{\partial}_{E,m,q}^{g,g})$ and  $\overline{\partial}_{E,m,q}^{g,g}\omega=\overline{\partial}_{E,m,q,\max}^{g,h}\omega$ for any $\omega\in \mathcal{D}(\overline{\partial}_{E,m,q,\max}^{g,h})$. 
We want to show now that $B\subset A$. Let us consider  any $\omega\in B$ and let $\eta_1\in \ker(\overline{\partial}_{E,m,q}^{g,g})$, $\eta_2\in A$ be such that $\omega=\eta_1+\eta_2$. It is clear that  $\ker(\overline{\partial}_{E,m,q,\max}^{g,h})=\ker(\overline{\partial}_{E,m,q}^{g,g})$.  Therefore we get immediately that $\eta_2\in \mathcal{D}(\overline{\partial}_{E,m,q,\max}^{g,h})$. Moreover, for any $\varphi\in \ker(\overline{\partial}_{E,m,q,\max}^{g,h})=\ker(\overline{\partial}_{E,m,q}^{g,g})$ we have $$\langle \varphi,\eta_2\rangle_{L^2\Omega^{m,q}(M,E,g,\rho)}+\langle\overline{\partial}_{E,m,q,\max}^{g,h}\varphi,\overline{\partial}_{E,m,q,\max}^{g,h}\eta_2\rangle_{L^2\Omega^{m,q+1}(A,E|_A,h,\rho)}=\langle \varphi,\eta_2\rangle_{L^2\Omega^{m,q}(M,E,g,\rho)}=0.$$ Hence $\eta_2\in B$ and thus $\eta_1=0$ and $\eta_2=\omega$ since $\eta_2-\omega=\eta_1\in \ker(\overline{\partial}_{E,m,q,\max}^{g,h})\cap B=\{0\}$. Finally let $\{\omega_k\}_{k\in \mathbb{N}}\subset B$ be a bounded sequence with respect to the graph norm of \eqref{opmix}. Thanks to the inclusion $B\subset A$ and the continuous inclusion $ L^2\Omega^{m,q+1}(A,E|_A,h,\rho)\hookrightarrow L^2\Omega^{m,q+1}(M,E,g,\rho)$ we know that $\{\omega_k\}_{k\in \mathbb{N}}\subset A$ and that it is bounded with respect to the graph norm of $\overline{\partial}_{E,m,q}^{g,g}$. Since \eqref{ci} is a compact inclusion there exists a subsequence $\{\psi_k\}_{k\in \mathbb{N}}\subset \{\omega_k\}_{k\in \mathbb{N}}$ and an element $\psi\in  L^2\Omega^{m,q}(M,E,g,\rho)=L^2\Omega^{m,q}(A,E|_A,g|_A,\rho|_A)$ such that $\psi_k\rightarrow \psi$ in $L^2\Omega^{m,q}(A,E|_A,g|_A,\rho|_A)$ as $k\rightarrow \infty$. Summarizing, given a  sequence  $\{\omega_k\}_{k\in \mathbb{N}}\subset B$ which is bounded with respect to the graph norm of \eqref{opmix}, we have proved the existence of a subsequence $\{\psi_k\}_{k\in \mathbb{N}}\subset \{\omega_k\}_{k\in \mathbb{N}}$ and an element $\psi\in L^2\Omega^{m,q}(A,E|_A,g|_A,\rho|_A)$ such that $\psi_k\rightarrow \psi$ in $L^2\Omega^{m,q}(A,E|_A,g|_A,\rho|_A)$ as $k\rightarrow \infty$. We can thus conclude that the Green operator  $$G_{\overline{\partial}_{E,m,q,\max}^{g,h}}:L^2\Omega^{m,q+1}(A,E|_A,h|_A,\rho|_A)\rightarrow L^2\Omega^{m,q}(A,E|_A,g|_A,\rho|_A)$$ is compact as desired.
\end{proof}

Let us now consider $M\times [0,1]$ and let  $p:M\times [0,1]\rightarrow M$ be the canonical projection. Let $g_s\in C^{\infty}(M\times [0,1], p^*T^*M\otimes p^*T^*M)$ be a smooth section of $p^*T^*M\otimes p^*T^*M\rightarrow M\times[0,1]$ such that:
\begin{enumerate}
\item $g_s(JX,JY)=g_s(X,Y)$ for any $X,Y\in \mathfrak{X}(M)$ and $s\in [0,1]$;
\item $g_s$ is a Hermitian metric on $M$ for any $s\in (0,1]$;
\item $g_0=h$;
\item There exists a positive constant $\frak{a}$ such that $g_0\leq \frak{a}g_s$ for each $s\in [0,1]$.
\end{enumerate}

Roughly speaking $g_s$ is a smooth family of $J$-invariant Riemannian metrics that degenerates to $h$ at $s=0$. Note that $(A,g_s|_A)$ is parabolic for any $s\in (0,1]$. Examples of such families of metrics are easy to build.  For instance if $f(s)$ is a smooth function on $[0,1]$ such that $f(0)=0$, $f(1)=1$ and $0<f(s)\leq 1$ for each $s\in (0,1)$ then $g_s:=(1-f(s))h+f(s)g$ satisfies the above requirements, see \cite[Prop. 4.2]{SpecBei}. Let $F_s\in C^{\infty}(M\times [0,1], p^*\mathrm{End}(TM))$ be a section of $p^*\mathrm{End}(TM)\rightarrow M\times [0,1]$ such that $g_1(F_s\cdot,\cdot)=g_s(\cdot,\cdot)$ for each $s\in [0,1]$. Clearly  $F_1=\id$ and $F_s$ is self-adjoint and positive definite on $M$ with respect to $g_1$ for each fixed $s\in (0,1]$. Following the notations of \S 1.1 we have $G_s:=(F_s^{-1})^{t}$ and the induced operators $$G_{s,\mathbb{C}}^{a,0}\in C^{\infty}(A\times [0,1], p^*\mathrm{End}(\Lambda^{a,0}(A)))\quad \mathrm{and}\quad G_{s,\mathbb{C}}^{0,b}\in C^{\infty}(A\times [0,1], p^*\mathrm{End}(\Lambda^{0,b}(A))).$$
The first main goal of this subsection is to show that the family of Green operators $\{G_{\overline{\partial}^{g_1,g_s}_{E,m,q}}\}$ converges in the compact sense to  $G_{\overline{\partial}^{g_1,h}_{E,m,q,\max}}$ as $s\rightarrow 0$. To prove this, we need  some preliminary results. 

\begin{prop}
\label{0110}
There exists a  suitable constant $\nu\geq 1$ such that the identity map $\id:\Omega_c^{m,q}(A,E|_A)\rightarrow \Omega^{m,q}_c(A,E|_A)$ gives rise to a continuous inclusion $$i:L^2\Omega^{m,q}(A,E|_A,g_{s}|_A,\rho|_A)\hookrightarrow L^2\Omega^{m,q}(A,E|_A,g_1|_A,\rho|_A),$$ which satisfies the following inequality
\begin{equation}
\label{fob}
\|\omega\|^2_{L^2\Omega^{m,q}(A,E|_A,g_1|_A,\rho|_A)}\leq \nu\|\omega\|^2_{L^2\Omega^{m,1}(A,E|_A,g_s|_A,\rho|_A)}
\end{equation}
for any $s\in [0,1]$, $q=0,...,m$ and $\omega\in L^2\Omega^{m,q}(A,E|_A,g_s|_A,\rho|_A)$. 
\end{prop}

\begin{proof}
This follows arguing as in \cite[Prop. 3.1 and Lemma 4.1]{SpecBei}. 
\end{proof}

We also have the following family of uniform continuous inclusions.

\begin{prop}
\label{unci}
There exists a suitable constant $a>0$ such that the identity map $\id:\Omega_c^{m,q}(A,E|_A)\rightarrow \Omega^{m,q}_c(A,E|_A)$ gives rise to a continuous inclusion $$i:L^2\Omega^{m,q}(A,E|_A,g_0|_A,\rho|_A)\hookrightarrow L^2\Omega^{m,q}(A,E|_A,g_s|_A,\rho|_A),$$ which satisfies the following inequality
\begin{equation}
\label{decia}
\|\omega\|^2_{L^2\Omega^{m,q}(A,E|_A,g_s|_A,\rho|_A)}\leq a\|\omega\|^2_{L^2\Omega^{m,q}(A,E|_A,g_0|_A,\rho|_A)}
\end{equation}
 for any  $s\in [0,1]$, $q=0,...,m$ and $\omega\in L^2\Omega^{m,q}(A,E|_A,g_0|_A,\rho|_A)$.
\end{prop}

\begin{proof}
This follows arguing as in \cite[Prop. 3.2]{SpecBei}. 
\end{proof}

We now recall  the following convergence result:

\begin{prop}
\label{relieve}
Let $\{s_n\}_{n\in \mathbb{N}}\subset [0,1]$ be any sequence such that $s_n\rightarrow 0$ as $n\rightarrow \infty$. Consider the Hilbert space $L^2\Omega^{m,q}(A,E|_A,g_{0}|_A,\rho|_A)$ and the sequence of Hilbert spaces $\{L^2\Omega^{m,q}(A,E|_A,g_{s_n}|_A,\rho|_A)\}_{n\in \mathbb{N}}$. Let $\mathcal{C}:=L^2\Omega^{m,q}(A,E|_A,g_{0}|_A,\rho|_A)$ and for any $n\in \mathbb{N}$, let $\Phi_n^{m,q}:\mathcal{C}\rightarrow L^2\Omega^{m,q}(A,E|_A,g_{s_n}|_A,\rho|_A)$ be the identity map $\id:L^2\Omega^{m,q}(A,E|_A,g_{0}|_A,\rho|_A)\rightarrow L^2\Omega^{m,q}(A,E|_A,g_{s_n}|_A,\rho|_A)$, which is well defined thanks to Prop. \ref{unci}. Then $$\{L^2\Omega^{m,q}(A,E|_A,g_{s_n}|_A,\rho|_A)\}_{n\in \mathbb{N}}\quad converges\ to\quad L^2\Omega^{m,q}(A,E|_A,g_{0}|_A,\rho|_A)$$ in the sense of \eqref{limitHilbert}.
\end{prop}

\begin{proof}
This follows by arguing as in \cite[Prop. 3.3]{SpecBei}. 
\end{proof}

We have the following immediate consequence :

\begin{cor}
\label{swim}
Let $\omega\in L^2\Omega^{m,q}(A,E|_A,g_0|_A,\rho|_A)$ be arbitrarily fixed. Then the constant sequence $\{\omega_n\}_{n\in \mathbb{N}}$, $\omega_n:=\omega$, viewed as a sequence where $\omega_n\in L^2\Omega^{m,q}(A,E|_A,g_{s_n}|_A,\rho|_A)$ for any $n\in \mathbb{N}$, converges strongly in the sense of Def. \ref{strong} to $\omega$ as $n\rightarrow \infty$.
\end{cor}

In the remaining part of this section we investigate the compact convergence of the operators $G_{\overline{\partial}_{E,m,q}^{g_1,g_s}}$ and $G_{\overline{\partial}_{E,m,q}^{g_s,h}}$, as $s\rightarrow 0$. To this aim, we need to prove various preliminary properties. 

\begin{lemma}
\label{lemma2}
Let $\phi\in \Omega_c^{m,q+1}(A,E|_A)$ and let $\{s_n\}_{n\in \mathbb{N}}\subset (0,1]$ be a sequence tending to $0$ as $n\rightarrow \infty$. Then $$\overline{\partial}_{E,m,q}^{g_1,g_{s_n},t}\phi\rightharpoonup \overline{\partial}_{E,m,q}^{g_1,h,t}\phi$$ as $n\rightarrow \infty$, that is, the sequence $\{\overline{\partial}_{E,m,q}^{g_1,g_{s_n},t}\phi\}$ converges weakly to $\overline{\partial}_{E,m,q}^{g_1,h,t}\phi$ in $L^2\Omega^{m,q}(A,E|_A,g_1|_A,\rho|_A)$ as $n\rightarrow \infty$.

\end{lemma}
\begin{proof}
Let $f:A\times [0,1]\rightarrow \mathbb{R}$ be the function that assigns to any $(p,s)\in A\times [0,1]$ the square of the pointwise norm of $\overline{\partial}^{g_1,g_{s},t}_{E,m,q}\phi$ in $p$ with respect to $g_1$ and $\rho$. By the facts $g_s\in C^{\infty}(A\times [0,1], p^*T^*M\otimes p^*T^*M)$ and $\phi\in \Omega^{m,1}_c(A)$, we know that $f$ is continuous on $A\times [0,1]$ and $\supp(f)\subset \supp(\phi)\times [0,1]$. In particular $\supp(f)$ is a compact subset of $A\times [0,1]$. Therefore there exists a positive constant $b\in \mathbb{R}$ such that $f(p,s)\leq b$ for any $p\in A$ and $s\in [0,1]$. This latter inequality tells us that $\|\overline{\partial}_{E,m,q}^{g_{1},g_{s_n},t}\phi\|^2_{L^2\Omega^{m,q}(A,E|_A,g_1|_A,\rho|_A)}=\int_Af(p,s)\dvol_{g_1}\leq b\vol_{g_1}(A)$ for any $s\in[0,1]$. Now, as we know that  $\{\|\overline{\partial}_{E,m,q}^{g_1,g_{s_n},t}\phi\|_{L^2\Omega^{m,q}(A,E|_A,g_1|_A,\rho|_A)}\}_{n\in \mathbb{N}}$ is a bounded sequence, to conclude the proof it is enough to fix a dense subset $Z$ of $L^2\Omega^{m,q}(A,E|_A,g_1|_A,\rho|_A)$ and to show that 
$$\lim_{n\rightarrow \infty}\langle\omega,\overline{\partial}_{E,m,q}^{g_1,g_{s_n},t}\phi\rangle_{L^2\Omega^{m,q}(A,E|_A,g_1|_A,\rho|_A)}=\langle\omega,\overline{\partial}_{E,m,q}^{g_1,h,t}\phi\rangle_{L^2\Omega^{m,q}(A,E|_A,g_1|_A,\rho|_A)}$$ for any $\omega\in Z$. Let us fix $Z:=\Omega^{m,q}_c(A,E|_A)$ and let $\omega\in \Omega_c^{m,q}(A,E|_A)$. 
Thanks to Prop. \ref{relieve} and Cor. \ref{swim}  we have 
\begin{align}
\nonumber  \lim_{n\rightarrow \infty}\langle\omega,\overline{\partial}_{E,m,q}^{g_1,g_{s_n},t}\phi\rangle_{L^2\Omega^{m,q}(A,E|_A,g_1|_A,\rho|_A)}&=\lim_{n\rightarrow \infty}\langle \overline{\partial}_{E,m,q}\omega,\phi\rangle_{L^2\Omega^{m,q+1}(A,E|_A,g_{s_n}|_A,\rho|_A)}\\
\nonumber &=\langle \overline{\partial}_{E,m,q}\omega,\phi\rangle_{L^2\Omega^{m,q+1}(A,E|_A,h|_A,\rho|_A)}\\
\nonumber &=\langle \omega,\overline{\partial}_{E,m,q}^{g_1,h,t}\phi\rangle_{L^2\Omega^{m,q}(A,E|_A,g_1|_A,\rho|_A)}
\end{align}
as desired.
\end{proof}

The next lemma provides an extension of Prop. \ref{sameop}.

\begin{lemma}
\label{Seb}
For each $s\in (0,1]$ the following three operators coincide:
\begin{align}
& \nonumber \overline{\partial}_{E,p,q,\max/\min}^{g_1,g_s}:L^2\Omega^{p,q}(A,E|_A,g|_A,\rho|_A)\rightarrow L^2\Omega^{p,q+1}(A,E|_A,g_s|_A,\rho|_A);\\
\label{ucsex}
& \overline{\partial}_{E,p,q}^{g_1,g_s}:L^2\Omega^{p,q}(M,E,g_1,\rho)\rightarrow L^2\Omega^{p,q+1}(M,E,g_s,\rho),
\end{align}
with \eqref{ucsex}  the unique closed extension of  $\overline{\partial}_{E,p,q}:\Omega^{p,q}(M,E)\rightarrow \Omega^{p,q+1}(M,E)$ viewed as an unbounded and densely defined operator acting between $L^2\Omega^{p,q}(M,E,g_1,\rho)$ and  $L^2\Omega^{p,q+1}(M,E,g_s,\rho).$
\end{lemma}
\begin{proof}
This is an immediate consequence of Prop. \ref{sameop} and the fact that $g_1$ and $g_s$ are quasi-isometric for each $s\in (0,1]$.
\end{proof}

As in the proof of Prop. \ref{Gianni}, we define $B^{g_1,h}_{m,q}\subset L^2\Omega^{m,q}(A,E|,g_1|,\rho|_A)$ as $$B^{g_1,h}_{m,q}:=\left(\ker(\overline{\partial}_{E,m,q,\max}^{g_1,h})\right)^{\bot}\cap \mathcal{D}(\overline{\partial}_{E,m,q,\max}^{g_1,h}).$$ When $s\in (0,1]$ we consider the operator \eqref{ucsex} and in analogy with the above construction we define $B^{g_1,g_s}_{m,q}\subset L^2\Omega^{m,q}(A,E|_A,g_1|_A,\rho|_A)$ as $$B^{g_1,g_s}_{m,q}:=\left(\ker(\overline{\partial}_{E,m,q}^{g_1,g_s})\right)^{\bot}\cap \mathcal{D}(\overline{\partial}_{E,m,q,}^{g_1,g_s}).$$ 

\begin{lemma}
\label{SebG}
For each $s\in (0,1]$  we have
$$\ker(\overline{\partial}_{E,m,q}^{g_1,g_1})=\ker(\overline{\partial}_{E,m,q}^{g_1,g_s})\quad\quad and\quad\quad B^{g_1,g_1}_{m,q}=B^{g_1,g_s}_{m,q}\quad  for\ each\ s\in (0,1].$$ 
If $s=0$  we have
$$\ker(\overline{\partial}_{E,m,q}^{g_1,g_1})=\ker(\overline{\partial}_{E,m,q,\max}^{g_1,h})\quad\quad and\quad\quad B^{g_1,h}_{m,q}\subset B^{g_1,g_1}_{m,q}.$$ 
\end{lemma}

\begin{proof}
When $s\in (0,1]$ the above equalities follow immediately by Lemma \ref{Seb} and the fact that $g_1$ and $g_s$ are quasi-isometric for each $s\in (0,1]$. When $s=0$ the above statements follow by Prop. \ref{Pp} and Prop. \ref{unci}.
\end{proof}
 We have now the following immediate 
\begin{cor}
\label{allGreen}
For each $s\in (0,1]$ the Green operator $$G_{\overline{\partial}_{E,m,q}^{g_1,g_s}}:L^2\Omega^{m,q+1}(M,E,g_s,\rho)\rightarrow L^2\Omega^{m,q}(M,E,g_1,\rho)$$ exists and is compact. 
\end{cor}

\begin{proof}
This follows immediately by Lemma \ref{Seb} and Lemma \ref{SebG}.
\end{proof}

\begin{lemma}
\label{lemma7}
For each $s\in (0,1]$ let $\lambda_{m,q}^{g_1,g_s}$ be defined as $$\lambda_{m,q}^{g_1,g_s}:=\inf_{0\neq \eta\in B^{g_1,g_s}_{m,q}}\frac{\|\overline{\partial}_{E,m,q}^{g_1,g_s}\eta\|_{L^2\Omega^{m,q+1}(A,E|_A,g_s|_A,\rho|_A)}}{\|\eta\|_{L^2\Omega^{m,q}(A,E|_A,g_1|_A,\rho|_A)}}.$$ Similarly let $$\lambda_{m,q}^{g_1,h}:=\inf_{0\neq\eta\in B^{g_1,h}_{m,q}}\frac{\|\overline{\partial}_{E,m,q,\max}^{g_1,h}\eta\|_{L^2\Omega^{m,q+1}(A,E|_A,h|_A,\rho|_A)}}{\|\eta\|_{L^2\Omega^{m,q}(A,E|_A,g_1|_A,\rho|_A)}}.$$
Let $\nu>0$ and $a>0$ be the constants appearing in \eqref{fob} and \eqref{decia}, respectively. Then we have $$0<\lambda_{m,q}^{g_1,g_1}\leq \sqrt{\nu}\lambda_{m,q}^{g_1,g_s}\leq \sqrt{a\nu}\lambda_{m,q}^{g_1,h}$$ for each $s\in (0,1]$ and $q=0,...,m$.
\end{lemma}

\begin{proof}
The inequality $0<\lambda_{m,q}^{g_1,g_1}$ follows by the fact that the operator $$\overline{\partial}_{E,m,q}^{g_1,g_1}:L^2\Omega^{m,q}(A,E|_A,g_1|_A,\rho|_A)\rightarrow L^2\Omega^{m,q+1}(A,E|_A,g_1|_A,\rho|_A)$$ has closed range. Let us show now that $\lambda_{m,q}^{g_1,g_1}\leq \sqrt{\nu}\lambda_{m,q}^{g_1,g_s}$ for each $s\in (0,1]$. Thanks to Lemma \ref{SebG} we know that  $B^{g_1,g_s}_{m,q}=B^{g_1,g_1}_{m,q}$ and $\overline{\partial}^{g_1,g_1}_{E,m,q}\eta=\overline{\partial}^{g_1,g_s}_{E,m,q}\eta$ for each $ s\in (0,1]$ and $\eta\in B^{g_1,g_s}_{m,q}$. In this way by Prop. \ref{0110} we obtain
$$
\begin{aligned}
\sqrt{\nu}\lambda_{m,q}^{g_1,g_s}&:=\inf_{0\ne\eta\in B_{m,q}^{g_1,g_s}}\frac{\sqrt{\nu}\|\overline{\partial}_{E,m,q}^{g_1,g_s}\eta\|_{L^2\Omega^{m,q+1}(A,E|_A,g_s|_A,\rho|_A)}}{\|\eta\|_{L^2\Omega^{m,q}(A,E|_A,g_1|_A,\rho|_A)}}\\
&=\inf_{0\neq\eta\in B_{m,q}^{g_1,g_1}}\frac{\sqrt{\nu}\|\overline{\partial}_{E,m,q}^{g_1,g_1}\eta\|_{L^2\Omega^{m,q+1}(A,E|_A,g_s|_A,\rho|_A)}}{\|\eta\|_{L^2\Omega^{m,q}(A,E|_A,g_1|_A,\rho|_A)}}\\
&\geq \inf_{0\neq\eta\in B_{m,q}^{g_1,g_1}}\frac{\|\overline{\partial}_{E,m,q}^{g_1,g_1}\eta\|_{L^2\Omega^{m,q+1}(A,E|_A,g_1|_A,\rho|_A)}}{\|\eta\|_{L^2\Omega^{m,q}(A,E|_A,g_1|_A,\rho|_A)}}\\
&=\lambda_{m,q}^{g_1,g_1}.
\end{aligned}
$$
We now tackle the remaining inequality. By Lemma \ref{SebG} we know that  $B_{m,q}^{g_1,h}\subset B_{m,q}^{g_1,g_s}$ and $\overline{\partial}_{E,m,q}^{g_1,g_s}\eta=\overline{\partial}_{E,m,q,\max}^{g_1,h}\eta$ for each $s\in (0,1]$ and $\eta\in B_{m,q}^{g_1,h}$. In this way thanks to Prop. \ref{unci} we obtain
$$
\begin{aligned}
\sqrt{a}\lambda_{m,q}^{g_1,h}&:=\inf_{0\neq\eta\in B_{m,q}^{g_1,h}}\frac{\sqrt{a}\|\overline{\partial}_{E,m,q,\max}^{g_1,h}\eta\|_{L^2\Omega^{m,q+1}(A,E|_A,h|_A,\rho|_A)}}{\|\eta\|_{L^2\Omega^{m,q}(A,E|_A,g_1|_A,\rho|_A)}}\\
&\geq\inf_{0\neq\eta\in B_{m,q}^{g_1,h}}\frac{\|\overline{\partial}_{E,m,q}^{g_1,g_s}\eta\|_{L^2\Omega^{m,q+1}(A,E|_A,g_s|_A,\rho|_A)}}{\|\eta\|_{L^2\Omega^{m,q}(A,E|_A,g_1|_A,\rho|_A)}}\\
&\geq \inf_{0\neq \eta\in B_{m,q}^{g_1,g_s}}\frac{\|\overline{\partial}_{E,m,q,}^{g_1,g_s}\eta\|_{L^2\Omega^{m,q+1}(A,E|_A,g_s|_A,\rho|_A)}}{\|\eta\|_{L^2\Omega^{m,q}(A,E|_A,g_1|_A,\rho|_A)}}\\
&=\lambda_{m,q}^{g_1,g_s}
\end{aligned}
$$
for each $s\in (0,1]$.
\end{proof}

We now have  the first main result of this section.

\begin{teo}
\label{compact1}
Let $\{s_n\}_{n\in \mathbb{N}}\subset (0,1]$ be any sequence such that $s_n\rightarrow 0$ as $n\rightarrow \infty$ and let $$G_{\overline{\partial}_{E,m,q}^{g_1,g_{s_n}}}:L^2\Omega^{m,q+1}(M,E,g_{s_n},\rho)\rightarrow L^2\Omega^{m,q}(M,E,g_1,\rho)$$ be the Green operator of $\overline{\partial}_{E,m,q}^{g_1,g_{s_n}}:L^2\Omega^{m,q}(M,E,g_{s_1},\rho)\rightarrow L^2\Omega^{m,q+1}(M,E,g_{s_n},\rho)$. Then $$G_{\overline{\partial}_{E,m,q}^{g_1,g_{s_n}}}\rightarrow G_{\overline{\partial}_{E,m,q,\max}^{g_1,h}}\ compactly\ as\ n\rightarrow \infty.$$
\end{teo}

\begin{proof}
Let $\{\eta_{s_n}\}_{n\in \mathbb{N}}$ with $\eta_{s_n}\in L^2\Omega^{m,q+1}(A,E|_A,g_{s_n}|_A,\rho|_A)$, be a weakly convergent sequence to some $\eta\in L^2\Omega^{m,q+1}(A,E|_A,h|_A,\rho|_A)$, as $n\rightarrow \infty$. Let $\eta_{s_n,1}$ be the orthogonal projection of $\eta_{s_n}$ on $\im(\overline{\partial}_{E,m,q}^{g_1,g_{s_n}})$ and let $\eta_{s_n,2}$ be the orthogonal projection of $\eta_{s_n}$ on $(\im(\overline{\partial}_{E,m,q}^{g_1,g_{s_n}}))^{\bot}$. Analogously let $\eta_{1}$ and $\eta_2$ be the orthogonal projection of $\eta$ on $\im(\overline{\partial}_{E,m,q,\max}^{g_1,h})$ and  $(\im(\overline{\partial}_{E,m,q,\max}^{g_1,h}))^{\bot}$, respectively. 
We have $G_{\overline{\partial}_{E,m,q}^{g_1,g_{s_n}}}\eta_{s_n,2}=0= G_{\overline{\partial}_{E,m,q,\max}^{g_1,h}}\eta_2$ for each $n\in \mathbb{N}$. Thus, to prove this proposition, we have to show that
$$G_{\overline{\partial}_{E,m,q}^{g_1,g_{s_n}}}\eta_{s_n,1}\rightarrow G_{\overline{\partial}_{E,m,q,\max}^{g_1,h}}\eta_1\quad \text{in}\quad L^2\Omega^{m,q}(A,E|_A,g_1|_A,\rho|_A)\ \text{as}\ n\rightarrow \infty.$$
As a first step we observe that there exists a constant $c>0$ such that $$\|\eta_{s_n,1}\|_{L^2\Omega^{m,q+1}(A,E|_A,g_{s_n}|_A,\rho|_A)}\leq \|\eta_{s_n}\|_{L^2\Omega^{m,q+1}(A,E|_A,g_{s_n}|_A,\rho|_A)}\leq c$$ for each $n\in \mathbb{N}$, see Prop. \ref{wibounded}.  Now consider the sequence $\{G_{\overline{\partial}_{E,m,q}^{g_1,g_{s_n}}}\eta_{s_n}\}_{n\in \mathbb{N}}$. By construction we have $G_{\overline{\partial}_{E,m,q}^{g_1,g_{s_n}}}\eta_{s_n}\in B_{m,q}^{g_1,g_{s_n}}$ and so by Lemma \ref{SebG} we obtain $G_{\overline{\partial}_{E,m,q}^{g_1,g_{s_n}}}\eta_{s_n}\in B_{m,q}^{g_1,g_1}$ and $$\overline{\partial}_{E,m,q}^{g_1,g_1}(G_{\overline{\partial}_{E,m,q}^{g_1,g_{s_n}}}\eta_{s_n})=\overline{\partial}_{E,m,q}^{g_1,g_{s_n}}(G_{\overline{\partial}_{E,m,q}^{g_1,g_{s_n}}}\eta_{s_n})=\eta_{s_n,1}.$$ In particular by applying Prop. \ref{0110} we obtain   
$$
\begin{aligned}
\|\overline{\partial}_{E,m,q}^{g_1,g_1}(G_{\overline{\partial}_{E,m,q}^{g_1,g_{s_n}}}\eta_{s_n})\|_{L^2\Omega^{m,q+1}(A,E|_A,g_1|_A,\rho|_A)}&=\|\eta_{s_n,1}\|_{L^2\Omega^{m,q+1}(A,E|_A,g_1|_A,\rho|_A)}\\
& \leq\sqrt{\nu}\|\eta_{s_n,1}\|_{L^2\Omega^{m,q+1}(A,E|_A,g_{s_n}|_A,\rho|_A)}\leq \sqrt{\nu}c.
\end{aligned}$$
 Moreover by Lemma \ref{lemma7} we have
$$
\begin{aligned}
\|G_{\overline{\partial}_{E,m,q}^{g_1,g_{s_n}}}\eta_{s_n}\|_{L^2\Omega^{m,q}(A,E|_A,g_1|_A,\rho|_A)}&=\|G_{\overline{\partial}_{E,m,q}^{g_1,g_{s_n}}}\eta_{s_n,1}\|_{L^2\Omega^{m,q}(A,E|_A,g_1|_A,\rho|_A)}\\
&\leq \frac{1}{\lambda_{m,q}^{g_1,g_s}}\|\eta_{s_n,1}\|_{L^2\Omega^{m,q+1}(A,E|_A,g_{s_n}|_A,\rho|_A)}\leq \frac{\sqrt{\nu}}{\lambda_{m,q}^{g_1,g_1}}c
\end{aligned}$$ for each $n\in \mathbb{N}$.
We have just shown that $\{G_{\overline{\partial}_{E,m,q}^{g_1,g_{s_n}}}\eta_{s_n}\}\subset B_{m,q}^{g_1,g_1}$ is a bounded sequence with respect to the graph norm of $\overline{\partial}_{E,m,q}^{g_1,g_1}$. Since $G_{\overline{\partial}_{E,m,q}^{g_1,g_1}}$ is compact there exists a subsequence $\{r_n\}_{n\in \mathbb{N}}\subset \{s_n\}_{n\in \mathbb{N}}$ and  elements $\psi\in L^2\Omega^{m,q}(A,E|_A,g_1|_A,\rho|_A)$ and $\chi\in L^2\Omega^{m,q+1}(A,E|_A,h|_A,\rho|_A)$ such that $$G_{\overline{\partial}_{E,m,q}^{g_1,g_1}}\eta_{r_n,1}\rightarrow \psi\ \text{in}\ L^2\Omega^{m,q}(A,E|_A,g_1|_A,\rho)\ \text{as}\ n\rightarrow \infty\quad\quad \text{and}\quad\quad \eta_{r_n,1}\rightarrow \chi\ \text{weakly}\ \text{as}\ n\rightarrow \infty.$$ 
Now, to complete the proof, we have to show that  $$\psi \in B_{m,q}^{g_1,h}\quad \text{and}\quad  \overline{\partial}_{E,m,q,\max}^{g_1,h}\psi=\chi=\eta_1.$$  Let $\phi\in \Omega^{m,q+1}_c(A,E|_A)$ be arbitrarily fixed. By Lemma \ref{lemma2} we have
$$
\begin{aligned}
 \langle \psi, \overline{\partial}_{E,m,q}^{g_1,h,t}\phi\rangle_{L^2\Omega^{m,q}(A,E|_A,g_1|_A,\rho|_A)}&=  \lim_{n\rightarrow \infty} \langle G_{\overline{\partial}_{E,m,q}^{g_1,g_{r_n}}}\eta_{r_n,1}, \overline{\partial}_{E,m,q}^{g_1,g_{r_n},t}\phi\rangle_{L^2\Omega^{m,q}(A,E|_A,g_1|_A,\rho|_A)}\\
& =\lim_{n\rightarrow \infty} \langle \overline{\partial}_{E,m,q}^{g_1,g_{r_n}}(G_{\overline{\partial}_{E,m,q}^{g_1,g_{r_n}}}\eta_{r_n,1}), \phi\rangle_{L^2\Omega^{m,q+1}(A,E|_A,g_{r_n}|_A,\rho|_A)}\\
&= \lim_{n\rightarrow \infty} \langle\eta_{r_n,1}, \phi\rangle_{L^2\Omega^{m,q+1}(A,E|_A,g_{r_n}|_A,\rho|_A)}\\
&= \langle\chi, \phi\rangle_{L^2\Omega^{m,q+1}(A,E|_A,h,\rho|_A)}.
\end{aligned}
$$
This shows that $\psi \in \mathcal{D}(\overline{\partial}_{E,m,q,\max}^{g_1,h})$ and that $\overline{\partial}_{E,m,q,\max}^{g_1,h}\psi=\chi$. Now, considering any $\alpha\in \ker(\overline{\partial}_{E,m,q,\max}^{g_1,h})$ and keeping in mind that $\ker(\overline{\partial}_{E,m,q,\max}^{g_1,h})=\ker(\overline{\partial}_{E,m,q}^{g_1,g_s})$ for each $s\in (0,1]$, we have $$\langle\alpha,\psi\rangle_{L^2\Omega^{m,q}(A,E|_A,g_1|_A,\rho|_A)}=\lim_{n\rightarrow \infty}\langle\alpha,G_{\overline{\partial}_{E,m,q}^{g_1,g_{r_n}}}\eta_{r_n,1}\rangle_{L^2\Omega^{m,q}(A,E|_A,g_1|_A,\rho|_A)}=0.$$
Hence we can conclude that $\psi \in B_{m,q}^{g_1,h}$ and that $\overline{\partial}_{E,m,q,\max}^{g_1,h}\psi=\chi$. We are left to show that $\eta_1=\chi$. Let $\xi\in \im(\overline{\partial}_{E,m,q,\max}^{g_1,h})$ be arbitrarily fixed. Keeping in mind that $\im(\overline{\partial}_{E,m,q,\max}^{g_1,h})\subset \im(\overline{\partial}_{E,m,q}^{g_1,g_{s_n}})$ for each $n\in \mathbb{N}$ we have 
$$
\begin{aligned}
\langle \xi,\chi\rangle_{L^2\Omega^{m,q+1}(A,E|_A,h|_A,\rho|_A)}&=\lim_{n\rightarrow \infty}\langle \xi,\eta_{r_n,1}\rangle_{L^2\Omega^{m,q+1}(A,E|_A,g_{r_n}|_A,\rho|_A)}\\
&=\lim_{n\rightarrow \infty}\langle \xi,\eta_{r_n}\rangle_{L^2\Omega^{m,q+1}(A,E|_A,g_{r_n}|_A,\rho|_A)}\\
&= \langle \xi,\eta\rangle_{L^2\Omega^{m,q+1}(A,E|_A,h|_A,\rho|_A)}\\
&=\langle \xi,\eta_1\rangle_{L^2\Omega^{m,q+1}(A,E|_A,h|_A,\rho|_A)}.
\end{aligned}
$$
In conclusion for each $\xi \in \im(\overline{\partial}_{E,m,q,\max}^{g_1,h})$ we have $$\langle \xi,\chi\rangle_{L^2\Omega^{m,q+1}(A,E|_A,h|_A,\rho|_A)}=\langle \xi,\eta_1\rangle_{L^2\Omega^{m,q+1}(A,E|_A,h|_A,\rho|_A)}$$ and so we can conclude that $\eta_1=\chi$. Therefore we have shown that $\psi \in B_{m,q}^{g_1,h}$ and that $\overline{\partial}_{E,m,q,\max}^{g_1,h}\psi=\chi=\eta_1$, that is $\psi=G_{\overline{\partial}_{E,m,q,\max}^{g_1,h}}\eta$.
Summarizing,  given a weakly convergent sequence $\eta_{s_n}\rightarrow \eta$ as $n\rightarrow \infty$ with $\eta_{s_n}\in L^2\Omega^{m,q+1}(A,E|_A,g_{s_n}|_A,\rho|_A)$ and $\eta\in L^2\Omega^{m,q+1}(A,E|_A,h|_A,\rho|_A)$, we have proved the existence of a subsequence $\{r_n\}_{n\in \mathbb{N}}$ such that $$G_{\overline{\partial}_{E,m,q}^{g_1,g_{r_n}}}\eta_{r_n}\rightarrow G_{\overline{\partial}_{E,m,q,\max}^{g_1,h}}\eta$$ in $L^2\Omega^{m,q}(A,g_1|_A,E|_A,\rho|_A)$ as $n\rightarrow \infty$. Now if we fix an arbitrary subsequence $\{\ell_n\}_{n\in \mathbb{N}}\subset \{s_n\}_{n\in \mathbb{N}}$ and we repeat the above argument with $\{\eta_{\ell_n}\}_{n\in \mathbb{N}}$ we obtain a further subsequence $\{t_n\}_{n\in\mathbb{N}}\subset \{\ell_n\}_{n\in \mathbb{N}}$ such that $G_{\overline{\partial}_{E,m,q}^{g_1,g_{t_n}}}\eta_{t_n}\rightarrow G_{\overline{\partial}_{E,m,q,\max}^{g_1,h}}\eta$ in $L^2\Omega^{m,q}(A,g_1|_A,E|_A,\rho|_A)$ as $n\rightarrow \infty$. Clearly this allows us to conclude that $$G_{\overline{\partial}_{E,m,q}^{g_1,g_{s_n}}}\eta_{s_n}\rightarrow G_{\overline{\partial}_{E,m,q,\max}^{g_1,h}}\eta\quad \text{in}\quad L^2\Omega^{m,q}(A,g_1|_A,E|_A,\rho|_A)\quad \text{as}\quad n\rightarrow \infty$$ and therefore $$G_{\overline{\partial}_{E,m,q}^{g_1,g_{s_n}}}\rightarrow G_{\overline{\partial}_{E,m,q,\max}^{g_1,h}}\quad \text{compactly as}\quad n\rightarrow \infty.$$ 
\end{proof}

The next goal is to establish the compact convergence of the sequence $\{G_{\overline{\partial}_{E,m,q,\max}^{g_{s_n},h}}\}$ to $G_{\overline{\partial}_{E,m,q,\max}^{h,h}}$. To do this, we need other auxiliary results.

\begin{lemma}
\label{lemma3}
Let $\phi\in \Omega_c^{m,q+1}(A,E|_A)$ and let $\{s_n\}_{n\in \mathbb{N}}\subset (0,1]$ be a sequence tending to $0$ as $n\rightarrow \infty$. Then $$\overline{\partial}_{E,m,q}^{g_{s_n},h,t}\phi\rightarrow \overline{\partial}_{E,m,q}^{h,h,t}\phi$$ strongly  as $n\rightarrow \infty$.
\end{lemma}
\begin{proof}
As a first step we want to show that over $A$ and for each $s\in [0,1]$ the operator $\overline{\partial}_{E,m,q}^{g_{s},h,t}$ can be written as the composition of $\overline{\partial}_{E,m,q}^{g_1,h,t}$ with an endomorphism of $\Lambda^{m,q}(A)\otimes E$ that depends smoothly on $s$. We then use this decomposition to tackle the above limit. As usual let $p:A\times [0,1]\rightarrow A$ be the left projection and let  $S_s^{m,q}\in C^{\infty}(A\times [0,1], \mathrm{End}(p^*\Lambda^{m,q}(A)\otimes p^*E))$ be defined as  $S^{m,q}_s:=\det(G^{1,0}_{\mathbb{C},s})\otimes G^{0,q}_{\mathbb{C},s}\otimes \id$. Note that $S^{m,q}_s$ is the family of endomorphisms of $p^*\Lambda^{m,q}(A)\otimes p^*E$ such that $g^*_{s,m,q,\rho}(\bullet,\bullet)=g^*_{1,m,q,\rho}(S^{m,q}_s\bullet,\bullet)$. The previous equality tells us that $$g^*_{s,m,q,\rho}((S_s^{m,q})^{-1}\bullet,\bullet)=g^*_{1,m,q,\rho}(\bullet,\bullet)$$ which in turn implies that $(S_s^{m,q})^{-1}$ is fiberwise self-adjoint with respect to $g^*_{s,m,q,\rho}$ for each fixed $s\in [0,1]$. Besides $S_s^{m,q}$ let us also introduce $T_s^{m,q}:=\id\otimes G^{0,q}_{\mathbb{C},s}\otimes \id$. Clearly also $T_s^{m,q}\in C^{\infty}(A\times [0,1], \mathrm{End}(p^*\Lambda^{m,q}(A)\otimes p^*E))$. Given any $\varphi\in \Omega_c^{m,q}(A,E|_A)$ and $\phi\in \Omega_c^{m,q+1}(A,E|_A)$ we have
$$
\begin{aligned}
\langle\overline{\partial}_{E,m,q}\varphi,\phi\rangle_{L^2\Omega^{m,q+1}(A,E|_A,h|_A,\rho|_A)}&=\langle\varphi,\overline{\partial}_{E,m,q}^{g_1,h,t}\phi\rangle_{L^2\Omega^{m,q}(A,E|_A,g_1|_A,\rho|_A)}\\
&=\int_Ag^*_{1,m,q,\rho}(\varphi,\overline{\partial}_{E,m,q}^{g_1,h,t}\phi)\dvol_{g_1}\\
&=\int_Ag^*_{s,m,q,\rho}((S^{m,q}_s)^{-1}\varphi,\overline{\partial}_{E,m,q}^{g_1,h,t}\phi){\det}^{-\frac{1}{2}}(F_s)\dvol_{g_s}\\
&=\int_Ag^*_{s,m,q,\rho}(\varphi,(S^{m,q}_s)^{-1}(\overline{\partial}_{E,m,q}^{g_1,h,t}\phi)){\det}^{-\frac{1}{2}}(F_s)\dvol_{g_s}\\
&=\int_Ag^*_{s,m,q,\rho}(\varphi,(\det(G^{1,0}_{\mathbb{C},s})\otimes G^{0,q}_{\mathbb{C},s}\otimes \id)^{-1}(\overline{\partial}_{E,m,q}^{g_1,h,t}\phi)){\det}^{-\frac{1}{2}}(F_s)\dvol_{g_s}\\
&= \int_Ag^*_{s,m,q,\rho}(\varphi,(\id\otimes G^{0,q}_{\mathbb{C},s}\otimes \id)^{-1}(\overline{\partial}_{E,m,q}^{g_1,h,t}\phi))\dvol_{g_s}\\
&= \int_Ag^*_{s,m,q,\rho}(\varphi,(T_s^{m,q})^{-1}(\overline{\partial}_{E,m,q}^{g_1,h,t}\phi))\dvol_{g_s}\\
&= \langle\varphi,(T_s^{m,q})^{-1}(\overline{\partial}_{E,m,q}^{g_1,h,t}\phi)\rangle_{L^2\Omega^{m,q}(A,E|_A,h|_A,g_s|_A)}.
\end{aligned}
$$
Summarising we have shown that for each $\varphi\in \Omega_c^{m,q}(A,E|_A)$ and $\phi\in \Omega_c^{m,q+1}(A,E|_A)$ we have
\begin{equation}
\label{adad}
\langle\overline{\partial}_{E,m,q}\varphi,\phi\rangle_{L^2\Omega^{m,q+1}(A,E|_A,h|_A,\rho|_A)}=\langle\varphi,(T_s^{m,q})^{-1}(\overline{\partial}_{E,m,q}^{g_1,h,t}\phi)\rangle_{L^2\Omega^{m,q}(A,E|_A,g_s|_A,\rho|_A)}
\end{equation}
and thus we can conclude that for each $s\in [0,1]$ it holds 
\begin{equation}
\label{sambu}
\overline{\partial}_{E,m,q}^{g_s,h,t}=(T^{m,q}_s)^{-1}\circ \overline{\partial}_{E,m,q}^{g_1,h,t}.
\end{equation}
 In particular for $s=0$ we have $$\overline{\partial}_{E,m,q}^{h,h,t}=(T_0^{m,q})^{-1}\circ \overline{\partial}_{E,m,q}^{g_1,h,t}.$$
We are now in position to show that given any arbitrarily fixed $\phi\in \Omega_c^{m+1,q}(A,E|_A)$ and a sequence $\{s_n\}_{n\in \mathbb{N}}\subset (0,1]$ with $s_n\rightarrow 0$ as $n\rightarrow \infty$ we have $\overline{\partial}_{E,m,q}^{g_{s_n},h,t}\phi\rightarrow \overline{\partial}_{E,m,q}^{h,h,t}\phi$ strongly  as $n\rightarrow \infty$. To put it differently, thanks to \eqref{sambu}, Prop. \ref{relieve} and Cor. \ref{swim}, we have to show that $$\lim_{n\rightarrow \infty}\|(T^{m,q}_{s_n})^{-1}( \overline{\partial}_{E,m,q}^{g_1,h,t}\phi)-\Phi_{s_n}^{m,q}((T^{m,q}_0)^{-1}( \overline{\partial}_{E,m,q}^{g_1,h,t}\phi))\|_{L^2\Omega^{m,q}(A,E|_A,g_{s_n}|_A,\rho|_A)}=0.$$
Let us denote $\psi:=\overline{\partial}_{E,m,q}^{g_1,h,t}\phi$. Since $$\Phi_{s_n}^{m,q}:L^2\Omega^{m,q}(A,E|_A,h|_A,\rho|_A)\rightarrow L^2\Omega^{m,q}(A,E|_A,g_{s_n}|_A,\rho|_A)$$ is nothing but the continuous inclusion induced by the identity $\id:\Omega_c^{m,q}(A,E|_A)\rightarrow \Omega_c^{m,q}(A,E|_A)$, the above limit amounts to proving that $$\lim_{n\rightarrow \infty}\|((T^{m,q}_{s_n})^{-1}-(T^{m,q}_{0})^{-1})(\psi)\|_{L^2\Omega^{m,q}(A,E|_A,g_{s_n}|_A,\rho|_A)}=0.$$
Let us define the function $f:A\times [0,1]\rightarrow \mathbb{R}$ as $$f(p,s):=|((T^{m,q}_{s})^{-1}-(T^{m,q}_{0})^{-1})(\psi)|_{g^*_{s,m,q,\rho}}^2(p).$$ In other words $f$ is the function that assigns to each $p\in A$ and $s\in [0,1]$ the square of the pointwise norm of the section $(T^{m,q}_s)^{-1}(\psi)-(T^{m,q}_0)^{-1}(\psi)$ in $p$ with respect to $g_s$ and $\rho$. Note that $f\in C^{\infty}_c(A\times [0,1],\mathbb{R})$ and for each fixed $p\in A$ we have 
\begin{equation}
\label{limww}
\lim_{s\rightarrow0}f(p,s)=0.
\end{equation}
In this way we obtain
$$
\begin{aligned}
\lim_{n\rightarrow \infty}&\|(T^{m,q}_{s_n})^{-1}( \overline{\partial}_{E,m,q}^{g_1,h,t}\phi)-(T^{m,q}_0)^{-1}( \overline{\partial}_{E,m,q}^{g_1,h,t}\phi)\|^2_{L^2\Omega^{m,q}(A,E|_A,g_{s_n}|_A,\rho|_A)}=\\
&=\lim_{n\rightarrow \infty}\|((T^{m,q}_{s_n})^{-1}-(T^{m,q}_0)^{-1})(\psi)\|^2_{L^2\Omega^{m,q}(A,E|_A,g_{s_n}|_A,\rho|_A)}\\
&= \lim_{n\rightarrow\infty}\int_A |((T^{m,q}_{s_n})^{-1}-(T^{m,q}_0)^{-1})(\psi)|_{g^*_{s_n,m,q,\rho}}^2\dvol_{g_{s_n}}\\
&=\lim_{n\rightarrow \infty}\int_Af(p,s_n)\sqrt{\det(F_{s_n})}\dvol_{g_1}.
\end{aligned}$$
Observe now that  $\det(F_s)\in C^{\infty}(A\times [0,1],\mathbb{R})$ and therefore we have $f(p,s)\sqrt{\det(F_{s})}\in C_c(A\times [0,1],\mathbb{R})$. Thus by the fact that $\vol_{g_1}(A)<\infty$ we can apply the Lebesgue dominated convergence theorem and \eqref{limww} to conclude that 
$$\lim_{n\rightarrow \infty}\int_Af(p,s_n)\sqrt{\det(F_{s_n})}\dvol_{g_1}=\int_A\lim_{n\rightarrow \infty}f(p,s_n)\sqrt{\det(F_{s_n})}\dvol_{g_1}=0.$$
Summarising, we  proved that $\overline{\partial}_{E,m,q}^{g_{s_n},h,t}\phi\rightarrow \overline{\partial}_{E,m,q}^{h,h,t}\phi$ strongly  as $n\rightarrow \infty$, as required.
\end{proof}

\begin{lemma}
\label{lemma4}
Let $\{s_n\}_{n\in \mathbb{N}}\subset (0,1]$ be a sequence with $s_n\rightarrow 0$ as $n\rightarrow \infty$. Let $\{\psi_{s_n}\}_{n\in \mathbb{N}}$, with $\psi_{s_n}\in \ker(\overline{\partial}_{E,m,q,\max}^{g_{s_n},h})$, be a weakly convergent sequence to some $\psi\in L^2\Omega^{m,q}(A,E|_A,h|_A,\rho|_A)$ as $n\rightarrow \infty$. Then $\psi\in \ker(\overline{\partial}_{E,m,q,\max}^{h,h})$.
\end{lemma}

\begin{proof}
In order to prove that $\psi\in \ker(\overline{\partial}^{h,h}_{E,m,q,\max})$ we have to show that $$\langle\psi,\overline{\partial}_{E,m,q}^{h,h,t}\phi\rangle_{L^2\Omega^{m,q}(A,E|_A,h|_A,\rho|_A)}=0$$ for each $\phi\in \Omega_c^{m,q+1}(A,E|_A)$.
Thanks to Lemma \ref{lemma3} we have
\begin{align}
 \nonumber \langle\psi,\overline{\partial}_{E,m,q}^{h,h,t}\phi\rangle_{L^2\Omega^{m,q}(A,E|_A,h,\rho)}&=\lim_{n\rightarrow\infty}\langle\psi_{s_n},\overline{\partial}_{E,m,q}^{g_{s_n},h,t}\phi\rangle_{L^2\Omega^{m,q}(A,E|_A,g_{s_n}|_A,\rho|_A)}\\
&= \nonumber \lim_{n\rightarrow\infty}\langle\overline{\partial}_{E,m,q,\max}^{g_{s_n},h}\psi_{s_n},\phi\rangle_{L^2\Omega^{m,q+1}(A,E|_A,g_{s_n}|_A,\rho|_A)}=0
\end{align}
as $\psi_{s_n}\in \ker(\overline{\partial}_{E,m,q,\max}^{g_{s_m},h})$.
\end{proof}

To state the next result we need some auxiliary notations. Let $s\in [0,1]$ and let us consider the orthogonal decomposition 
\begin{equation}
\label{ordec}
L^2\Omega^{m,q}(A,E|_A,g_s|_A,\rho|_A)=\ker(\overline{\partial}_{E,m,q,\max}^{g_s,h})\oplus \left(\ker(\overline{\partial}_{E,m,q,\max}^{g_s,h})\right)^{\bot}.
\end{equation}
 We denote by $\pi_{s}$ and $\tau_s$ the orthogonal projection on $\ker(\overline{\partial}_{E,m,q,\max}^{g_{s},h})$ and  $\left(\ker(\overline{\partial}_{E,m,q,\max}^{g_s,h})\right)^{\bot}$, respectively. 

\begin{lemma}
\label{lemmalemma}
Let $\eta\in (\ker(\overline{\partial}_{E,m,q,\max}^{h,h}))^{\bot}.$ Then given any sequence $\{s_n\}_{n\in \mathbb{N}}\subset (0,1]$ with  $s_n\rightarrow 0$ as $n\rightarrow \infty$ we have $$\tau_{s_n}(\Phi_{s_n}^{m,q}(\eta))\rightarrow \eta$$ strongly as $n\rightarrow \infty$. 
\end{lemma}
\begin{proof}
Let $\psi_{s_n}:=\pi_{s_n}(\Phi_{s_n}(\eta))$. Then $\psi_{s_n}\in \ker(\overline{\partial}_{E,m,q,\max}^{g_{s_n},h})$ and in order to prove the above lemma we have to show that $\psi_{s_n}\rightarrow 0$ strongly as $n\rightarrow \infty$, that is $$\lim_{n\rightarrow \infty}\|\psi_{s_n}\|_{L^2\Omega^{m,q}(A,E|_A,g_{s_n}|_A,\rho|_A)}=0.$$ Thanks to \eqref{decia} we know that $\{\psi_{s_n}\}$ is a bounded sequence as $$\|\psi_{s_n}\|_{L^2\Omega^{m,q}(A,E|_A,g_{s_n}|_A,\rho|_A)}\leq \|\Phi_{s_n}(\eta)\|_{L^2\Omega^{m,q}(A,E|_A,g_{s_n}|_A,\rho|_A)}\leq a\|\eta\|_{L^2\Omega^{m,q}(A,E|_A,h|_A,\rho|_A)}$$ for each $n\in \mathbb{N}$ and thus there exists a subsequence $\{\psi_{r_n}\}\subset \{\psi_{s_n}\}$ such that $\psi_{r_n}\rightarrow \psi$ weakly as $n\rightarrow \infty$ to some $\psi \in L^2\Omega^{m,q}(A,E|_A,h|_A,\rho|_A)$. We claim now that $\psi=0$ and that $\psi_{r_n}\rightarrow 0$ strongly as $n\rightarrow \infty$. Indeed we have 
$$
\begin{aligned}
\lim_{n\rightarrow \infty}\|\psi_{r_n}\|^2_{L^2\Omega^{m,q}(A,E|_A,g_{r_n}|_A,\rho|_A)}&=\lim_{n\rightarrow \infty}\langle\psi_{r_n},\pi_{r_n}(\Phi_{r_n}(\eta))\rangle_{L^2\Omega^{m,q}(A,E|_A,g_{r_n}|_A,\rho|_A)}\\
&=\lim_{n\rightarrow \infty}\langle\psi_{r_n},\Phi_{r_n}(\eta)\rangle_{L^2\Omega^{m,q}(A,E|_A,g_{r_n}|_A,\rho|_A)}=\langle\psi,\eta\rangle_{L^2\Omega^{m,q}(A,E|_A,h|_A,\rho|_A)}=0.
\end{aligned}$$
Note that the second-to-last equality above follows by the fact that $\psi_{r_n}\rightarrow \psi$ weakly and $\Phi_{r_n}(\eta)\rightarrow \eta$ strongly while the last equality follows by the fact that $\eta\in (\ker(\overline{\partial}_{E,m,q,\max}^{h,h}))^{\bot}$ and $\psi\in \ker(\overline{\partial}_{E,m,q,\max}^{h,h})$, see Lemma \ref{lemma4}.
Now if we fix an arbitrary subsequence $\{\psi_{s'_n}\}\subset \{\psi_{s_n}\}$ and we repeat the above argument with respect to $\{\psi_{s'_n}\}$ we find a subsequence $\{\psi_{r'_n}\}\subset\{\psi_{s'_n}\}$  such that $$\lim_{n\rightarrow \infty}\|\psi_{r'_n}\|_{L^2\Omega^{m,q}(A,E|_A,g_{r'_n}|_A,\rho|_A)}=0.$$ Summarizing every subsequence of $\{\psi_{s_n}\}$ has a further subsequence strongly convergent to $0$. We can thus conclude that $\psi_{s_n}\rightarrow 0$ strongly as $n\rightarrow \infty$.
\end{proof}

We now have all the ingredients to prove the next result.

\begin{teo}
\label{compact2}
Let $\{s_n\}_{n\in \mathbb{N}}\subset (0,1]$ be any sequence such that $s_n\rightarrow 0$ as $n\rightarrow \infty$ and let $$G_{\overline{\partial}_{E,m,q,\max}^{g_{s_n},h}}:L^2\Omega^{m,q+1}(A,E|_A,h|_A,\rho|_A)\rightarrow L^2\Omega^{m,q}(A,E|_A,g_{s_n}|_A,\rho|_A)$$ be the Green operator of $$\overline{\partial}_{E,m,q,\max}^{g_{s_n},h}:L^2\Omega^{m,q}(A,E|_A,g_{s_n}|_A,\rho|_A)\rightarrow L^2\Omega^{m,q+1}(A,E|_A,h|_A,\rho|_A).$$ If $$G_{\overline{\partial}_{E,m,q,\max}^{h,h}}:L^2\Omega^{m,q+1}(A,E|_A,h|_A,\rho|_A)\rightarrow L^2\Omega^{m,q}(A,E|_A,h|_A,\rho|_A)$$ is compact and  $\im(\overline{\partial}_{E,m,q,\max}^{h,h})=\im(\overline{\partial}_{E,m,q,\max}^{g_1,h})$
then $$G_{\overline{\partial}_{E,m,q,\max}^{g_{s_n},h}}\rightarrow G_{\overline{\partial}_{E,m,q,\max}^{h,h}}\ compactly\ as\ n\rightarrow \infty.$$
\end{teo}

\begin{proof}
First of all we observe that by the assumption $\im(\overline{\partial}_{E,m,q,\max}^{h,h})=\im(\overline{\partial}_{E,m,q,\max}^{g_1,h})$ we obtain immediately  $\im(\overline{\partial}_{E,m,q,\max}^{h,h})=\im(\overline{\partial}_{E,m,q,\max}^{g_s,h})$ for each $s\in (0,1]$ as $g_s$ and $g_{s'}$ are quasi-isometric for every $s,s'\in (0,1]$.
Now let $\{\alpha_{n}\}_{n\in \mathbb{N}}\subset L^2\Omega^{m,q+1}(A,E|_A,h|_A,\rho|_A)$ be a weakly convergent sequence to some $\alpha\in L^2\Omega^{m,q+1}(A,E|_A,h|_A,\rho|_A)$, that is $\alpha_n\rightharpoonup \alpha$ in $L^2\Omega^{m,q+1}(A,E|_A,h,\rho|_A)$ as $n\rightarrow \infty$. Let us define $\beta_n:=G_{\overline{\partial}_{E,m,q,\max}^{h,h}}\alpha_{n}$ and $\beta:=G_{\overline{\partial}_{E,m,q,\max}^{h,h}}\alpha$. Since we assumed that $$G_{\overline{\partial}_{E,m,q,\max}^{h,h}}:L^2\Omega^{m,q+1}(A,E|_A,h|_A,\rho)\rightarrow L^2\Omega^{m,q}(A,E|_A,h|_A,\rho|_A)$$ is compact we know that $\beta_n\rightarrow \beta$ in $L^2\Omega^{m,q}(A,E|_A,h|_A,\rho|_A)$ as $n\rightarrow \infty$. Let us now define $\beta_{s_n}:=\Phi_{s_n}(\beta_n)\in L^2\Omega^{m,q}(A,E|_A,g_{s_n}|_A,\rho|_A)$. Thanks to Prop. \ref{Pp}   we know that $\beta_{s_n}\in \mathcal{D}(\overline{\partial}_{E,m,q,\max}^{g_{s_n},h})$ and $$\overline{\partial}_{E,m,q,\max}^{g_{s_n},h}\beta_{s_n}=\overline{\partial}_{E,m,q,\max}^{h,h}\beta_n=\alpha_{n,1}$$ with $\alpha_{n,1}$ the orthogonal projection of $\alpha_n$ on $\im(\overline{\partial}_{E,m,q,\max}^{h,h})$. Note that the above equalities and the fact that $\im(\overline{\partial}_{E,m,q,\max}^{h,h})=\im(\overline{\partial}_{E,m,q,\max}^{g_{s_n},h})$ for each $n\in \mathbb{N}$
imply $$\tau_{s_n}(\beta_{s_n})=G_{\overline{\partial}_{E,m,q,\max}^{g_{s_n},h}}\alpha_n$$ for each $n\in \mathbb{N}$, with $\tau_{s_n}$ defined in \eqref{ordec}.
We are in position to prove that $G_{\overline{\partial}_{E,m,q,\max}^{g_{s_n},h}}\alpha_n\rightarrow \beta$ strongly as $n\rightarrow \infty$. Thanks to Prop. \ref{relieve} and Cor. \ref{swim} this means that we have to show that 
$$\lim_{n\rightarrow \infty}\|\Phi_{s_n}(\beta)-G_{\overline{\partial}_{E,m,q,\max}^{g_{s_n},h}}\alpha_n\|_{L^2\Omega^{m,q}(A,E|_A,g_{s_n}|_A,\rho|_A)}=0.$$
We have
$$
\begin{aligned}
& \|\Phi_{s_n}(\beta)-G_{\overline{\partial}_{E,m,q,\max}^{g_{s_n},h}}\alpha_n\|_{L^2\Omega^{m,q}(A,E|_A,g_{s_n}|_A,\rho|_A)}=\|\Phi_{s_n}(\beta)-\tau_{s_n}(\beta_{s_n})\|_{L^2\Omega^{m,q}(A,E|_A,g_{s_n}|_A,\rho|_A)}\\
& = \|\Phi_{s_n}(\beta)-\tau_{s_n}(\Phi_{s_n}(\beta))+\tau_{s_n}(\Phi_{s_n}(\beta))-\tau_{s_n}(\beta_{s_n})\|_{L^2\Omega^{m,q}(A,E|_A,g_{s_n}|_A,\rho|_A)}\\
& \leq \|\Phi_{s_n}(\beta)-\tau_{s_n}(\Phi_{s_n}(\beta))\|_{L^2\Omega^{m,q}(A,E|_A,g_{s_n}|_A,\rho|_A)}+\|\tau_{s_n}(\Phi_{s_n}(\beta))-\tau_{s_n}(\beta_{s_n})\|_{L^2\Omega^{m,q}(A,E|_A,g_{s_n}|_A,\rho|_A)}\\
& \leq \|\Phi_{s_n}(\beta)-\tau_{s_n}(\Phi_{s_n}(\beta))\|_{L^2\Omega^{m,q}(A,E|_A,g_{s_n}|_A,\rho|_A)}+\|\Phi_{s_n}(\beta)-\Phi_{s_n}(\beta_n)\|_{L^2\Omega^{m,q}(A,E|_A,g_{s_n}|_A,\rho|_A)}\\
(\mathrm{by\ \eqref{decia}}) & \leq  \|\Phi_{s_n}(\beta)-\tau_{s_n}(\Phi_{s_n}(\beta))\|_{L^2\Omega^{m,q}(A,E|_A,g_{s_n}|_A,\rho|_A)}+a\|\beta-\beta_n\|_{L^2\Omega^{m,q}(A,E|_A,h|_A,\rho|_A)}.
\end{aligned}
$$
We have already seen above that $$\lim_{n\rightarrow \infty}\|\beta-\beta_n\|_{L^2\Omega^{m,q}(A,E|_A,h,\rho)}=0.$$ Moreover by applying Lemma \ref{lemmalemma} we know that $$\lim_{n\rightarrow \infty}\|\Phi_{s_n}(\beta)-\tau_{s_n}(\Phi_{s_n}(\beta))\|_{L^2\Omega^{m,q}(A,E|_A,g_{s_n}|_A,\rho)}=0.$$ Summarizing we proved that
$$\lim_{n\rightarrow \infty}\|\Phi_{s_n}(\beta)-G_{\overline{\partial}_{E,m,q,\max}^{g_{s_n},h}}\alpha_n\|_{L^2\Omega^{m,q}(A,E|_A,g_{s_n}|_A,\rho)}=0$$ and so we can conclude that $G_{\overline{\partial}_{E,m,q,\max}^{g_{s_n},h}}\rightarrow G_{\overline{\partial}_{E,m,q,\max}^{h,h}}$ compactly as $n\rightarrow \infty.$
\end{proof}

As in Th. \ref{compact2} we continue to assume that $\im(\overline{\partial}_{E,m,q,\max}^{h,h})=\im(\overline{\partial}_{E,m,q,\max}^{g_1,h})$. For each $s\in (0,1]$ let us consider the following orthogonal decomposition 
\begin{equation}
\label{orto}
L^2\Omega^{m,q}(A,E|_A,g_s|_A,\rho|_A)=\left(\ker(\overline{\partial}_{E,m,q,\max}^{g_s,h})\cap \ker(\overline{\partial}_{E,m,q-1}^{g_1,g_s,t})\right)\oplus \im(\overline{\partial}_{E,m,q-1}^{g_1,g_s})\oplus \im(\overline{\partial}_{E,m,q,\min}^{g_s,h,t})
\end{equation}
and let 
$$
\pi_{K,s}^{m,q}:L^2\Omega^{m,q}(A,E|_A,g_s|_A,\rho|_A)\rightarrow \ker(\overline{\partial}_{E,m,q,\max}^{g_s,h})\cap \ker(\overline{\partial}_{E,m,q-1}^{g_1,g_s,t})$$
and 
$$ \pi_{I,s}^{m,q}:L^2\Omega^{m,q}(A,E,g_s,\rho)\rightarrow \im(\overline{\partial}_{E,m,q-1}^{g_1,g_s})$$ be the orthogonal projections on $\ker(\overline{\partial}_{E,m,q,\max}^{g_s,h})\cap \ker(\overline{\partial}_{E,m,q-1}^{g_1,g_s,t})$ and  $\im(\overline{\partial}_{E,m,q-1}^{g_1,g_s})$ induced by \eqref{orto}.
For $s=0$ we consider the orthogonal decomposition 
\begin{equation}
\label{orto2}
L^2\Omega^{m,q}(A,E|_A,h|_A,\rho|_A)=\left(\ker(\overline{\partial}_{E,m,q,\max}^{h,h})\cap \ker(\overline{\partial}_{E,m,q-1,\min}^{g_1,h,t})\right)\oplus \im(\overline{\partial}_{E,m,q-1,\max}^{g_1,h})\oplus \im(\overline{\partial}_{E,m,q,\min}^{h,h,t})
\end{equation}
and the corresponding orthogonal projections
$$\pi_{K,0}^{m,q}:L^2\Omega^{m,q}(A,E|_A,h|_A,\rho|_A)\rightarrow \ker(\overline{\partial}_{E,m,q,\max}^{h,h})\cap \ker(\overline{\partial}_{E,m,q-1}^{g_1,h,t})$$ and $$ \pi_{I,0}:L^2\Omega^{m,q}(A,E|_A,h|_A,\rho|_A)\rightarrow \im(\overline{\partial}_{E,m,q-1}^{g_1,h}).$$

The last goal of this subsection  is to show that $\pi^{m,q}_{k,s_n}\rightarrow \pi^{m,q}_{k,0}$ compactly, as $n\rightarrow +\infty$.

\begin{lemma}
\label{lemmaker}
Let $\{s_n\}_{n\in \mathbb{N}}\subset (0,1]$ be a sequence with $s_n\rightarrow 0$ as $n\rightarrow \infty$. Let $\psi\in \ker(\overline{\partial}_{E,m,q,\max}^{h,h})\cap \ker(\overline{\partial}_{E,m,q-1,\min}^{g_1,h,t})$ be arbitrarily fixed. Then $$\pi_{K,s_n}^{m,q}(\Phi_{s_n}^{m,q}(\psi))\rightarrow \psi$$ strongly as $n\rightarrow \infty$.
\end{lemma} 

\begin{proof}
Thanks to Cor. \ref{swim} we know that $\Phi_{s_n}^{m,q}(\psi)\rightarrow \psi$ strongly as $n\rightarrow \infty$. Thus we need to show that both $\pi_{I,s_n}^{m,q}(\Phi_{s_n}^{m,q}(\psi))$ and 
$\Phi_{s_n}^{m,q}(\psi)-\pi^{m,q}_{K,s_n}(\Phi_{s_n}^{m,q}(\psi))-\pi_{I,s_n}^{m,q}(\Phi_{s_n}^{m,q}(\psi))$ converge strongly to $0$ as $n\rightarrow \infty.$ By Prop. \ref{Pp} we know that $\Phi_{s_n}^{m,q}(\psi)\in \ker(\overline{\partial}_{E,m,q,\max}^{g_{s_n,h}})$ for each $n$. Hence, $$\Phi_{s_n}^{m,q}(\psi)-\pi^{m,q}_{K,s_n}(\Phi_{s_n}^{m,q}(\psi))-\pi_{I,s_n}^{m,q}(\Phi_{s_n}^{m,q}(\psi))=0$$ for each $n$ and consequently $$\Phi_{s_n}^{m,q}(\psi)-\pi^{m,q}_{K,s_n}(\Phi_{s_n}^{m,q}(\psi))-\pi_{I,s_n}^{m,q}(\Phi_{s_n}^{m,q}(\psi))\rightarrow 0$$ strongly as $n\rightarrow \infty$. Concerning $\pi_{I,s_n}^{m,q}(\Phi_{s_n}^{m,q}(\psi))$ we know that
$$
\begin{aligned}
\|\pi_{I,s_n}^{m,q}(\Phi_{s_n}^{m,q}(\psi))\|_{L^2\Omega^{M,q}(A,E|_A,g_{s_n}|_A,\rho|_A)}&\leq \|\Phi_{s_n}^{m,q}(\psi)\|_{L^2\Omega^{M,q}(A,E|_A,g_{s_n}|_A,\rho|_A)}\\
& \leq a\|\psi\|_{L^2\Omega^{M,q}(A,E_A,h|_A,\rho|_A)}
\end{aligned}
$$ see \eqref{decia}.
Thus $\{\pi_{I,s_n}^{m,q}(\Phi_{s_n}^{m,q}(\psi))\}_{n\in \mathbb{N}}$ is a bounded sequence and therefore by Prop. \ref{bounded} there exists a subsequence $\{\pi_{I,s'_n}^{m,q}(\Phi_{s'_n}^{m,q}(\psi))\}_{n\in \mathbb{N}}\subset \{\pi_{I,s_n}^{m,q}(\Phi_{s_n}^{m,q}(\psi))\}_{n\in \mathbb{N}}$ and an element $\psi_0\in L^2\Omega^{m,q}(A,E|_A,h|_A,\rho|_A)$ such that $$\pi_{I,s'_n}^{m,q}(\Phi_{s'_n}^{m,q}(\psi))\rightarrow \psi_0$$ weakly as $n\rightarrow \infty$. Since $\pi_{I,s'_n}^{m,q}(\Phi_{s'_n}^{m,q}(\psi))\in \im(\overline{\partial}_{E,m,q-1}^{g_1,g_{s'_n}})$ and $\pi_{I,s'_n}^{m,q}(\Phi_{s'_n}^{m,q}(\psi))\rightarrow \psi_0$ weakly as $n\rightarrow \infty$ we can argue as in the proof of  Th. \ref{compact1} to conclude that $\psi_0\in \im(\overline{\partial}_{E,m,q,\max}^{g_1,h})$.  We want to show that $\psi_0=0$ and that $\pi_{I,s'_n}^{m,q}(\Phi_{s'_n}^{m,q}(\psi))\rightarrow 0$ strongly as $n\rightarrow \infty$. Let $\beta\in \im(\overline{\partial}_{E,m,q,\max}^{g_1,h})$ be arbitrarily fixed. We have
$$
\begin{aligned}
\langle\psi_0,\beta\rangle_{L^2\Omega^{m,q}(A,E|_A,h|_A,\rho|_A)}&=\lim_{n\rightarrow \infty}\langle\pi_{I,s'_n}^{m,q}(\Phi_{s'_n}^{m,q}(\psi)),\Phi_{s'_n}^{m,q}(\beta)\rangle_{L^2\Omega^{m,q}(A,E|_A,g_{s'_n}|_A,\rho|_A)}\\
&=\lim_{n\rightarrow \infty}\langle\pi_{I,s'_n}^{m,q}(\Phi_{s'_n}^{m,q}(\psi))+\pi^{m,q}_{K,s'_n}(\Phi_{s'_n}^{m,q}(\psi)),\Phi_{s'_n}(\beta)\rangle_{L^2\Omega^{m,q}(A,E|_A,g_{s'_n}|_A,\rho|_A)}\\
&= \lim_{s\rightarrow \infty}\langle\Phi_{s'_n}^{m,q}(\psi),\Phi_{s'_n}^{m,q}(\beta)\rangle_{L^2\Omega^{m,q}(A,E|_A,g_{s'_n}|_A,\rho|_A)}\\
&=\langle\psi,\beta\rangle_{L^2\Omega^{m,q}(A,E|_A,h|_A,\rho|_A)}=0.
\end{aligned}
$$
Note that the first equality above follows by the fact that $\pi_{I,s'_n}^{m,q}(\Phi_{s'_n}^{m,q}(\psi))\rightarrow \psi_0$ weakly and $\Phi_{s'_n}^{m,q}(\beta)\rightarrow \beta$ strongly. The second equality is a consequence of the fact that $\Phi_{s}^{m,q}(\beta)\in \im(\overline{\partial}_{E,m,q}^{g_1,g_s})$ as $\im(\overline{\partial}_{E,m,q,\max}^{g_1,h})\subset \im(\overline{\partial}_{E,m,q}^{g_1,g_s})$ for each $0<s\leq 1$. Finally the third equality follows by the fact that 
$\Phi_{s'_n}^{m,q}(\psi)-\pi^{m,q}_{K,s'_n}(\Phi_{s'_n}^{m,q}(\psi))-\pi_{I,s'_n}^{m,q}(\Phi_{s'_n}^{m,q}(\psi))=0$ for each $n$. We can thus conclude that $\psi_0=0$ as $\langle\psi_0,\beta\rangle_{L^2\Omega^{m,q}(A,E|_A,h|_A,\rho|_A)}=0$ for any  arbitrarily fixed $\beta\in \im(\overline{\partial}_{E,m,q,\max}^{g_1,h})$. We are left to show that $\pi_{I,s'_n}^{m,q}(\Phi_{s'_n}(\psi))\rightarrow 0$ strongly as $n\rightarrow \infty$. To this aim we have
$$
\begin{aligned}
&\lim_{n\rightarrow \infty}\langle\pi_{I,s'_n}^{m,q}(\Phi_{s'_n}^{m,q}(\psi)),\pi_{I,s'_n}^{m,q}(\Phi_{s'_n}^{m,q}(\psi))\rangle_{L^2\Omega^{m,q}(A,E|_A,g_{s'_n}|_A,\rho|_A)}=\\
&\lim_{n\rightarrow \infty}\langle\pi_{I,s'_n}^{m,q}(\Phi_{s'_n}^{m,q}(\psi)),\pi^{m,q}_{K,s'_n}(\Phi_{s'_n}^{m,q}(\psi))+\pi_{I,s'_n}^{m,q}(\Phi_{s'_n}^{m,q}(\psi))\rangle_{L^2\Omega^{m,q}(A,E|_A,g_{s'_n}|_A,\rho|_A)}=\\
& \lim_{n\rightarrow \infty}\langle\pi_{I,s'_n}^{m,q}(\Phi_{s'_n}^{m,q}(\psi)),\Phi_{s'_n}^{m,q}(\psi)\rangle_{L^2\Omega^{m,q}(A,E|_A,g_{s'_n}|_A,\rho|_A)}=\langle\psi_0,\psi\rangle_{L^2\Omega^{m,q}(A,E|_A,h|_A,\rho|_A)}=0.
\end{aligned}
$$
In conclusion  $\pi_{I,s'_n}^{m,q}(\Phi_{s'_n}^{m,q}(\psi))\rightarrow 0$ strongly as $n\rightarrow \infty$. Now if we fix an arbitrary subsequence $\{r_n\}_{n\in \mathbb{N}}\subset \{s_n\}_{n\in \mathbb{N}}$ and we repeat the above argument we conclude that there exists a subsequence $\{r'_n\}_{n\in\mathbb{R}}\subset \{r_n\}_{n\in \mathbb{N}}$ such that  
$\pi_{I,r'_n}(\Phi_{r'_n}^{m,q}(\psi))\rightarrow 0$ strongly as $n\rightarrow \infty$. We can thus conclude that $\pi_{I,s_n}^{m,q}(\Phi_{s_n}^{m,q}(\psi))\rightarrow 0$ strongly as $n\rightarrow \infty$ and hence that  $\pi^{m,q}_{K,s_n}(\Phi_{s_n}^{m,q}(\psi))\rightarrow \psi$ strongly as $n\rightarrow \infty$.
\end{proof}

\begin{lemma}
\label{lemmaker2}
If $\dim(H_{\overline{\partial}_E}^{m,q}(M,E))=\dim(H_{2,\overline{\partial}_{\max}}^{m,q}(A,E|_A,h|_A,\rho|_A))$, then for  each $s\in (0,1]$ the linear map $$\pi^{m,q}_{K,s}\circ\Phi_s^{m,q}:\ker(\overline{\partial}_{E,m,q-1,\min}^{g_1,h,t})\cap\ker(\overline{\partial}_{E,m,q,\max}^{h,h})\longrightarrow \ker(\overline{\partial}_{E,m,q-1}^{g_1,g_s,t})\cap\ker(\overline{\partial}_{E,m,q,\max}^{g_s,h})$$
is an isomorphism.
\end{lemma}
\begin{proof}
By assumption we know that $g_1$ and $g_s$ are quasi-isometric for each $s\in (0,1]$ and that $\im(\overline{\partial}_{E,m,q-1,\max}^{g_1,h})=\im(\overline{\partial}_{E,m,q-1,\max}^{h,h})$ in $L^2\Omega^{m,q}(A,E|_A,h|_A,\rho|_A)$. This tells us that for each $s\in (0,1]$ we have $\im(\overline{\partial}_{E,m,q-1,\max}^{g_1,h})=\im(\overline{\partial}_{E,m,q-1,\max}^{g_s,h})$ in $L^2\Omega^{m,q}(A,E|_A,h|_A,\rho|_A)$ and thus $$\Phi_s^{m,q}:L^2\Omega^{m,q}(A,E|_A,h|_A,\rho|_A)\rightarrow L^2\Omega^{m,q}(A,E|_A,g_s|_A,\rho|_A)$$ induces an injective linear map between the $L^2$-$\overline{\partial}$-cohomology groups
\begin{equation}
\label{Sinner}
\Phi_{s}^{m,q}:\ker(\overline{\partial}_{E,m,q,\max}^{h,h})/\im(\overline{\partial}_{E,m,q-1,\max}^{g_1,h})\longrightarrow \ker(\overline{\partial}_{E,m,q,\max}^{g_s,h})/\im(\overline{\partial}_{E,m,q-1}^{g_1,g_s}).
\end{equation}
Note that we have 
$$\begin{aligned}
\dim(H_{2,\overline{\partial}_{\max}}^{m,q}(A,E|_A,h|_A,\rho|_A))&= \dim(\ker(\overline{\partial}_{E,m,q,\max}^{h,h})/\im(\overline{\partial}_{E,m,q-1,\max}^{g_1,h}))\\
&\leq \dim(\ker(\overline{\partial}_{E,m,q,\max}^{g_s,h})/\im(\overline{\partial}_{E,m,q-1}^{g_1,g_s}))\\
&\leq \dim(\ker(\overline{\partial}_{E,m,q}^{g_1,g_1})/\im(\overline{\partial}_{E,m,q-1}^{g_1,g_1}))\\
&=\dim(H_{\overline{\partial}_E}^{m,q}(M,E))\\
&=\dim(H_{\overline{\partial}_{\max}}^{m,q}(A,E|_A,h|_A,\rho|_A)).
\end{aligned}$$
We can thus conclude that \eqref{Sinner} is an isomorphism for each $0<s\leq 1$. Thanks to \eqref{orto} and \eqref{orto2} we have isomorphisms
\begin{equation}
\label{SDS}
\ker(\overline{\partial}_{E,m,q-1,\min}^{g_1,h,t})\cap\ker(\overline{\partial}_{E,m,q,\max}^{h,h})\cong\ker(\overline{\partial}_{E,m,q,\max}^{h,h})/\im(\overline{\partial}_{E,m,q-1,\max}^{g_1,h})
\end{equation}
and 
\begin{equation}
\label{SDSD}
\ker(\overline{\partial}_{E,m,q-1}^{g_1,g_s,t})\cap\ker(\overline{\partial}_{E,m,q,\max}^{g_s,h})\cong\ker(\overline{\partial}_{E,m,q,\max}^{g_s,h})/\im(\overline{\partial}_{E,m,q-1}^{g_1,g_s}).
\end{equation}
It is easy to check that if $\alpha\in \ker(\overline{\partial}_{E,m,q-1,\min}^{g_1,h,t})\cap\ker(\overline{\partial}_{E,m,q,\max}^{h,h})$ is the unique representative in $\ker(\overline{\partial}_{E,m,q-1,\min}^{g_1,h,t})\cap\ker(\overline{\partial}_{E,m,q,\max}^{h,h})$ of $[\alpha]\in\ker(\overline{\partial}_{E,m,q,\max}^{h,h})/\im(\overline{\partial}_{E,m,q-1,\max}^{g_1,h})$ then $\pi^{m,q}_{K,s}(\Phi_s^{m,q}(\alpha))$ is the unique representative in $\ker(\overline{\partial}_{E,m,q-1}^{g_1,g_s,t})\cap\ker(\overline{\partial}_{E,m,q,\max}^{g_s,h})$ of $[\Phi_s^{m,q}(\alpha)]\in \ker(\overline{\partial}_{E,m,q,\max}^{g_s,h})/\im(\overline{\partial}_{E,m,q-1}^{g_1,g_s})$. The conclusion now follows immediately by \eqref{Sinner}, \eqref{SDS} and \eqref{SDSD}.
\end{proof}

\begin{teo}
\label{projker}
In the setting of Th. \ref{compact2} assume in addition that $\dim(H_{\overline{\partial}}^{m,q}(M,E))=\dim(H_{2,\overline{\partial}_{\max}}^{m,q}(A,E,h,\rho))$. Let $\{s_n\}_{n\in \mathbb{N}}\subset (0,1]$ be a sequence with $s_n\rightarrow 0$ as $n\rightarrow \infty$. Then $$\pi^{m,q}_{K,s}\rightarrow \pi_{K,0}^{m,q}$$ compactly as $n\rightarrow \infty$.
\end{teo}

\begin{proof}
Let $\{\psi_1,...,\psi_{\ell}\}$ be an orthonormal basis of $\ker(\overline{\partial}_{E,m,q-1,\min}^{g_1,h,t})\cap\ker(\overline{\partial}_{E,m,q,\max}^{h,h})$ and for each $j\in \{1,...,\ell\}$ let $\psi_{j,s_{n}}:=\pi^{m,q}_{K,s_n}(\Phi_{s_n}^{m,q}(\psi_j)).$ Then by Lemmas \ref{lemmaker} and \ref{lemmaker2} we know that $\{\psi_{1,s_n},....,\psi_{\ell,s_n}\}$ is a basis for $\ker(\overline{\partial}_{E,m,q-1}^{g_1,g_{s_n},t})\cap\ker(\overline{\partial}_{E,m,q,\max}^{g_{s_n},h})$ and $\psi_{j,s_n}\rightarrow \psi_j$ strongly as $n\rightarrow \infty$.  Let $\{\chi_{1,s_n,}...,\chi_{\ell,s_n}\}$ be the basis of $\ker(\overline{\partial}_{E,m,q-1}^{g_1,g_{s_n},t})\cap\ker(\overline{\partial}_{E,m,q,\max}^{g_{s_n},h})$ made by pairwise orthogonal elements obtained by applying the Gram-Schmidt procedure to the basis $\{\psi_{1,s_n},....,\psi_{\ell,s_n}\}$. Explicit we have 
$$
\begin{aligned}
&\chi_{1,s_n}=\psi_{1,s_n}\\
& \chi_{2,s_n}=\psi_{2,s_n}-\mathrm{pr}_{\chi_{1,s_n}}(\psi_{2,s_n})\\
&......\\
& \chi_{j,s_n}=\psi_{j,s_n}-\sum_{k=1}^{j-1}\mathrm{pr}_{\chi_{k,s_n}}(\psi_{j,s_n})\\
&......\\
& \chi_{\ell,s_n}=\psi_{\ell,s_n}-\sum_{k=1}^{\ell-1}\mathrm{pr}_{\chi_{k,s_n}}(\psi_{\ell,s_n})
\end{aligned}
$$ 
where 
\begin{equation}
\label{GrS}
\mathrm{pr}_{\chi_{k,s_n}}(\psi_{j,s_n}):=\frac{\langle\psi_{j,s_n},\chi_{k,s_n}\rangle_{L^2\Omega^{m,q}(A,E|_A,g_{s_n}|_A,\rho|_A)}}{\langle\chi_{k,s_n},\chi_{k,s_n}\rangle_{L^2\Omega^{m,q}(A,E|_A,g_{s_n}|_A,\rho|_A)}}\chi_{k,s_n}
\end{equation}
for each $j\in \{1,...,\ell\}$ and $k\in \{1,...,j-1\}$.
Looking at \eqref{GrS} and arguing by induction it is easy to check that $$\lim_{n\rightarrow\infty}\mathrm{pr}_{\chi_{k,s_n}}(\psi_{j,s_n})=0$$ strongly and consequently $$\chi_{j,s_n}\rightarrow \psi_{j}$$ strongly as $n\rightarrow \infty$ for each $j=1,...,\ell$. In particular $$\|\chi_{j,s_n}\|_{L^2\Omega^{m,q}(A,E|_A,g_{s_n}|_A,\rho|_A)}\rightarrow 1=\|\psi_{j}\|_{L^2\Omega^{m,q}(A,E|_A,h|_A,\rho|_A)}$$ as $n\rightarrow \infty$ for each $j=1,...,\ell$. Therefore, by defining $$\varphi_{j,s_n}:=\chi_{j,s_n}/\|\chi_{j,s_n}\|_{L^2\Omega^{m,q}(A,E|_A,g_{s_n}|_A,\rho|_A)},$$ we obtain an orthonormal basis of $\ker(\overline{\partial}_{E,m,q-1}^{g_1,g_{s_n},t})\cap\ker(\overline{\partial}_{E,m,q,\max}^{g_{s_n},h})$ made by $\{\varphi_{1,s_n},...,\varphi_{\ell,s_n}\}$ such that $\varphi_{j,s_n}\rightarrow \psi_{j}$ strongly as $n\rightarrow \infty$ for each $j=1,...,\ell$. Let now $\{\beta_{s_n}\}_{n\in \mathbb{N}}$, $\beta_{s_n}\in L^2\Omega^{m,q}(A,E|_A,g_{s_n}|_A,\rho|_A)$, be a weakly convergent sequence to some $\beta\in L^2\Omega^{m,q}(A,E|_A,h|_A,\rho|_A)$ as $n\rightarrow \infty$. We want to show that $\pi^{m,q}_{K,s_n}\beta_{s_n}\rightarrow \pi_{K,0}^{m,q}\beta$ strongly as $n\rightarrow \infty$.
We have $$\pi^{m,q}_{K,s_n}\beta_{s_n}=\sum_{j=1}^{\ell}\langle\varphi_{j,s_n},\beta_{s_n}\rangle_{L^2\Omega^{m,q}(A,E|_A,g_{s_n}|_A,\rho|_A)}\varphi_{j,s_n}.$$
Thanks to the first part of the proof we know that $\varphi_{j,s_n}\rightarrow \psi_j$ strongly as $n\rightarrow \infty$. Since $\beta_{s_n}$ converges weakly to $\beta$ as $n\rightarrow \infty$ we get that $$\langle\varphi_{j,s_n},\beta_{s_n}\rangle_{L^2\Omega^{m,q}(A,E|_A,g_{s_n}|_A,\rho|_A)}\rightarrow \langle\psi_{j},\beta\rangle_{L^2\Omega^{m,q}(A,E|_A,h|_A,\rho|_A)}$$ as $n\rightarrow \infty.$ Therefore we can conclude that $$\sum_{j=1}^{\ell}\langle\varphi_{j,s_n},\beta_{s_n}\rangle_{L^2\Omega^{m,q}(A,E|_A,g_{s_n}|_A,\rho|_A)}\varphi_{j,s_n}\rightarrow \sum_{j=1}^{\ell}\langle\psi_j,\beta\rangle_{L^2\Omega^{m,q}(A,E|_A,h|_A,\rho|_A)}\psi_j$$
strongly as $j\rightarrow \infty$, that is $$\pi^{m,q}_{K,s_n}\beta_{s_n}\rightarrow \pi_{K,0}^{m,q}\beta$$ strongly as $n\rightarrow \infty$.
\end{proof}

\subsection{From compact convergence to convergence in norm operator}

As in Lemma \ref{lemma3} let $p:A\times [0,1]\rightarrow A$ be the left projection and let  $$S_s^{m,q}\in C^{\infty}(A\times [0,1], \mathrm{End}(p^*\Lambda^{m,q}(A)\otimes p^*E))$$ be defined as  $S^{m,q}_s:=\det(G^{1,0}_{\mathbb{C},s})\otimes G^{0,q}_{\mathbb{C},s}\otimes \id$. We recall that $S^{m,q}_s$ is the family of endomorphisms of $p^*\Lambda^{m,q}(A)\otimes p^*E$ such that $g^*_{m,q,\rho,s}(\bullet,\bullet)=g^*_{m,q,\rho,1}(S^{m,q}_s\bullet,\bullet)$. It is not difficult to see that there exists 
$\Gamma_s^{m,q}\in C(A\times [0,1], \mathrm{End}(p^*\Lambda^{m,q}(A)\otimes p^*E))$, that is, a continuous section of $\mathrm{End}(p^*\Lambda^{m,q}(A)\otimes p^*E)\rightarrow A\times [0,1]$, such that 
\begin{enumerate}
\item $(\Gamma_s^{m,q})^2=S^{m,q}_s$;
\item $g^*_{1,m,q,\rho}(\Gamma_s^{m,q}\bullet,\bullet)=g^*_{1,m,q,\rho}(\bullet,\Gamma_s^{m,q}\bullet)$, that is $\Gamma_s^{m,q}$ is fiberwise self-adjoint w.r.t. $g^*_{1,m,q,\rho}$;
\item $g^*_{1,m,q,\rho}(\Gamma_s^{m,q}\bullet,\bullet)>0$ whenever $\bullet\neq 0$, that is $\Gamma_s^{m,q}$ is positive definite w.r.t. $g^*_{1,m,q,\rho}$;
\end{enumerate}
see e.g. \cite[Problem 2-E, p. 24]{Milnor}. Note that $g^*_{1,m,q,\rho}(\Gamma_s^{m,q}\bullet,\Gamma_s^{m,q}\bullet)=g^*_{s,m,q,\rho}(\bullet,\bullet)$. In other words $\Gamma_s^{m,q}$ is a fiberwise isometry between $p^*\Lambda^{m,q}(A)\otimes p^*E$ endowed with $g^*_{s,m,q,\rho}$ and $p^*\Lambda^{m,q}(A)\otimes p^*E$ endowed with $g^*_{m,q,\rho,1}$. Now let us define $\Psi^{m,q}_s\in C(A\times [0,1], \mathrm{End}(p^*\Lambda^{m,q}(A)\otimes p^*E))$ as 
\begin{equation}
\label{isoiso}
\Psi^{m,q}_s:=(\det(F^{1,0}_{\mathbb{C},s}))^{\frac{1}{2}}\otimes \Gamma_s^{m,q}.
\end{equation}
Let us check that $$\Psi_s^{m,q}:L^2\Omega^{m,q}(A,E|_A,g_s|_A,\rho|_A)\rightarrow L^2\Omega^{m,q}(A,E|_A,g_1|_A,\rho|_A)$$ is an isometry for each $s\in [0,1]$.
Let $\eta,\omega\in \Omega_c^{m,q}(A,E|_A)$. We have 
$$
\begin{aligned}
\langle\Psi_s^{m,q}\eta,\Psi_s^{m,q}\omega\rangle_{L^2\Omega^{m,q}(A,E|_A,g_1|_A,\rho|_A)}&=\int_Ag^*_{1,m,q,\rho}(\Psi_s^{m,q}\eta,\Psi_s^{m,q}\omega)\dvol_{g_1}\\
&=\int_Ag^*_{1,m,q,\rho}(\Gamma_s^{m,q}\eta,\Gamma_s^{m,q}\omega)\det(F_{\mathbb{C},s}^{1,0})\dvol_{g_1}\\
&=\int_Ag^*_{1,m,q,\rho}(S_s^{m,q}\eta,\omega)(\det(F_{s}))^{\frac{1}{2}}\dvol_{g_1}\\&=\int_Ag^*_{s,m,q,\rho}(\eta,\omega)\dvol_{g_s}\\
&=\langle\eta,\omega\rangle_{L^2\Omega^{m,q}(A,E|_A,g_s|_A,\rho|_A)}.
\end{aligned}
$$

We now prove various properties concerning $\Psi_s^{m,q}$.

\begin{lemma}
\label{lemma5}
Given any $\eta\in L^2\Omega^{m,q}(A,E|_A,h|_A,\rho|_A)$  it holds $$\lim_{s\rightarrow 0}\|\Psi_s^{m,q}(\Phi_s^{m,q}(\eta))-\Psi_0^{m,q}(\eta)\|_{L^2\Omega^{m,q}(A,E|_A,g_1|_A,\rho|_A)}=0.$$ 
\end{lemma}
\begin{proof}
First we deal with the  case $\eta\in \Omega^{m,q}_c(A,E|_A)$. In this case 
$$\|\Psi_s^{m,q}(\Phi_s^{m,q}(\eta))-\Psi_0^{m,q}(\eta)\|^2_{L^2\Omega^{m,q}(A,E|_A,g_1|_A,\rho|_A)}=\int_Ag_{1,m,q,\rho}^*(\Psi_s^{m,q}\eta-\Psi_0^{m,q}\eta,\Psi_s^{m,q}\eta-\Psi_0^{m,q}\eta)\dvol_{g_1}.$$ Since $\eta\in \Omega_c^{m,q}(A,E|_A)$, $\Psi^{m,q}_s\in C(A\times [0,1], p^*\mathrm{End}(p^*\Lambda^{m,q}(A)\otimes p^*E))$ and $\mathrm{vol}_{g_1}(A)<\infty$, we can apply the Lebesgue dominated convergence theorem:
$$
\begin{aligned}
\lim_{s\rightarrow 0}\|\Psi_s^{m,q}(\Phi_s^{m,q}(\eta))-\Psi_0^{m,q}(\eta)\|^2_{L^2\Omega^{m,q}(A,E|_A,g_1|_A,\rho|_A)}&=\lim_{s\rightarrow 0}\int_Ag_{1,m,q,\rho}^*(\Psi_s^{m,q}\eta-\Psi_0^{m,q}\eta,\Psi_s^{m,q}\eta-\Psi_0^{m,q}\eta)\dvol_{g_1}\\
&=\int_A\lim_{s\rightarrow 0}g_{1,m,q,\rho}^*(\Psi_s^{m,q}\eta-\Psi_0^{m,q}\eta,\Psi_s^{m,q}\eta-\Psi_0^{m,q}\eta)\dvol_{g_1}\\
&=0.
\end{aligned}
$$
Now we consider the general case $\eta\in L^2\Omega^{m,q}(A,E|_A,h|_A,\rho|_A)$. Let $\epsilon >0$ be arbitrarily fixed and let $\varphi\in \Omega^{m,q}_c(A,E|_A)$ be such that $\|\eta-\varphi\|_{L^2\Omega^{m,q}(A,E|_A,h|_A,\rho|_A)}<\epsilon$. 
Let $a$ be the positive constant appearing in \eqref{decia}. Since $\Psi_s^{m,q}:L^2\Omega^{m,q}(A,E|_A,\rho|_A,g_s|_A)\rightarrow L^2\Omega^{m,q}(A,E|_A,\rho|_A,g_1|_A)$ is an isometry for each $s\in [0,1]$ we have 
$$
\begin{aligned}
&\|\Psi_s^{m,q}(\Phi_s^{m,q}(\eta))-\Psi_0^{m,q}(\eta)\|^2_{L^2\Omega^{m,q}(A,E|_A,g_1|_A,\rho|_A)}=\\
&\|\Psi_s^{m,q}(\Phi_s^{m,q}(\eta))-\Psi_s^{m,q}(\Phi_s^{m,q}(\varphi))+\Psi_s^{m,q}(\Phi_s^{m,q}(\varphi))-\Psi_0^{m,q}(\eta)\|^2_{L^2\Omega^{m,q}(A,E|_A,g_1|_A,\rho|_A)}\leq\\
& \|\Psi_s^{m,q}(\Phi_s^{m,q}(\eta))-\Psi_s^{m,q}(\Phi_s^{m,q}(\varphi))\|_{L^2\Omega^{m,q}(A,E|_A,g_1|_A,\rho|_A)}+\|\Psi_s^{m,q}(\Phi_s^{m,q}(\varphi))-\Psi_0^{m,q}(\eta)\|^2_{L^2\Omega^{m,q}(A,E|_A,g_1|_A,\rho|_A)}=\\
&\|\eta-\varphi\|_{L^2\Omega^{m,q}(A,E|_A,g_s|_A,\rho|_A)}+\|\Psi_s^{m,q}(\Phi_s^{m,q}(\varphi))-\Psi_0^{m,q}(\varphi)+\Psi_0^{m,q}(\varphi)-\Psi_0^{m,q}(\eta)\|^2_{L^2\Omega^{m,q}(A,E|_A,g_1|_A,\rho|_A)}\leq\\
& a\|\eta-\varphi\|_{L^2\Omega^{m,q}(A,E|_A,h|_A,\rho|_A)}+\|\Psi_s^{m,q}(\Phi_s^{m,q}(\varphi))-\Psi_0^{m,q}(\varphi)\|_{L^2\Omega^{m,q}(A,E|_A,g_1|_A,\rho|_A)}+\\
&\|\Psi_0^{m,q}(\varphi)-\Psi_0^{m,q}(\eta)\|^2_{L^2\Omega^{m,q}(A,E|_A,g_1|_A,\rho|_A)}\leq\\
& a\epsilon+\|\Psi_s^{m,q}(\Phi_s^{m,q}(\varphi))-\Psi_0^{m,q}(\varphi)\|_{L^2\Omega^{m,q}(A,E|_A,g_1|_A,\rho|_A)}+\epsilon.
\end{aligned}
$$
Since $\phi\in \Omega_c^{m,q}(A,E|_A)$ we can conclude as above that $$\lim_{s\rightarrow 0}\|\Psi_s^{m,q}(\Phi_s^{m,q}(\varphi))-\Psi_0^{m,q}(\varphi)\|_{L^2\Omega^{m,q}(A,E|_A,g_1|_A,\rho|_A)}=0$$ which in turn gives us 
$$\limsup_{s\rightarrow 0}\|\Psi_s^{m,q}(\Phi_s^{m,q}(\eta))-\Psi_0^{m,q}(\eta)\|^2_{L^2\Omega^{m,q}(A,E|_A,g_1|_A,\rho|_A)}\leq (a+1)\epsilon.$$
Since $\epsilon$ is arbitrarily fixed we can conclude that $$\lim_{s\rightarrow 0}\|\Psi_s^{m,q}(\Phi_s^{m,q}(\eta))-\Psi_0^{m,q}(\eta)\|^2_{L^2\Omega^{m,q}(A,E|_A,g_1|_A,\rho|_A)}=0$$ as desired.
\end{proof}

\begin{lemma}
\label{lemma6}
Let $\{s_n\}_{n\in \mathbb{N}}\subset (0,1]$ be a sequence such that $s_n\rightarrow 0$ as $n\rightarrow \infty$. Let $\eta\in L^2\Omega^{m,q}(A,E|_A,h|_A,\rho|_A)$ and let $\{\eta_{s_n}\}_{n\in \mathbb{N}}$ be a sequence such that $\eta_{s_n}\in L^2\Omega^{m,q}(A,E|_A,g_{s_n}|_A,\rho|_A)$. Then  $\eta_{s_n}\rightarrow \eta$ strongly as $n\rightarrow \infty$ if and only if  
\begin{equation}
\label{lim}
\lim_{n\rightarrow \infty}\|\Psi_{s_n}^{m,q}\eta_{s_n}-\Psi_0^{m,q}\eta\|_{L^2\Omega^{m,q}(A,E|_A,g_1|_A,\rho|_A)}=0.
\end{equation}
\end{lemma}

\begin{proof}
Let us assume that $\eta_{s_n}\rightarrow \eta$ strongly as $n\rightarrow \infty$.
We have 
$$
\begin{aligned}
&\|\Psi_{s_n}^{m,q}\eta_{s_n}-\Psi_0^{m,q}\eta\|_{L^2\Omega^{m,q}(A,E|_A,g_1|_A,\rho|_A)}=\\
&\|\Psi_{s_n}^{m,q}\eta_{s_n}-\Psi_{s_n}^{m,q}(\Phi^{m,q}_{s_n}(\eta))+\Psi_{s_n}^{m,q}(\Phi_{s_n}^{m,q}(\eta))-\Psi_0^{m,q}\eta\|_{L^2\Omega^{m,q}(A,E|_A,g_1|_A,\rho|_A)}\leq\\
& \|\Psi_{s_n}^{m,q}\eta_{s_n}-\Psi_{s_n}^{m,q}(\Phi^{m,q}_{s_n}(\eta))\|_{L^2\Omega^{m,q}(A,E|_A,g_1|_A,\rho|_A)}+\|\Psi_{s_n}^{m,q}(\Phi_{s_n}^{m,q}(\eta))-\Psi_0^{m,q}\eta\|_{L^2\Omega^{m,q}(A,E|_A,g_1|_A,\rho|_A)}=\\
&\|\eta_{s_n}-\Phi^{m,q}_{s_n}(\eta)\|_{L^2\Omega^{m,q}(A,E|_A,g_{s_n}|_A,\rho)}+\|\Psi_{s_n}^{m,q}(\Phi_{s_n}^{m,q}(\eta))-\Psi_0^{m,q}\eta\|_{L^2\Omega^{m,q}(A,E|_A,g_1|_A,\rho|_A)}.
\end{aligned}
$$
As $\eta_{s_n}\rightarrow \eta$ strongly as $n\rightarrow \infty$ we know that $$\lim_{n\rightarrow \infty}\|\eta_{s_n}-\Phi^{m,q}_{s_n}(\eta)\|_{L^2\Omega^{m,q}(A,E|_A,g_{s_n}|_A,\rho|_A)}=0.$$ Furthermore Lemma \ref{lemma5} tells us that $$\lim_{n\rightarrow \infty}\|\Psi_{s_n}^{m,q}(\Phi_{s_n}^{m,q}(\eta))-\Psi_0^{m,q}\eta\|_{L^2\Omega^{m,q}(A,E|_A,g_1|_A,\rho|_A)}=0.$$
We can thus conclude that $$\lim_{n\rightarrow \infty}\|\Psi_{s_n}^{m,q}\eta_{s_n}-\Psi_0^{m,q}\eta\|_{L^2\Omega^{m,q}(A,E|_A,g_1|_A,\rho|_A)}=0.$$
Conversely let us assume \eqref{lim}. We want to show that $\eta_{s_n}\rightarrow \eta$ strongly as $n\rightarrow \infty$, that is $$\lim_{n\rightarrow \infty}\|\eta_{s_n}-\Phi_{s_n}^{m,q}\eta\|_{L^2\Omega^{m,q}(A,E|_A,g_{s_n}|_A,\rho|_A)}=0.$$
We have 
$$\begin{aligned}
&\|\eta_{s_n}-\Phi_{s_n}^{m,q}\eta\|_{L^2\Omega^{m,q}(A,E|_A,g_{s_n}|_A,\rho|_A)}=\|\Psi_{s_n}^{m,q}\eta_{s_n}-\Psi_{0}^{m,q}\eta+\Psi_{0}^{m,q}\eta-\Psi^{m,q}_{s_n}(\Phi_{s_n}^{m,q}\eta)\|_{L^2\Omega^{m,q}(A,E|_A,g_1|_A,\rho|_A)}\leq\\
& \|\Psi_{s_n}^{m,q}\eta_{s_n}-\Psi_{0}^{m,q}\eta\|_{L^2\Omega^{m,q}(A,E|_A,g_1|_A,\rho|_A)}+\|\Psi_{0}^{m,q}\eta-\Psi^{m,q}_{s_n}(\Phi_{s_n}^{m,q}\eta)\|_{L^2\Omega^{m,q}(A,E|_A,g_1|_A,\rho|_A)}.
\end{aligned}$$
We assumed that $$\lim_{n\rightarrow \infty}\|\Psi_{s_n}^{m,q}\eta_{s_n}-\Psi_{0}^{m,q}\eta\|_{L^2\Omega^{m,q}(A,E|_A,g_1|_A,\rho|_A)}=0$$ and by Lemma \ref{lemma5} we know that $$\lim_{n\rightarrow \infty}\|\|\Psi_{0}^{m,q}\eta-\Psi^{m,q}_{s_n}(\Phi_{s_n}^{m,q}\eta)\|_{L^2\Omega^{m,q}(A,E|_A,g_1|_A,\rho|_A)}\|=0.$$
We can thus conclude that $\eta_{s_n}\rightarrow \eta$ strongly as $n\rightarrow \infty$.
\end{proof}

\begin{lemma}
\label{lemma6bis}
Let $\{s_n\}_{n\in \mathbb{N}}\subset (0,1]$ be a sequence such that $s_n\rightarrow 0$ as $n\rightarrow \infty$. Let $\eta\in L^2\Omega^{m,q}(A,E|_A,h|_A,\rho|_A)$ and let $\{\eta_{s_n}\}_{n\in \mathbb{N}}$ be a sequence such that $\eta_{s_n}\in L^2\Omega^{m,q}(A,E|_A,g_{s_n}|_A,\rho|_A)$. Then  $\eta_{s_n}\rightarrow \eta$ weakly as $n\rightarrow \infty$ if and only if  
\begin{equation}
\label{limw}
\Psi_{s_n}^{m,q}\eta_{s_n} \rightharpoonup \Psi_0^{m,q}\eta
\end{equation}
in $L^2\Omega^{m,q}(A,E|_A,g_1|_A,\rho|_A)$ as $n\rightarrow \infty$.
\end{lemma}

\begin{proof}
Assume that $\eta_{s_n}\rightarrow \eta$ weakly as $n\rightarrow \infty$. Let $\omega\in L^2\Omega^{m,q}(A,g_1|_A,E|_A,\rho|_A)$. Thanks to Lemma \ref{lemma6} we know that $(\Psi_{s_n}^{m,q})^{-1}\omega\rightarrow (\Psi_0^{m,q})^{-1}\omega$ strongly as $n\rightarrow \infty$. Thus we obtain
$$\begin{aligned}
\lim_{n\rightarrow \infty }\langle\Psi_{s_n}^{m,q}\eta_{s_n},\omega\rangle_{L^2\Omega^{m,q}(A,E|_A,g_1|_A,\rho|_A)}&=\lim_{n\rightarrow \infty}\langle\eta_{s_n},(\Psi_{s_n}^{m,q})^{-1}\omega\rangle_{L^2\Omega^{m,q}(A,E|_A,g_{s_n}|_A,\rho|_A)}\\
&= \langle\eta,(\Psi_{0}^{m,q})^{-1}\omega\rangle_{L^2\Omega^{m,q}(A,E|_A,h|_A,\rho|_A)}\\
&= \langle\Psi_{0}^{m,q}\eta,\omega\rangle_{L^2\Omega^{m,q}(A,E|_A,g_1|_A,\rho|_A)}
\end{aligned}$$
as required. Conversely let us assume that $\Psi_{s_n}^{m,q}\eta_{s_n} \rightharpoonup \Psi_0^{m,q}\eta$
in $L^2\Omega^{m,q}(A,E|_A,g_1|_A,\rho|_A)$ as $n\rightarrow \infty$. Let $\{\phi_{s_n}\}_{n\in \mathbb{N}}$ be a sequence such that $\phi_{s_n}\in L^2\Omega^{m,q}(A,E|_A,g_{s_n}|_A,\rho|_A)$ and $\phi_{s_n}\rightarrow \phi$ strongly to some $\phi\in L^2\Omega^{m,q}(A,E|_A,h|_A,\rho|_A)$ as $n\rightarrow \infty$. Lemma \ref{lemma6} tells us that $\|\Psi_{s_n}^{m,q}\phi_{s_n}-\Psi_0^{m,q}\phi\|_{L^2\Omega^{m,q}(A,E|_A,g_1|_A,\rho|_A)}\rightarrow 0$ as $n\rightarrow \infty$. Hence we obtain
$$\begin{aligned}
\lim_{n\rightarrow \infty}\langle\eta_{s_n},\phi_{s_n}\rangle_{L^2\Omega^{m,q}(A,E|_A,g_{s_n}|_A,\rho|_A)}&=\lim_{n\rightarrow \infty}\langle\Psi^{m,q}_{s_n}\eta_{s_n},\Psi_{s_n}^{m,q}\phi_{s_n}\rangle_{L^2\Omega^{m,q}(A,E|_A,g_1|_A,\rho|_A)}\\
&=\langle\Psi^{m,q}_{0}\eta,\Psi_{0}^{m,q}\phi\rangle_{L^2\Omega^{m,q}(A,E|_A,g_1|_A,\rho|_A)}\\
&=\langle \eta,\phi\rangle_{L^2\Omega^{m,q}(A,E|_A,h|_A,\rho|_A)}
\end{aligned}$$
as desired.
\end{proof}

\begin{lemma}
\label{lemma7bis}
Let $\{s_n\}_{n\in \mathbb{N}}\subset (0,1]$ be a sequence such that $s_n\rightarrow 0$ as $n\rightarrow \infty$. Then:
$$\lim_{n\rightarrow \infty}\|\Psi_{s_n}^{m,q}\circ G_{\overline{\partial}_{E,m,q,\max}^{g_{s_n},h}}\circ (\Psi_{0}^{m,q+1})^{-1}-\Psi_{0}^{m,q}\circ G_{\overline{\partial}_{E,m,q,\max}^{h,h}}\circ (\Psi_{0}^{m,q+1})^{-1}\|_{\mathrm{op}}=0$$
and $$\lim_{n\rightarrow \infty}\|G_{\overline{\partial}_{E,m,q}^{g_1,g_{s_n}}}\circ (\Psi_{s_n}^{m,q+1})^{-1}- G_{\overline{\partial}_{E,m,q,\max}^{g_1,h}}\circ (\Psi_{0}^{m,q+1})^{-1}\|_{\mathrm{op}}=0.$$
\end{lemma}

\begin{proof}
The first limit above follows immediately by Prop.\ref{usefulprop}, Th.\ref{compact2} and Lemmas \ref{lemma6} and \ref{lemma6bis}. The second one follows immediately by Prop.\ref{usefulprop}, Th.\ref{compact1} and Lemmas \ref{lemma6} and \ref{lemma6bis}.
\end{proof}

\begin{lemma}
\label{lemma8}
Let $\{s_n\}_{n\in \mathbb{N}}\subset (0,1]$ be a sequence such that $s_n\rightarrow 0$ as $n\rightarrow \infty$. Then: $$\lim_{n\rightarrow \infty}\|\Psi_{0}^{m,q+1}\circ G_{\overline{\partial}_{E,m,q,\min}^{g_{s_n},h,t}}\circ (\Psi_{s_n}^{m,q})^{-1}-\Psi_{0}^{m,q+1}\circ G_{\overline{\partial}_{E,m,q,\min}^{h,h,t}}\circ (\Psi_{0}^{m,q})^{-1}\|_{\mathrm{op}}=0$$
and $$\lim_{n\rightarrow \infty}\|(\Psi_{s_n}^{m,q+1})\circ G_{\overline{\partial}_{E,m,q}^{g_1,g_{s_n},t}} - (\Psi_{0}^{m,q+1})\circ G_{\overline{\partial}_{E,m,q,\min}^{g_1,h,t}} \|_{\mathrm{op}}=0.$$
\end{lemma}
\begin{proof}
Note that $$\left(\Psi_{s_n}^{m,q}\circ G_{\overline{\partial}_{E,m,q,\max}^{g_{s_n},h}}\circ (\Psi_{0}^{m,q+1})^{-1}\right)^*=\Psi_{0}^{m,q+1}\circ G_{\overline{\partial}_{E,m,q,\min}^{g_{s_n},h,t}}\circ (\Psi_{s_n}^{m,q})^{-1}$$ and $$\left(\Psi_{0}^{m,q}\circ G_{\overline{\partial}_{E,m,q,\max}^{h,h}}\circ (\Psi_{0}^{m,q+1})^{-1}\right)^*=\Psi_{0}^{m,q+1}\circ G_{\overline{\partial}_{E,m,q,\min}^{h,h,t}}\circ (\Psi_{0}^{m,q})^{-1}.$$
Analogously $$\left(G_{\overline{\partial}_{E,m,q}^{g_1,g_{s_n}}}\circ (\Psi_{s_n}^{m,q+1})^{-1}\right)^*=(\Psi_{s_n}^{m,q+1})\circ G_{\overline{\partial}_{E,m,q}^{g_1,g_{s_n},t}}$$
and $$\left(G_{\overline{\partial}_{E,m,q,\max}^{g_1,h}}\circ (\Psi_{0}^{m,q+1})^{-1}\right)^*=(\Psi_{0}^{m,q+1})\circ G_{\overline{\partial}_{E,m,q,\min}^{g_1,h,t}}.$$
The conclusion now follows by Lemma \ref{lemma7bis}.
\end{proof}

Let $\pi^{m,q}_{K,s}$ and $\pi^{m,q}_{K,0}$ be the projections defined in \eqref{orto} and \eqref{orto2}, respectively.

\begin{lemma}
\label{lemma9}
Let $\{s_n\}_{n\in \mathbb{N}}\subset (0,1]$ be a sequence such that $s_n\rightarrow 0$ as $n\rightarrow \infty$. Then $$\lim_{n\rightarrow \infty}\|\Psi^{m,q}_{s_n}\circ \pi_{K,s_n}^{m,q}\circ (\Psi_{s_n}^{m,q})^{-1}-\Psi^{m,q}_{0}\circ \pi_{K,0}^{m,q}\circ (\Psi_{0}^{m,q})^{-1}\|_{\mathrm{op}}=0.$$
\end{lemma}

\begin{proof}
This follows by Prop.\ref{usefulprop}, Th. \ref{projker} and Lemmas \ref{lemma6} and \ref{lemma6bis}.
\end{proof}

Now for each $s\in [0,1]$ let us consider the following complex

\begin{align}
\nonumber &L^2\Omega^{m,0}(A,E|_A,g_1|_A,\rho|_A)\stackrel{\overline{\partial}^{g_1,g_1}_{E,m,0}}{\longrightarrow}...\stackrel{\overline{\partial}^{g_1,g_1}_{E,m,q-2}}{\longrightarrow}L^2\Omega^{m,q-1}(A,E|_A,g_1|_A,\rho|_A)\stackrel{\overline{\partial}^{g_1,g_s}_{E,m,q-1,\max}}{\longrightarrow} L^2\Omega^{m,q}(A,E|_A,g_s|_A,\rho|_A)\\
& \label{Giava} \stackrel{\overline{\partial}^{g_s,h}_{E,m,q,\max}}{\longrightarrow}
L^2\Omega^{m,q+1}(A,E|_A,h|_A,\rho|_A)\stackrel{\overline{\partial}^{h,h}_{E,m,q+1,\max}}{\longrightarrow}... \stackrel{\overline{\partial}^{h,h}_{E,m,m-1,\max}}{\longrightarrow}L^2\Omega^{m,m}(A,E|_A,h|_A,\rho|_A).
\end{align}

In other words from $0$ up to $q-1$ we have the $L^2$-$\overline{\partial}_{E}$-complex with respect to $g_1$ and $(E,\rho)$, from $q+1$ up $m$ we have the maximal $L^2$-$\overline{\partial}_{E}$-complex with respect to $h$ and $(E,\rho)$ and the connecting piece is given by $$\stackrel{\overline{\partial}^{g_1,g_s}_{E,m,q-1,\max}}{\longrightarrow}L^2\Omega^{m,q}(A,E|_A,g_s|_A,\rho|_A)\stackrel{\overline{\partial}^{g_s,h}_{E,m,q,\max}}{\longrightarrow}L^2\Omega^{m,q+1}(A,E|_A,h|_A,\rho|_A).$$ Let us now introduce the following complex

\begin{align}
& \nonumber L^2\Omega^{m,0}(A,E|_A,g_1|_A,\rho|_A)\stackrel{\overline{\partial}^{g_1,g_1}_{E,m,0}}{\longrightarrow}...\stackrel{\overline{\partial}^{g_1,g_1}_{E,m,q-2}}{\longrightarrow}L^2\Omega^{m,q-1}(A,E|_A,g_1|_A,\rho|_A)\stackrel{D^{g_1,g_s}_{m,q-1}}{\longrightarrow}L^2\Omega^{m,q}(A,E|_A,g_1|_A,\rho|_A)\\
& \label{ucomplex} \stackrel{D^{g_s,h}_{m,q}}{\longrightarrow}L^2\Omega^{m,q+1}(A,E|_A,g_1|_A,\rho|_A)\stackrel{D^{h,h}_{m,q+1}}{\longrightarrow}...\stackrel{D^{h,h}_{m,m-1}}{\longrightarrow}L^2\Omega^{m,m}(A,E|_A,g_1|_A,\rho|_A)
\end{align}
with $D^{g_1,g_s}_{m,q-1}:= \Psi_{s}^{m,q}\circ \overline{\partial}_{E,m,q-1,\max}^{g_1,g_s}$, $D^{g_s,h}_{m,q}:= \Psi_{0}^{m,q+1}\circ \overline{\partial}_{E,m,q,\max}^{g_s,h}\circ (\Psi_s^{m,q})^{-1}$ and  $D^{h,h}_{m,r}:= \Psi_{0}^{m,r+1}\circ \overline{\partial}_{E,m,r,\max}^{h,h}\circ (\Psi_0^{m,r})^{-1}$\quad for each  $r=q+1,...,m.$ Let 
\begin{equation}
\label{urolledup}
P^{m,q}_s:L^2\Omega^{m,\bullet}(A,E|_A,g_1|_A,\rho|_A)\rightarrow L^2\Omega^{m,\bullet}(A,E,g_1,\rho)
\end{equation}
be the rolled-up operator of the complex \eqref{ucomplex}, see \cite[p. 91-92]{BL}. This means nothing but $$\begin{aligned}
&P^{m,q}_s|_{L^2\Omega^{m,r}(A,E|_A,g_1|_A,\rho|_A)}:=\overline{\partial}_{E,m,r}^{g_1,g_1}+\overline{\partial}_{E,m,r-1}^{g_1,g_1,t},\ r=0,...,q-2\\
&P^{m,q}_s|_{L^2\Omega^{m,q-1}(A,E|_A,g_1|_A,\rho|_A)}:=D^{g_1,g_s}_{m,q-1}+\overline{\partial}_{E,m,q-2}^{g_1,g_1,t}\\
&P^{m,q}_s|_{L^2\Omega^{m,q}(A,E|_A,g_1|_A,\rho|_A)}:=D^{g_s,h}_{m,q}+(D^{g_1,g_s}_{m,q-1})^*\\ &P^{m,q}_s|_{L^2\Omega^{m,q+1}(A,E|_A,g_1|_A,\rho|_A)}:=D^{h,h}_{m,q+1}+(D^{g_s,h}_{m,q})^*\\ 
&P^{m,q}_s|_{L^2\Omega^{m,r}(A,E|_A,g_1|_A,\rho|_A)}:=D^{h,h}_{m,r}+(D^{h,h}_{m,r-1})^*,\ r=q+2,...,m.
\end{aligned}$$
Note that $P^{m,q}_0=P^{m,q-1}_1$ for each $q\in \{1,...,m\}$. Moreover we have $(D^{g_1,g_s}_{m,q-1})^*= \overline{\partial}_{E,m,q-1,\min}^{g_1,g_s,t}\circ (\Psi_s^{m,q})^{-1}$, $(D^{g_s,h}_{m,q})^*=\Psi^{m,q}_s\circ \overline{\partial}^{g_s,h,t}_{E,m,q,\min}\circ (\Psi^{m,q+1}_0)^{-1}$ and $(D^{h,h}_{m,r-1})^*=\Psi^{m,r-1}_0\circ \overline{\partial}_{E,m,r-1,\min}^{h,h,t}\circ (\Psi^{m,r}_0)^{-1}$ with $r=q+2,...,m$. Furthermore, we point out that when $q=m$ and $s=1$ the complex \eqref{ucomplex} is nothing but the $L^2$-$\overline{\partial}_E$-complex on $M$ w.r.t. $g_1$ and $\rho$ whereas when $q=0=s$ the complex \eqref{ucomplex} is unitarily equivalent to the maximal $L^2$-$\overline{\partial}_{E}$-complex on $A$ with respect to $h$ and $(E,\rho)$. Consequently in the case $q=m$ and $s=1$ the operator \eqref{urolledup} is the Dirac-Dolbeault operator on $M$ w.r.t. $g_1$ and $(E,\rho)$ whereas in the case $q=0=s$ the operator \eqref{urolledup} is unitary equivalent to the rolled-up operator of the maximal $L^2$-$\overline{\partial}_E$ complex over $A$ with respect to $h$ and $(E,\rho)$. We have now the following property:
\begin{lemma}
\label{lemma10}
In the setting of Th. \ref{projker}, the operator $P^{m,q}_s$ is self-adjoint and  has entirely discrete spectrum for each $q\in \{0,...,m\}$ and $s\in [0,1]$.
\end{lemma}
\begin{proof}
The fact that $P^{m,q}_s$ is self-adjoint follows immediately by its definition because it is the rolled-up operator of a Hilbert complex, see \cite[p. 92]{BL}. The discreteness of its spectrum is an easy consequence of the fact that  the complex \eqref{Giava} has finite cohomology and each operator has a compact Green operator.
\end{proof}

We are now in a position to prove the main result of this subsection.

\begin{teo}
\label{resolv}
In the above setting  we have $$\lim_{s\rightarrow 0}\|(P^{m,q}_s+i)^{-1}-(P^{m,q}_0+i)^{-1}\|_{\mathrm{op}}=0$$ for each $q=0,...,m$.
\end{teo}

\begin{proof}
According to Lemma \ref{lemmata} it suffices to show that $$\lim_{s\rightarrow 0}\|G_{P^{m,q}_s}-G_{P_0^{m,q}}\|_{\mathrm{op}}=0\quad \mathrm{and}\quad \lim_{s\rightarrow 0}\|\pi_{K,P^{m,q}_s}-\pi_{K,P^{m,q}_0}\|_{\mathrm{op}}=0$$ where $\pi_{K,P^{m,q}_s}:L^2\Omega^{m,\bullet}(A,E|_A,g_1|_A,\rho|_A)\rightarrow L^2\Omega^{m,\bullet}(A,E|_A,g_1|_A,\rho|_A)$ stands for  the orthogonal projection on $\ker(P^{m,q}_s)$. Since $P_s^{m,q}$ is the rolled-up operator of the complex \eqref{ucomplex} it is easy to check that 
\begin{equation}
\label{acciuga}
\ker(P_s^{m,q})=\bigoplus_{r=0}^m\ker(P_s^{m,q}|_{L^2\Omega^{m,r}(A,E|_A,g_1|_A,\rho|_A)})\quad \mathrm{and}\quad G_{P_s^{m,q}}=\bigoplus_{r=0}^mG_{P_s^{m,q}}|_{L^2\Omega^{m,r}(A,E|_A,g_1|_A,\rho|_A)}.
\end{equation}
Concerning $\ker(P^{m,q}_s)$ we obtain the following decomposition:  
$$\begin{aligned}\ker(P_s^{m,q})&=\left(\bigoplus_{r=0}^{q-2}\ker(\overline{\partial}^{g_1,g_1}_{E,m,r})\cap\ker(\overline{\partial}^{g_1,g_1,t}_{E,m,r})\right)\bigoplus \left(\ker(D^{g_1,g_s}_{m,q-1})\cap\ker(\overline{\partial}^{g_1,g_1,t}_{E,m,q-2})\right)\\
&\bigoplus \left(\ker(D^{g_s,h}_{m,q})\cap\ker((D^{g_1,g_s}_{m,q-1})^*)\right)\bigoplus \left(\ker(D^{h,h}_{m,q+1})\cap\ker((D^{g_s,h}_{m,q})^*)\right)\\
& \bigoplus \left(\bigoplus_{r=q+2}^{m}\ker(D^{h,h}_r)\cap\ker((D^{h,h}_r)^*)\right).
\end{aligned}$$
Note that $\ker(D^{g_1,g_s}_{m,q-1})$ is independent on $s\in[0,1]$. Moreover also $\ker((D^{g_s,h}_{m,q})^*)$ does not depend on $s\in[0,1]$. This latter assertion follows because by assumption $\im(\overline{\partial}_{E,m,q,\max}^{g_1,h})=\im(\overline{\partial}_{E,m,q,\max}^{h,h})$ which in turn implies that $\im(\overline{\partial}_{E,m,q,\max}^{g_s,h})=\im(\overline{\partial}_{E,m,q,\max}^{h,h})$ for each $s\in [0,1]$. Thus $\im(D^{g_s,h}_{m,q})=\im(D^{h,h}_{m,q})$ for each $s\in [0,1]$ and eventually we can conclude that $\ker((D^{g_s,h}_{m,q})^*)$ does not depend on $s\in[0,1]$ since $\ker((D^{g_s,h}_{m,q})^*)=(\im(D^{g_s,h}_{E,m,q}))^{\bot}$ in $L^2\Omega^{m,q+1}(A,E,h,\rho)$. Thus in the above decomposition the only term depending on $s$ is $\ker(D^{g_s,h}_{m,q})\cap\ker((D^{g_1,g_s}_{m,q-1})^*)$ and so the limit 
\begin{equation}
\label{boil}
\lim_{s\rightarrow 0}\|\pi_{K,P^{m,q}_s}-\pi_{K,P^{m,q}_0}\|_{\mathrm{op}}=0
\end{equation}

 boils down to proving that 
\begin{equation}
\label{limprojk}
\lim_{n\rightarrow \infty}\|\Psi^{m,q}_{s_n}\circ \pi_{K,s_n}^{m,q}\circ (\Psi_{s_n}^{m,q})^{-1}-\Psi^{m,q}_{0}\circ \pi_{K,0}^{m,q}\circ (\Psi_{0}^{m,q})^{-1}\|_{\mathrm{op}}=0.
\end{equation}
where $\pi_{K,s}^{m,q}$ is the orthogonal projection defined in \eqref{orto}.  
By Lemma \ref{lemma9} we already know that \eqref{limprojk} holds true and thus \eqref{boil} holds true as well. Let us go back to the second half of \eqref{acciuga}. Looking at \eqref{ucomplex} we note that the only terms in the decomposition of $G_{P_s^{m,q}}$ that depend on $s$ are $$G_{P_s^{m,q}}|_{L^2\Omega^{m,r}(A,E|_A,g_1|_A,\rho|_A)},\quad r=q-1,q,q+1.$$ We have 
$$
\begin{aligned}
G_{P_s^{m,q}}&|_{L^2\Omega^{m,q-1}(A,E|_A,g_1|_A,\rho|_A)}=\\
&G_{\overline{\partial}_{E,m,q-2}^{g_1,g_1}}+G_{(D^{g_1,g_s}_{m,q-1})^*}: L^2\Omega^{m,q-1}(A,E|_A,g_1|_A,\rho|_A)\rightarrow L^2\Omega^{m,q-2}(A,E|_A,g_1|_A,\rho|_A)\oplus L^2\Omega^{m,q}(A,E|_A,g_1|_A,\rho|_A);\\
G_{P_s^{m,q}}&|_{L^2\Omega^{m,q}(A,E|_A,g_1|_A,\rho|_A)}=\\
&G_{D_{m,q-1}^{g_1,g_s}}+G_{(D^{g_s,h}_{m,q})^*}: L^2\Omega^{m,q}(A,E|_A,g_1|_A,\rho|_A)\rightarrow L^2\Omega^{m,q-1}(A,E|_A,g_1|_A,\rho|_A)\oplus L^2\Omega^{m,q+1}(A,E|_A,g_1|_A,\rho|_A);\\
G_{P_s^{m,q}}&|_{L^2\Omega^{m,q+1}(A,E|_A,g_1|_A,\rho|_A)}=\\
&G_{D_{m,q}^{g_s,h}}+G_{(D^{h,h}_{m,q+1})^*}: L^2\Omega^{m,q+1}(A,E|_A,g_1|_A,\rho|_A)\rightarrow L^2\Omega^{m,q}(A,E|_A,g_1|_A,\rho|_A)\oplus L^2\Omega^{m,q+2}(A,E|_A,g_1|_A,\rho|_A).
\end{aligned}$$
Therefore
$$\begin{aligned}
\|G_{P^{m,q}_s}-G_{P^{m,q}_0}\|_{\mathrm{op}}&=\|\bigoplus_{r=0}^mG_{P^{m,q}_s}|_{L^2\Omega^{m,r}(A,E|_A,g_1|_A,\rho|_A)}-\bigoplus_{r=0}^mG_{P^{m,q}_0}|_{L^2\Omega^{m,r}(A,E|_A,g_1|_A,\rho|_A)}\|_{\mathrm{op}}\\
&\leq \sum_{r=q-1}^{q+1}\|G_{P^{m,q}_s}|_{L^2\Omega^{m,r}(A,E|_A,g_1|_A,\rho|_A)}-G_{P^{m,q}_0}|_{L^2\Omega^{m,r}(A,E|_A,g_1|_A,\rho|_A)}\|_{\mathrm{op}}\\
&\leq \|G_{(D^{g_1,g_s}_{m,q-1})^*}-G_{(D^{g_1,h}_{m,q-1})^*}\|_{\mathrm{op}}+\|G_{D_{m,q-1}^{g_1,g_s}}-G_{D_{m,q-1}^{g_1,h}}\|_{\mathrm{op}}\\
&+\|G_{(D^{g_s,h}_{m,q})^*}-G_{(D^{h,h}_{m,q})^*}\|_{\mathrm{op}}+\|G_{D^{g_s,h}_{m,q}}-G_{D^{h,h}_{m,q}}\|_{\mathrm{op}}.
\end{aligned}$$
Clearly for each $s\in [0,1]$ we have $$\begin{aligned}
&G_{D^{g_1,g_s}_{m,q-1}}= G_{\overline{\partial}^{g_1,g_s}_{E,m,q-1,\max}}\circ (\Psi_s^{m,q})^{-1},\quad G_{(D_{m,q-1}^{g_1,g_s})^*}=\Psi^{m,q}_s\circ G_{\overline{\partial}_{E,m,q-1,\min}^{g_1,g_s,t}}\\ & G_{D^{g_s,h}_{m,q}}=\Psi^{m,q}_s\circ G_{\overline{\partial}^{g_s,h}_{E,m,q,\max}}\circ (\Psi^{m,q+1}_0)^{-1},\quad  G_{(D^{g_s,h}_{m,q})^*}=\Psi_0^{m,q+1}\circ G_{(\overline{\partial}^{g_s,h,t}_{E,m,q,\min})^*}\circ (\Psi_s^{m,q})^{-1}.
\end{aligned}$$ Thanks to Lemmas \ref{lemma7bis} and \ref{lemma8} we know that 
$$\begin{aligned}
&\lim_{s\rightarrow0}\|G_{D^{g_1,g_s}_{m,q-1}}-G_{D^{g_1,h}_{m,q-1}}\|_{\mathrm{op}}=0,\quad \lim_{s\rightarrow0}\|G_{(D_{m,q-1}^{g_1,g_s})^*}-G_{(D_{m,q-1}^{g_1,h})^*}\|_{\mathrm{op}}=0\\
& \lim_{s\rightarrow0}\|G_{D^{g_s,h}_{m,q}}-G_{D^{h,h}_{m,q}}\|_{\mathrm{op}}=0,\quad \lim_{s\rightarrow0}\|G_{(D^{g_s,h}_{m,q})^*}-G_{(D^{h,h}_{m,q})^*}\|_{\mathrm{op}}=0
\end{aligned}.$$
We can thus conclude that $$\lim_{s\rightarrow 0}\|G_{P^{m,q}_s}-G_{P^{m,q}_0}\|_{\mathrm{op}}=0$$ as required.
\end{proof}

\section{Resolutions and canonical $K$-homology classes}
Let $(M,g)$ be a compact complex manifold of complex dimension $m$ endowed with a Hermitian metric $g$. Let $(E,\rho)\rightarrow M$ be a Hermitian holomorphic vector bundle over $M$. For each $p\in\{0,...,m\}$ let us consider the Hilbert space $L^2\Omega^{p,\bullet}(M,E,g,\rho)$ endowed with the grading induced by the splitting in $L^2$ $E$-valued $(p,\bullet)$-forms with even/odd antiholomorphic degree. Furthermore we consider the corresponding Dirac-Dolbeault operator $\overline{\eth}_{E,p}:=\overline{\partial}_{E,p}+\overline{\partial}_{E,p}^t$ $$\overline{\eth}_{E,p}:L^2\Omega^{p,\bullet}(M,E,g,\rho)\rightarrow L^2\Omega^{p,\bullet}(M,E,g,\rho)$$ and the $C^*$-algebra $C(M):=C(M,\mathbb{C})$ acting on $L^2\Omega^{p,\bullet}(M,E,g,\rho)$ by pointwise multiplication: 
\begin{equation}
\label{pointwisem}
C(M)\ni f\mapsto m_f\in \mathcal{B}(L^2\Omega^{p,\bullet}(M,E,g,\rho))\ \mathrm{given\ by}\ m_f\psi:=f\psi
\end{equation}
for every $\psi\in L^2\Omega^{p,\bullet}(M,E,g,\rho)$. Finally let us consider as a dense subalgebra $\mathcal{A}:=C^{\infty}(M)$. It is well known that the triplet $(L^2\Omega^{p,\bullet}(M,E,g,\rho),m, \overline{\eth}_{E,p})$ is an even unbounded Fredholm module, see, for example \cite[\S 10]{Nigel}, and thus the triplet  $(L^2\Omega^{p,\bullet}(M,E,g,\rho),m, \overline{\eth}_{E,p}\circ (\id+(\overline{\eth}_{E,p})^2)^{-\frac{1}{2}})$ gives a class in $KK_0(C(M),\mathbb{C})$, see Prop. \ref{unbounded}. We denote this class with $$[\overline{\eth}_{E,p}]\in KK_0(C(M),\mathbb{C})$$ and when $p=m$ we call it the canonical analytic $K$-homology class of $M$ and $E$. In particular when $p=m$ and $E$ is the trivial holomorphic line bundle $M\times \mathbb{C}\rightarrow M$ we call the above class the canonical analytic $K$-homology class of $M$. This terminology is based on the fact that $\Lambda^{m,0}(M)$ is called the canonical bundle of $M$. We remark that since $M$ is compact the class $[\overline{\eth}_{E,p}]$ depends neither on  $g$ nor on $\rho$ since all the Hermitian metrics on $M$ as well as all the Hermitian metrics on $E$ are quasi-isometric, see e.g. \cite{Hilsum}. Now we briefly recall  the notion of Hermitian complex space. Complex spaces are a classical topic in complex geometry and we refer for instance to \cite{GrRe} for definitions and properties. 
We recall that a  paracompact  and reduced complex space $X$ is said to be \emph{Hermitian} if  the regular part of $X$ carries a Hermitian metric $\gamma$  such that for every point $p\in X$ there exists an open neighbourhood $U\ni p$ in $X$, a proper holomorphic embedding of $U$ into a polydisc $\phi: U \rightarrow \mathbb{D}^N\subset \mathbb{C}^N$ and a Hermitian metric $g$ on $\mathbb{D}^N$ such that $(\phi|_{\reg(U)})^*g=\gamma$, see e.g. \cite{JRu}. In this case we will write $(X,\gamma)$ and with a little abuse of language we will say that $\gamma$ is a \emph{Hermitian metric on $X$}. Clearly any analytic subvariety of a complex Hermitian manifold endowed with the metric induced by the ambient space metric is an example of Hermitian complex space. Note that when $X$ is compact all the Hermitian metrics on $X$ belong to the same quasi-isometry class. This follows easily by the lifting lemma, see \cite[Rmk 1.30.1 p. 37]{SingDef} . Let now $F\rightarrow X$ be a holomorphic vector bundle of complex rank $s$. We assume that $F|_{\reg(X)}$ is equipped with a Hermitian metric $\tau$ such  that for each point $p\in X$ the following property holds true: there exists an open neighbourhood $U$, a positive constant $c$ and a holomorphic trivialization $\psi:E|_U\rightarrow U\times \mathbb{C}^s$ such that, denoting by $\sigma$ the Hermitian metric on $\reg(U)\times \mathbb{C}^s$ induced by $\tau$ through $\psi$, we have 
\begin{equation}
\label{trappolo}
c^{-1}\upsilon_{\mathrm{std}}\leq \sigma\leq c\upsilon_{\mathrm{std}}
\end{equation}
where $\upsilon_{\mathrm{std}}$ is the Hermitian metric on $\reg(U)\times \mathbb{C}^s$ that assigns to each point of $\reg(U)$ the standard Euclidean K\"ahler metric of $\mathbb{C}^s$. In order to state the next results, we also need to recall the existence of a resolution of singularities, see \cite{Hiro}. Let $X$ be a compact and irreducible complex space. There then exists a compact complex manifold $M$, a divisor with only normal crossings $D\subset M$ and a surjective holomorphic map $\pi:M\rightarrow X$ such that $\pi^{-1}(\sing(X))=D$ and 
\begin{equation}
\label{hiro}
\pi|_{M\setminus D}: M\setminus D\longrightarrow X\setminus \sing(X)
\end{equation}
is a biholomorphism. Consider now the maximal $L^2$-$\overline{\partial}_F$-complex $$0\rightarrow L^2\Omega^{m,0}(\reg(X),F,\gamma,\tau)\stackrel{\overline{\partial}_{E,m,0,\max}^{\gamma,\gamma}}{\longrightarrow}...\stackrel{\overline{\partial}_{E,m,m-1,\max}^{\gamma,\gamma}}{\longrightarrow}L^2\Omega^{m,m}(\reg(X),F,\gamma,\tau)\rightarrow 0$$ and let
\begin{equation}
\label{rol}
\overline{\eth}_{F,m,\mathrm{abs}}:L^2\Omega^{m,\bullet}(\reg(X),F,\gamma,\tau)\rightarrow L^2\Omega^{m,\bullet}(\reg(X),F,\gamma,\tau)
\end{equation}
be the corresponding rolled-up operator.  Note that we can write
\begin{equation}
\label{bdbd}
\overline{\eth}_{F,m,\mathrm{abs}}=\overline{\partial}_{F,m,\max}^{\gamma,\gamma}+\overline{\eth}_{F,m,\min}^{\gamma,\gamma,t}
\end{equation}

with 
\begin{equation}
\label{roll}
\overline{\partial}_{F,m,\max}^{\gamma,\gamma}:L^2\Omega^{m,\bullet}(\reg(X),F,\gamma,\tau)\rightarrow L^2\Omega^{m,\bullet}(\reg(X),F,\gamma,\tau)
\end{equation}
defined by  $$\overline{\partial}_{F,m,\max}^{\gamma,\gamma}|_{L^2\Omega^{m,r}(\reg(X),F,\gamma,\tau)}:=\overline{\partial}_{E,m,r,\max}^{\gamma,\gamma}$$ for each $r=0,...,m$ and with $\overline{\partial}_{F,m,\min}^{\gamma,\gamma,t}$ defined in the obvious analogous way. We have now the following 

\begin{prop}
\label{minclass}
Let $(X,\gamma)$ be a compact and irreducible Hermitian complex space of complex dimension $m$ such that $\dim(\sing(X))=0$. Let $(F,\tau)\rightarrow X$ be a Hermitian holomorphic vector bundle over $X$ which satisfies \eqref{trappolo}.
 Then the triplet
\begin{equation}
\label{hdmin}
(L^2\Omega^{m,\bullet}(\reg(X),F,\gamma,\tau),m,\overline{\eth}_{F,m,\mathrm{abs}})
\end{equation}
defines an even unbounded Fredholm module for $C(X)$ and thus a class 
\begin{equation}\label{eq:classes}
[\overline{\eth}_{F,m,\mathrm{abs}}]\in KK_0(C(X),\mathbb{C}).
\end{equation}
Moreover this  class does not depend  on the particular Hermitian metric $\gamma$ that we fix on $X$.
\end{prop}
\begin{proof}
The proof follows by arguing as in \cite[Prop. 3.6]{BeiPiazza}.  In particular we use $S_c(X)$ defined as 
\begin{align}
\label{subalgebra}
& S_c(X):=\{f\in C(X)\cap C^{\infty}(\reg(X))\ \text{such that for each}\ p\in \sing(X)\ \text{there exists an}\\ 
& \nonumber \text{open   neighbourhood}\ U\ \text{of}\ p\ \text{with}\ f|_U=c\in \mathbb{C}\}
\end{align} 
as a dense $*$-subalgebra of $C(X)$. Moreover we recall that $m$ denotes the pointwise multiplication, see \eqref{pointwisem}. The only point that needs an explanation is the discreteness of the spectrum of $\overline{\eth}_{F,m,\mathrm{abs}}$. This is settled in the next lemma.
\end{proof}

\begin{lemma}
\label{tris}
Let $(X,\gamma)$ be a compact and irreducible Hermitian complex space of complex dimension $m$ with $\dim(\sing(X))=0$. Then $$\overline{\eth}_{F,m,\mathrm{abs}}:L^2\Omega^{m,\bullet}(\reg(X),F,\gamma,\tau)\rightarrow L^2\Omega^{m,\bullet}(\reg(X),F,\gamma,\tau)$$ has entirely discrete spectrum, with $(F,\tau)\rightarrow X$ any Hermitian holomorphic vector bundle over $X$ which satisfies \eqref{trappolo}. 
\end{lemma}

\begin{proof}
Let $\sing(X)=\{p_1,...,p_{\ell}\}$. First, we prove this lemma under some additional requirements: the holomorphic vector bundle $F$ is endowed with a Hermitian metric $\tau'$ such that for each $p_k\in \sing(X)$, $k=1,...,\ell$ there exists an open neighbourhood $U_k$ and a trivialization $\chi_k:F|_{U_k}\rightarrow U_k\times \mathbb{C}^n$, with $n:=\mathrm{rnk}(F)$, such that $\chi^*(\upsilon_{\mathrm{std}})=\tau'$, with $\upsilon_{std}$ defined in \eqref{trappolo}. Clearly we can always endow $F$ with such a metric. Thanks to \cite[Th. 5.2]{FBei} and \cite[Th. 1.2]{OvRu} we know that $$\overline{\eth}_{m,\mathrm{abs}}:L^2\Omega^{m,\bullet}(\reg(X),\gamma)\rightarrow L^2\Omega^{m,\bullet}(\reg(X),\gamma)$$ has entirely discrete spectrum. From this we get immediately that  the twisted Dirac-Dolbeault operator with respect to the trivial holomorphic vector bundle $\reg(X)\times\mathbb{C}^n$ endowed with the standard Euclidean K\"ahler metric $\upsilon_{\mathrm{std}}$ 
\begin{equation}
\label{b-twist}
\overline{\eth}_{\mathbb{C}^n,m,\mathrm{abs}}:L^2\Omega^{m,\bullet}(\reg(X),\reg(X)\times\mathbb{C}^n,\gamma,\upsilon_{\mathrm{std}})\rightarrow L^2\Omega^{m,\bullet}(\reg(X),\reg(X)\times\mathbb{C}^n,\gamma,\upsilon_{\mathrm{std}})
\end{equation}
 has entirely discrete spectrum, as well. Let now $U_0$ be an open subset of $\reg(X)$ such that $\{U_0,U_1,...,U_{\ell}\}$ is an open covering of $X$. We also assume that $U_0\cap \sing(X)=\emptyset$ and that $U_i\cap U_j=\emptyset$ for each $1\leq i<j\leq \ell$. Let $\{\phi_0,...,\phi_{\ell}\}$ be a partition of unity subordinated to $\{U_0,U_1,...,U_{\ell}\}$ such that $\phi_i\in C^{\infty}(\reg(X))\cap C(X)$ for each $i=0,...,\ell$. Note that for every $1\leq i\leq \ell$ there exists an open neighbourhood $V_i$ of $p_i$ with $V_i\subset U_i$ and $\phi_i|_{V_i}= 1$. In particular we have $d_0\phi_i\in \Omega_c^1(\reg(U_i))$. Consider now a sequence $\{\eta_j\}_{j\in \mathbb{N}}\subset \mathcal{D}(\overline{\eth}_{F,m,\mathrm{abs}})$ which is bounded with respect to the corresponding graph norm. Clearly the sequence $\{\phi_0\eta_j\}_{j\in\mathbb{N}}$ still lies in $\mathcal{D}(\overline{\eth}_{F,m,\mathrm{abs}})$ and  is bounded with respect to the corresponding graph norm. Moreover the support of $\phi_0\eta_j$ is contained in $U_0$ for each $j\in \mathbb{N}$. Since $U_0$ is relatively compact in $\reg(X)$ we can use elliptic estimates, see e.g. \cite[Lemma 1.1.17]{Lesch}, and a  Rellich-type compactness theorem  to deduce the existence of a subsequence $\{\eta_{0,j}\}_{j\in \mathbb{N}}\subset \{\eta_{j}\}_{j\in \mathbb{N}}$ such that $\{\phi_0\eta_{0,j}\}_{j\in \mathbb{N}}$ converges in $L^2\Omega^{m,\bullet}(\reg(X),F,\gamma,\tau')$. Consider now the sequence $\{\phi_1\psi_{0,j}\}_{j\in \mathbb{N}}\subset L^2\Omega^{m,\bullet}(\reg(X),\reg(X)\times\mathbb{C}^n,\gamma,\upsilon_{\mathrm{std}})$ with $\psi_{0,j}:=(\chi_1^{-1})^*(\eta_{0,j}|_{\reg(U_1)})$. It is clear that the sequence $\{\phi_1\psi_{0,j}\}_{j\in \mathbb{N}}$ lies both in the domain of $$\overline{\partial}^{\gamma,\gamma}_{\mathbb{C}^n,m,\max}:L^2\Omega^{m,\bullet}(\reg(U_1),\reg(X)\times\mathbb{C}^n,\gamma|_{\reg(U_1)},\upsilon_{\mathrm{std}}|_{\reg(U_1)})\rightarrow L^2\Omega^{m,\bullet}(\reg(U_1),\reg(X)\times\mathbb{C}^n,\gamma|_{\reg(U_1)},\upsilon_{\mathrm{std}}|_{\reg(U_1)})$$ and $$\overline{\partial}_{\mathbb{C}^n,m,\min}^{\gamma,\gamma,t}:L^2\Omega^{m,\bullet}(\reg(U_1),\reg(X)\times\mathbb{C}^n,\gamma|_{\reg(U_1)},\upsilon_{\mathrm{std}}|_{\reg(U_1)})\rightarrow L^2\Omega^{m,\bullet}(\reg(U_1),\reg(X)\times\mathbb{C}^n,\gamma|_{\reg(U_1)},\upsilon_{\mathrm{std}}|_{\reg(U_1)})$$ and it is bounded in the corresponding graph norm. From the definition of minimal domain we get immediately that $\{\phi_1\psi_{0,j}\}_{j\in \mathbb{N}}$ lies in the domain of  $$\overline{\partial}_{\mathbb{C}^n,m,\min}^{\gamma,\gamma,t}:L^2\Omega^{m,\bullet}(\reg(X),\reg(X)\times \mathbb{C}^n,\gamma,\upsilon_{\mathrm{std}})\rightarrow L^2\Omega^{m,\bullet}(\reg(X),\reg(X)\times\mathbb{C}^n,\gamma,\upsilon_{\mathrm{std}}).$$ Moreover since $\phi_1$ has compact support contained in $U_1$ it is not difficult to see that $\{\phi_1\psi_{0,j}\}_{j\in \mathbb{N}}$ lies also in the domain of  $$\overline{\partial}^{\gamma,\gamma}_{\mathbb{C}^n,m,\max}:L^2\Omega^{m,\bullet}(\reg(X),\reg(X)\times\mathbb{C}^n,\gamma,\upsilon_{\mathrm{std}})\rightarrow L^2\Omega^{m,\bullet}(\reg(X),\reg(X)\times\mathbb{C}^n,\gamma,\upsilon_{\mathrm{std}}).$$ Summarizing we have shown that $\{\phi_1\psi_{0,j}\}_{j\in \mathbb{N}}$ lies in the domain of  \eqref{b-twist} and it is bounded in the corresponding graph norm. Therefore there exists a subsequence $\{\psi_{1,j}\}_{j\in\mathbb{N}}\subset \{\psi_{0,j}\}_{j\in \mathbb{N}}$ such  that $\{\phi_1\psi_{1,j}\}_{j\in\mathbb{N}}$ converges in $L^2\Omega^{m,\bullet}(\reg(X),\reg(X)\times \mathbb{C}^n,\gamma,\upsilon_{\mathrm{std}})$. Eventually we can conclude that there exists a  subsequence $\{\eta_{1,j}\}_{j\in \mathbb{N}}\subset \{\eta_{0,j}\}_{j\in \mathbb{N}}$, which satisfies $(\chi_1^{-1})^*(\eta_{1,j}|_{\reg(U_1)})=\psi_{1,j}$, such that the sequence $\{\phi_1\eta_{1,j}\}_{j\in \mathbb{N}}$ converges in $L^2\Omega^{m,\bullet}(\reg(X),F,\gamma,\tau')$. Repeating this procedure up to $n$ we construct a subsequence  $\{\eta_{n,j}\}_{j\in \mathbb{N}}\subset \{\eta_{j}\}_{j\in \mathbb{N}}$ such that $\{\phi_i\eta_{n,j}\}_{j\in \mathbb{N}}$ converges in $L^2\Omega^{m,\bullet}(\reg(X),F,\gamma,\tau')$ for each $i=0,...,n$. We can thus conclude that the sequence $\{\eta_{n,j}\}_{j\in \mathbb{N}}$ converges in $L^2\Omega^{m,\bullet}(\reg(X),F,\gamma,\tau')$ and this completes the first part of the proof. Note that as a by-product of this first part of the proof we get that  $\im(\overline{\partial}^{\gamma,\gamma}_{F,m,q,\max})$ is a closed subspace of $L^2\Omega^{m,q}(\reg(X),F,\gamma,\tau')$  for each $q=0,...,m$,  $$\left(\ker(\overline{\partial}^{\gamma,\gamma}_{F,m,q,\max})\cap\ker(\overline{\partial}_{F,m,q-1,\min}^{\gamma,\gamma,t})\right)\cong \ker(\overline{\partial}_{F,m,q,\max}^{\gamma,\gamma})/\im(\overline{\partial}^{\gamma,\gamma}_{F,m,q-1,\max})$$ is a finite-dimensional vector space
and $$G_{\overline{\partial}_{F,m,q,\max}}^{\tau'}:L^2\Omega^{m,q+1}(\reg(X),F,\gamma,\tau')\rightarrow L^2\Omega^{m,q}(\reg(X),F,\gamma,\tau')$$ is a compact operator where we have denoted with $G_{\overline{\partial}_{F,m,q,\max}}^{\tau'}$ the Green operator of 
$$\overline{\partial}^{\gamma,\gamma}_{F,m,q,\max}:L^2\Omega^{m,q}(\reg(X),F,\gamma,\tau')\rightarrow L^2\Omega^{m,q+1}(\reg(X),F,\gamma,\tau').$$ Let now $\tau$ be an arbitrarily fixed Hermitian metric on $F\rightarrow X$ which satisfies \eqref{trappolo}. Since $\tau$ and $\tau'$ are quasi-isometric we know that for each $q=0,...,m$ $$\overline{\partial}^{\gamma,\gamma}_{F,m,q,\mathrm{abs}}:L^2\Omega^{m,q}(\reg(X),F,\gamma,\tau)\rightarrow L^2\Omega^{m,q+1}(\reg(X),F,\gamma,\tau)$$ has closed range and  $$\left(\ker(\overline{\partial}^{\gamma,\gamma}_{F,m,q,\max})\cap\ker(\overline{\partial}_{F,m,q-1,\min}^{\gamma,\gamma,t})\right)\cong\ker(\overline{\partial}^{\gamma,\gamma}_{F,m,q,\max})/\im(\overline{\partial}^{\gamma,\gamma}_{F,m,q-1,\max})$$ is a finite-dimensional vector space. Let us now consider the following $L^2$-decomposition $$L^2\Omega^{m,q}(\reg(X),F,\gamma,\tau)=\left(\ker(\overline{\partial}^{\gamma,\gamma}_{F,m,q,\max})\cap\ker(\overline{\partial}_{F,m,q-1,\min}^{\gamma,\gamma,t})\right)\oplus \im(\overline{\partial}_{F,m,q-1,\max}^{\gamma,\gamma})\oplus \im(\overline{\partial}_{F,m,q,\min}^{\gamma,\gamma,t}).$$ We already know that $\ker(\overline{\partial}_{F,m,q,\max}^{\gamma,\gamma})\cap\ker(\overline{\partial}_{F,m,q-1,\min}^{\gamma,\gamma,t})$ is finite dimensional. Since
$$
\begin{aligned}
\mathcal{D}(\overline{\eth}_{F,m,\mathrm{abs}})=\bigoplus_{q=0}^m \left(\mathcal{D}(\overline{\partial}_{F,m,q,\max}^{\gamma,\gamma})\cap \mathcal{D}(\overline{\partial}_{F,m,q-1,\min}^{\gamma,\gamma,t})\right)\\
\ker(\overline{\eth}_{F,m,\mathrm{abs}})=\bigoplus_{q=0}^m \left(\ker(\overline{\partial}_{F,m,q,\max}^{\gamma,\gamma})\cap \ker(\overline{\partial}_{F,m,q-1,\min}^{\gamma,\gamma,t})\right)\\
\mathrm{im}(\overline{\eth}_{F,m,\mathrm{abs}})=\bigoplus_{q=0}^m \left(\mathrm{im}(\overline{\partial}_{F,m,q,\max}^{\gamma,\gamma})\oplus \mathrm{im}(\overline{\partial}_{F,m,q-1,\min}^{\gamma,\gamma,t})\right)
\end{aligned}
$$
we know that $\ker(\overline{\eth}_{F,m,\mathrm{abs}})$ has finite dimension and $\im(\overline{\eth}_{F,m,\mathrm{abs}})$ is closed. Thus in order to conclude that $\overline{\eth}_{F,m,\mathrm{abs}}:L^2\Omega^{m,\bullet}(\reg(X),F,\gamma,\tau)\rightarrow L^2\Omega^{m,\bullet}(\reg(X),F,\gamma,\tau)$ has entirely discrete spectrum it is enough to prove that the corresponding Green operator is compact, see Prop. \ref{useful}. Thanks to \eqref{bdbd} this boils down to show that  the Green operators of $\overline{\partial}_{F,m,q,\max}^{\gamma,\gamma}$ and $\overline{\partial}_{F,m,q,\min}^{\gamma,\gamma,t}$ with respect to $\tau$:
\begin{equation}
\label{okm}
G_{\overline{\partial}_{F,m,q,\max}}^{\tau}:L^2\Omega^{m,q+1}(\reg(X),F,\gamma,\tau)\rightarrow L^2\Omega^{m,q}(\reg(X),F,\gamma,\tau)
\end{equation}
and 
\begin{equation}
\label{mko}
G_{\overline{\partial}_{F,m,q-1,\min}^t}^{\tau}:L^2\Omega^{m,q-1}(\reg(X),F,\gamma,\tau)\rightarrow L^2\Omega^{m,q}(\reg(X),F,\gamma,\tau)
\end{equation}
 are both compact for each $q=0,...,m$. Note that the compactness of \eqref{mko} follows from the compactness of \eqref{okm}. Indeed $\overline{\partial}_{F,m,q-1,\min}^{\gamma,\gamma,t}=(\overline{\partial}_{F,m,q-1,\max}^{\gamma,\gamma})^*$ and consequently $G_{\overline{\partial}_{F,m,q-1,\min}^t}^{\tau}=(G_{\overline{\partial}_{F,m,q-1,\max}}^{\tau})^*.$ Thus we are left to prove the compactness of \eqref{okm}. To this aim we point out that since $\tau$  and $\tau'$ are quasi-isometric we have an equality of topological vector spaces $L^2\Omega^{m,q}(\reg(X),F,\gamma,\tau')=L^2\Omega^{m,q}(\reg(X),F,\gamma,\tau)$. In particular the identity map $$\mathrm{Id}:L^2\Omega^{m,q}(\reg(X),F,\gamma,\tau')\rightarrow L^2\Omega^{m,q}(\reg(X),F,\gamma,\tau)$$ is bijective, continuous with continuous inverse. Moreover it is clear that the identity induces a continuous isomorphism, that we denote by $K_q$, with continuous inverse $$L^2\Omega^{m,q}(\reg(X),F,\gamma,\tau')\supset \im(\overline{\partial}_{F,m,q-1,\max}^{\gamma,\gamma})\stackrel{K_q}{\longrightarrow } \im(\overline{\partial}_{F,m,q-1,\max}^{\gamma,\gamma})\subset L^2\Omega^{m,q}(\reg(X),F,\gamma,\tau)$$
with $K_q$ defined as the restriction of $\mathrm{Id}$ on $\im(\overline{\partial}^{\gamma,\gamma}_{F,m,q-1,\mathrm{abs}})\subset L^2\Omega^{m,q}(\reg(X),F,\gamma,\tau')$. Furthermore let us introduce the map  $J_q$:
$$L^2\Omega^{m,q}(\reg(X),F,\gamma,\tau')\supset \im(\overline{\partial}_{F,m,q,\min}^{\gamma,\gamma,t})\stackrel{J_q}{\longrightarrow } \im(\overline{\partial}_{F,m,q,\min}^{\gamma,\gamma,t})\subset L^2\Omega^{m,q}(\reg(X),F,\gamma,\tau)$$ defined
as  $J_q:=\pi^{\tau}_q\circ \mathrm{Id}$  with $\pi^{\tau}_q$ the orthogonal projection $\pi^{\tau}_q:L^2\Omega^{m,q}(\reg(X),F,\gamma,\tau)\rightarrow  \im(\overline{\partial}_{F,m,q,\min}^{\gamma,\gamma,t}).$
Using again the fact that $\tau$ and $\tau'$ are quasi-isometric it is not difficult to check that  $J_q$ is  bounded, bijective with bounded inverse and that 
$$G_{\overline{\partial}_{F,m,q,\max}}^{\tau}|_{\im(\overline{\partial}_{F,m,q,\max}^{\gamma,\gamma})}:\im(\overline{\partial}_{F,m,q,\max}^{\gamma,\gamma})\rightarrow L^2\Omega^{m,q}(\reg(X),F,\gamma,\tau)$$ equals $$J_q\circ G_{\overline{\partial}_{F,m,q,\max}}^{\tau'}\circ K_{q+1}^{-1}:\im(\overline{\partial}^{\gamma,\gamma}_{F,m,q,\max})\rightarrow L^2\Omega^{m,q}(\reg(X),F,\gamma,\tau).$$ Since both $J_q$ and $K_q^{-1}$ are continuous and $$ G_{\overline{\partial}_{F,m,q,\max}^{\tau'}}:L^2\Omega^{m,q+1}(\reg(X),F,\gamma,\tau')\rightarrow L^2\Omega^{m,q}(\reg(X),F,\gamma,\tau')$$ is compact we obtain that $$G_{\overline{\partial}_{F,m,q,\max}}^{\tau}|_{\im(\overline{\partial}_{F,m,q,\max}^{\gamma,\gamma})}:\im(\overline{\partial}_{F,m,q,\max}^{\gamma,\gamma})\rightarrow L^2\Omega^{m,q}(\reg(X),F,\gamma,\tau)$$ is compact. Finally this implies immediately that also \eqref{okm} is compact.
\end{proof}

We have now the main result of this paper.

\begin{teo}
\label{tata}
Let $(X,\gamma)$ be a compact and irreducible Hermitian complex space of complex dimension $m$ such that $\dim(\sing(X))=0$. Let $(F,\tau)\rightarrow X$ be a Hermitian holomorphic vector bundle over $X$ which satisfies \eqref{trappolo}.
Let $\pi:M\rightarrow X$ be a resolution of $X$ and let $E\rightarrow M$ be the holomorphic vector bundle defined as $E:=\pi^*F$. We then have the following equality in $KK_0(C(X),\mathbb{C})$: $$\pi_*[\overline{\eth}_{E,m}]=[\overline{\eth}_{F,m,\mathrm{abs}}]\in KK_0(C(X),\mathbb{C}).$$
\end{teo}

In order to prove the above theorem we need some preliminary results.
\begin{lemma}
\label{imago}
Let $(X,\gamma)$ be a compact and irreducible Hermitian complex space of complex dimension $m$ with $\dim(\sing(X))=0$. Let $(F,\tau)\rightarrow X$ be a Hermitian holomorphic vector bundle over $X$ which satisfies \eqref{trappolo}. Let $\pi:M\rightarrow X$ be a resolution of $X$ and let $g$ be an arbitrarily fixed Hermitian metric on $M$. Let $E\rightarrow M$ be the holomorphic vector bundle defined as $E:=\pi^*F$ and let $h=\pi^*\gamma$, $\rho:=\pi^*\tau$ and $A:=\pi^{-1}(\reg(X))$. Consider the operators $$\overline{\partial}_{E,m,q,\max}^{g,h}:L^2\Omega^{m,q}(A,E|_A,g|_A,\rho|_A)\rightarrow L^2\Omega^{m,q+1}(A,E|_A,h|_A,\rho|_A)$$ and $$\overline{\partial}_{E,m,q,\max}^{h,h}:L^2\Omega^{m,q}(A,E|_A,h|_A,\rho|_A)\rightarrow L^2\Omega^{m,q+1}(A,E|_A,h|_A,\rho|_A).$$ Then for each $q=0,...,m$  the following equalities hold in $L^2\Omega^{m,q+1}(A,E|_A,h|_A,\rho|_A)$: $$\im(\overline{\partial}_{E,m,q,\max}^{h,h})=\overline{\im(\overline{\partial}_{E,m,q,\max}^{h,h})}=\overline{\im(\overline{\partial}_{E,m,q,\max}^{g,h})}=\im(\overline{\partial}_{E,m,q,\max}^{g,h}).$$ Moreover $$\dim (H^{m,q}_{\overline{\partial}}(M,E))=\dim (H^{m,q}_{2,\overline{\partial}_{\max}}(A,E|_A,h|_A,\rho|_A)).$$
\end{lemma}

\begin{proof}
First, we note that when $F=\reg(X)\times \mathbb{C}^n$ and $\tau=\upsilon_{\mathrm{std}}$ the above chain of equalities is an immediate consequence of \cite[Th. 1.5]{JRu}. We now tackle the general case. To this aim we introduce the following presheaves  $C^{m,q}_{F,\gamma}$ on $X$ given by the assignments  
\begin{equation}
\label{presheaf}
C^{m,q}_{F,\gamma}(U):=\{\mathcal{D}(\overline{\partial}_{F,m,q,\max}^{\gamma,\gamma})\ \text{on}\ \reg(U)\};
\end{equation}
in other words to every open subset $U$ of $X$ we assign the maximal domain of $\overline{\partial}_{F,m,q}$ over $\reg(U)$ with respect to $F|_U$, $\gamma|_{\reg(U)}$ and $\tau|_{\reg(U)}$.  We denote by $\mathcal{C}^{m,q}_{F,\gamma}$ the corresponding sheafification. Finally, let $(\mathcal{C}^{m,\bullet}_{F,\gamma},\overline{\partial}_{F,m,\bullet}^{\gamma,\gamma})$ be the complex of sheaves where the action of $\overline{\partial}_{F,m,\bullet}^{\gamma,\gamma}$ is understood in the distributional sense. Let $\sigma$ be the Hermitian metric on $\reg(X)$ defined as $\sigma:=((\pi|_A)^{-1})^*g$. Let us consider the corresponding complex of sheaves $(\mathcal{C}^{m,\bullet}_{F,\sigma},\overline{\partial}_{F,m,\bullet}^{\sigma,\sigma})$. Thanks to Prop. \ref{Pp} it is easy to check that the continuous inclusion $I:L^2\Omega^{m,q}(\reg(X),F,\gamma,\tau)\hookrightarrow L^2\Omega^{m,q}(\reg(X),F,\sigma,\tau)$ gives rise to a morphism of sheaves 
\begin{equation}
\label{mor}
\mathcal{I}:(\mathcal{C}^{m,\bullet}_{F,\gamma},\overline{\partial}_{F,m,\bullet}^{\gamma,\gamma})\rightarrow (\mathcal{C}^{m,\bullet}_{F,\sigma},\overline{\partial}_{F,m,\bullet}^{\sigma,\sigma}).
\end{equation}
Let  $\mathcal{K}_M(E)$ be the sheaf of holomorphic sections of $K_M\otimes E\rightarrow M$. Since we assumed \eqref{trappolo} we can argue as in the proof of \cite[Th. 1.5]{JRu} to show that both $(\mathcal{C}^{m,\bullet}_{F,\gamma},\overline{\partial}_{F,m,\bullet}^{\gamma,\gamma})$ and $(\mathcal{C}^{m,\bullet}_{F,\sigma},\overline{\partial}_{F,m,\bullet}^{\sigma,\sigma})$ are fine resolutions of $\pi_*(\mathcal{K}_M(E))$. This in turn implies that the morphism \eqref{mor} induces an isomorphism, still denoted with $\mathcal{I}$, between the cohomology groups: \begin{equation}
\label{isocom}
\mathcal{I}:H^q(X,\mathcal{C}^{m,\bullet}_{F,\gamma}(X))\rightarrow H^q(X,\mathcal{C}^{m,\bullet}_{F,\sigma}(X)), q=0,...,m
\end{equation}
where, by $H^q(X,\mathcal{C}^{m,\bullet}_{F,\gamma}(X))$ and $H^q(X,\mathcal{C}^{m,\bullet}_{F,\sigma}(X))$ we mean the cohomology groups of the complexes of global sections of $\mathcal{C}^{m,\bullet}_{F,\gamma}$ and $\mathcal{C}^{m,\bullet}_{F,\sigma}$, that is the cohomology of the complexes:
$$0\rightarrow \mathcal{C}^{m,0}_{F,\gamma}(X)\stackrel{\overline{\partial}_{F,m,0}^{\gamma,\gamma}}{\rightarrow}...\stackrel{\overline{\partial}_{F,m,m-1}^{\gamma,\gamma}}{\rightarrow}\mathcal{C}^{m,m}_{F,\gamma}(X)\rightarrow 0\quad \mathrm{and}\quad 0\rightarrow \mathcal{C}^{m,0}_{F,\sigma}(X)\stackrel{\overline{\partial}_{F,m,0}^{\sigma,\sigma}}{\rightarrow}...\stackrel{\overline{\partial}_{F,m,m-1}^{\sigma,\sigma}}{\rightarrow}\mathcal{C}^{m,m}_{F,\sigma}(X)\rightarrow 0.$$ 
It is clear that on $X$ we have the equalities $\mathcal{C}^{m,q}_{F,\gamma}(X)=\{\mathcal{D}(\overline{\partial}_{F,m,q,\max}^{\gamma,\gamma})\ \mathrm{on}\ \reg(X)\}$  and analogously $\mathcal{C}^{m,q}_{F,\sigma}(X)=\{\mathcal{D}(\overline{\partial}_{F,m,q,\max}^{\sigma,\sigma})\ \mathrm{on}\ \reg(X)\}$ which in turn imply the equalities $$H^q(X,\mathcal{C}^{m,\bullet}_{F,\gamma}(X))=H^{m,q}_{\overline{\partial}_{F,\max}^{\gamma,\gamma}}(\reg(X),F,\sigma,\tau)$$ $$H^q(X,\mathcal{C}^{m,\bullet}_{F,\sigma}(X)) =H^{m,q}_{\overline{\partial}_{F,\max}^{\sigma,\sigma}}(\reg(X),F,\sigma,\tau).$$ Therefore, by the fact that \eqref{isocom} is an isomorphism we obtain that the continuous inclusion $I:L^2\Omega^{m,q}(\reg(X),F,\gamma,\tau)\hookrightarrow L^2\Omega^{m,q}(\reg(X),F,\sigma,\tau)$ induces an isomorphism between the $L^2$-$\overline{\partial}$ cohomology groups 
\begin{equation}
\label{aboveiso}
H^{m,q}_{2,\overline{\partial}_{\max}}(\reg(X),F,\gamma,\tau)\cong H^{m,q}_{2,\overline{\partial}_{\max}}(\reg(X),F,\sigma,\tau).
\end{equation} 
By using \eqref{aboveiso} we get immediately that $\im(\overline{\partial}_{F,m,q,\max}^{\gamma,\gamma})=\im(\overline{\partial}_{F,m,q,\max}^{\sigma,\gamma})$ and therefore $\im(\overline{\partial}_{E,m,q,\max}^{h,h})=\im(\overline{\partial}_{E,m,q,\max}^{g,h})$, as required. Moreover since $H^{m,q}_{2,\overline{\partial}_{\max}}(\reg(X),F,\gamma,\tau)$ is finite dimensional we have that $\im(\overline{\partial}_{E,m,q,\max}^{h,h})$ is closed. Hence $\im(\overline{\partial}_{E,m,q,\max}^{g,h})=\overline{\im(\overline{\partial}_{E,m,q,\max}^{g,h})}$ as required. Finally as remarked above we know that both the complexes $(\mathcal{C}^{m,\bullet}_{F,\gamma},\overline{\partial}_{F,m,\bullet}^{\gamma,\gamma})$ and $(\mathcal{C}^{m,\bullet}_{F,\sigma},\overline{\partial}_{F,m,\bullet}^{\sigma,\sigma})$ are fine resolutions of the sheaf $\pi_*(\mathcal{K}_M(E))$. Hence 
$$
\begin{aligned}
\dim(H^{m,q}_{2,\overline{\partial}_{\max}}(A,E|_A,h|_A,\rho|_A))&=\dim(H^{m,q}_{2,\overline{\partial}_{\max}}(\reg(X),F,\gamma,\tau))=\dim(H^q(X,\mathcal{C}^{m,\bullet}_{F,\gamma}(X)))\\
&=\dim(H^q(X,\mathcal{C}^{m,\bullet}_{F,\sigma}(X)))=\dim(H^{m,q}_{2,\overline{\partial}_{\max}}(\reg(X),F,\sigma,\tau))\\
&=\dim(H^{m,q}_{2,\overline{\partial}_{\max}}(A,E|_A,g|_A,\rho|_A))=\dim(H^{m,q}_{\overline{\partial}}(M,E))
\end{aligned}
$$ where the last equality follows by Prop. \ref{sameop}. The proof is thus complete.
\end{proof}

In order to continue we need to introduce various tools. Let $\pi:M\rightarrow X$ be a resolution of $X$ with $A:=\pi^{-1}(\reg(X))$. As in the previous section we consider $M\times [0,1]$ and the canonical projection  $p:M\times [0,1]\rightarrow M$. Let $g_s\in C^{\infty}(M\times [0,1], p^*T^*M\otimes p^*T^*M)$ be a smooth section of $p^*T^*M\otimes p^*T^*M\rightarrow M\times[0,1]$ such that:
\begin{enumerate}
\item $g_s(JX,JY)=g_s(X,Y)$ for any $X,Y\in \mathfrak{X}(M)$ and $s\in [0,1]$;
\item $g_s$ is a Hermitian metric on $M$ for any $s\in (0,1]$;
\item $g_1=g$ and $g_0=h$ with $h:=\pi^*\gamma$;
\item There exists a positive constant $\frak{a}$ such that $g_0\leq \frak{a}g_s$ for each $s\in [0,1]$.
\end{enumerate}
Let us also denote with $p:\reg(X)\times [0,1]\rightarrow \reg(X)$ the left projection and let $$\sigma_s\in C^{\infty}(\reg(X)\times [0,1], p^*T^*\reg(X)\otimes p^*T^*\reg(X))$$ be the smooth section of $p^*T^*\reg(X)\otimes p^*T^*\reg(X)\rightarrow \reg(X)$ induced by $g_s$ and $\pi$.
Note that   $\sigma_s$ is the Hermitian metric over $\reg(X)$ given by $\sigma_s:=((\pi|_{A})^{-1})^*g_s$ for each $s\in [0,1]$. In particular $\sigma_1=((\pi|_{A})^{-1})^*g$ whereas $\sigma_0=\gamma$. Let us consider the following complex 

\begin{align}
& \nonumber L^2\Omega^{m,0}(\reg(X),F,\sigma_1,\tau)\stackrel{\overline{\partial}^{\sigma_1,\sigma_1}_{F,m,0}}{\longrightarrow}......\stackrel{\overline{\partial}^{\sigma_1,\sigma_1}_{F,m,q-2}}{\longrightarrow}L^2\Omega^{m,q-1}(\reg(X),F,\sigma_1,\tau)\stackrel{\overline{\partial}^{\sigma_1,\sigma_s}_{F,m,q-1,\max}}{\longrightarrow}L^2\Omega^{m,q}(\reg(X),F,\sigma_s,\tau)\\
& \label{nev} \stackrel{\overline{\partial}^{\sigma_s,\gamma}_{F,m,q,\max}}{\longrightarrow}
L^2\Omega^{m,q+1}(\reg(X),F,\gamma,\tau)\stackrel{\overline{\partial}^{\gamma,\gamma}_{F,m,q+1,\max}}{\longrightarrow}......\stackrel{\overline{\partial}^{\gamma,\gamma}_{F,m,m-1,\max}}{\longrightarrow}L^2\Omega^{m,m}(\reg(X),F,\gamma,\tau).
\end{align}

Let $\chi^{m,q}_s\in C(\reg(X)\times [0,1], \mathrm{End}(p^*\Lambda^{m,q}(\reg(X))\otimes p^*F))$ be the family of endomorphisms defined as $\chi_s^{m,q}:=(\pi^*)^{-1}\circ\Psi^{m,q}_s\circ \pi^*$, see \eqref{isoiso} for the definition of $\Psi_s^{m,q}$. Clearly $\chi_s^{m,q}:L^2\Omega^{m,q}(\reg(X),F,\sigma_s,\tau)\rightarrow L^2\Omega^{m,q}(\reg(X),F,\sigma_1,\tau)$ is an isometry for each $s\in [0,1]$. Let also define the following family of endomorphisms $$\chi^{m,\bullet}_s\in C(\reg(X)\times [0,1], \mathrm{End}(p^*\Lambda^{m,\bullet}(\reg(X))\otimes p^*F)),\quad \chi_s^{m,\bullet}:= \bigoplus_{r=0}^m\chi_s^{m,r}.$$ It is clear that $\chi_s^{m,\bullet}:L^2\Omega^{m,\bullet}(\reg(X),F,\sigma_s,\tau)\rightarrow L^2\Omega^{m,\bullet}(\reg(X),F,\sigma_1,\tau)$ is an isometry. Following \eqref{ucomplex} we introduce the following complex 
\begin{align}
\label{uucomplex}
L^2\Omega^{m,0}(\reg(X),F,\sigma_1,\tau)&\stackrel{\overline{\partial}^{\sigma_1,\sigma_1}_{F,m,0}}{\longrightarrow}......\stackrel{\overline{\partial}^{\sigma_1,\sigma_1}_{F,m,q-2}}{\longrightarrow}L^2\Omega^{m,q-1}(\reg(X),F,\sigma_1,\tau)\stackrel{D^{\sigma_1,\sigma_s}_{m,q-1}}{\longrightarrow}L^2\Omega^{m,q}(\reg(X),F,\sigma_1,\tau)\\
\nonumber & \stackrel{D^{\sigma_s,\gamma}_{m,q}}{\longrightarrow} L^2\Omega^{m,q+1}(\reg(X),F,\sigma_1,\tau)\stackrel{D^{\gamma,\gamma}_{m,q+1}}{\longrightarrow}......\stackrel{D^{\gamma,\gamma}_{m,m-1}}{\longrightarrow}L^2\Omega^{m,m}(\reg(X),F,\sigma_1,\tau)
\end{align}
with $D^{\sigma_1,\sigma_s}_{m,q-1}:= \chi_{s}^{m,q}\circ \overline{\partial}_{F,m,q-1,\max}^{\sigma_1,\sigma_s}$, $D^{\sigma_s,\gamma}_{m,q}:= \chi_{0}^{m,q+1}\circ \overline{\partial}_{E,m,q,\max}^{\sigma_s,\gamma}\circ (\chi_s^{m,q})^{-1}$ and  $D^{\gamma,\gamma}_{m,r}:= \chi_{0}^{m,r+1}\circ \overline{\partial}_{F,m,r,\max}^{\gamma,\gamma}\circ (\chi_0^{m,r})^{-1}$ for each  $r=q+1,...,m$. Finally let 
\begin{equation}
\label{uurolledup}
Q^{m,q}_s:L^2\Omega^{m,\bullet}(\reg(X),F,\sigma_1,\tau)\rightarrow L^2\Omega^{m,\bullet}(\reg(X),F,\sigma_1,\tau)
\end{equation}
be the rolled-up operator of the complex \eqref{uucomplex}. Note that $Q_1^{m,q-1}=Q_{0}^{m,q}$ for each $q\in \{1,...,m\}$, see \eqref{urolledup}. The next Lemma is the key tool to prove Th. \ref{tata}.

\begin{lemma}
\label{parameter}
In the setting of Theorem \ref{tata} the following properties hold true
\begin{enumerate}
\item The triplet $$(L^2\Omega^{m,\bullet}(\reg(X),F,\sigma_1,\tau),m,Q^{m,q}_s)$$ is an even unbounded Fredholm module over $C(X)$. We denote with $[Q^{m,q}_s]$ the corresponding class in $KK_0(C(X),\mathbb{C})$.
\item For each $q\in {0,...,m}$ and $s\in [0,1]$, we have the following equality in $KK_0(C(X),\mathbb{C})$: $$[Q^{m,q}_s]=[Q^{m,q}_1].$$
\end{enumerate}
\end{lemma}
\begin{proof}
Let $f\in C(X)$. Since in particular $f\in L^{\infty}(X)$ we obtain immediately that  $m_f:L^2\Omega^{m,r}(\reg(X),F,\sigma_s,\tau)\rightarrow L^2\Omega^{m,r}(\reg(X),F,\sigma_s,\tau)$ is bounded for each $r\in \{0,...,m\}$ and $s\in [0,1]$. Let us now fix $S_c(X)$ as a dense $*$-subalgebra of $C(X)$. Clearly we have $\overline{\partial}f\in \Omega_c^{0,1}(\reg(X))$ and therefore the map  $\overline{\partial}f\wedge $ given by \begin{equation}
\label{www}
L^2\Omega^{m,r}(\reg(X),F,\sigma_{s_1},\tau)\ni \eta\mapsto \overline{\partial}f\wedge\eta \in L^2\Omega^{m,r+1}(\reg(X),F,\sigma_{s_2},\tau)
\end{equation}
is continuous for any choice of $r\in \{0,...,m\}$ and $s_1,s_2\in [0,1]$.  Consequently the adjoint map $(\overline{\partial}f\wedge)^*$ given by $$L^2\Omega^{m,r+1}(\reg(X),F,\sigma_{s_2},\tau)\ni \varphi\mapsto (\overline{\partial}f\wedge)^*\varphi \in L^2\Omega^{m,r}(\reg(X),F,\sigma_{s_1},\tau)$$ is continuous as well. Note that we can write the above map $(\overline{\partial}f\wedge)^*$ as $(U_{s_1}^{m,r})^{-1}\circ i_{(\nabla_1f)^{0,1}} \circ U_{s_2}^{m,r+1}$ with $$U_s^{m,r}:=(\pi^*)^{-1}\circ S_s^{m,r}\circ \pi^*,$$ $$S_s^{m,r}\in C^{\infty}(A\times [0,1], \mathrm{End}(p^*\Lambda^{m,q}(A)\otimes p^*E))$$ defined in the proof of Lemma \ref{lemma3},  $\nabla_1 f$ the gradient of $f$ w.r.t. $\sigma_1$,  $(\nabla_1 f)^{0,1}$ the $(0,1)$ component of $\nabla_1 f$ and  $i_{(\nabla_1f)^{0,1}}$ the interior multiplication w.r.t. $(\nabla_1 f)^{0,1}$.
Since $f\in L^{\infty}(X)$ and $df\in \Omega_c^1(\reg(X))$ we can argue as in \cite[Prop. 2.3]{BeiSym} to conclude that $m_f$ preserves the domain of
\begin{equation}
\label{preserve1}
\overline{\partial}_{F,m,r,\max}^{\sigma_1,\sigma_2}:L^2\Omega^{m,r}(\reg(X),F,\sigma_{s_1},\tau)\rightarrow L^2\Omega^{m,r+1}(\reg(X),F,\sigma_{s_2},\tau).
\end{equation}

Moreover it is also easy to see that $m_f$ preserves the domain of 
\begin{equation}
\label{preserve2}
\overline{\partial}_{F,m,r,\min}^{\sigma_1,\sigma_2,t}:L^2\Omega^{m,r+1}(\reg(X),F,\sigma_{s_1},\tau)\rightarrow L^2\Omega^{m,r}(\reg(X),F,\sigma_{s_2},\tau).
\end{equation}
Indeed if $\eta$ lies in the domain of \eqref{preserve2} and $\{\eta_k\}_{k\in \mathbb{N}}\in \Omega_c^{r+1}(\reg(X),F)$ is a sequence converging to $\eta$ in the graph norm of $\overline{\partial}_{F,m,r}^{\sigma_1,\sigma_2,t}$, then $f\eta_k\rightarrow f\eta$ in $L^2\Omega^{m,r+1}(\reg(X),F,\sigma_2,\tau)$ as $k\rightarrow \infty$ and  $\overline{\partial}_{F,m,r}^{\sigma_1,\sigma_2,t}(f \eta_k)=f\overline{\partial}_{F,m,r}^{\sigma_1,\sigma_2,t}\eta_k-(\overline{\partial}f\wedge)^*\eta_k$ $\rightarrow f\overline{\partial}_{F,m,r,\min}^{\sigma_1,\sigma_2,t}\eta-(\overline{\partial}f\wedge)^*\eta$ in $L^2\Omega^{m,r}(\reg(X),F,\sigma_2,\tau)$ as $k\rightarrow \infty$. Hence we can conclude that also $f\eta$ lies in the domain of \eqref{preserve2}. Consider now the complex \eqref{nev}
$$\begin{aligned}
& \nonumber L^2\Omega^{m,0}(\reg(X),F,\sigma_1,\tau)\stackrel{\overline{\partial}^{\sigma_1,\sigma_1}_{F,m,0}}{\longrightarrow}...\stackrel{\overline{\partial}^{\sigma_1,\sigma_1}_{F,m,q-2}}{\longrightarrow}L^2\Omega^{m,q-1}(\reg(X),F,\sigma_1,\tau)\stackrel{\overline{\partial}^{\sigma_1,\sigma_s}_{F,m,q-1,\max}}{\longrightarrow}L^2\Omega^{m,q}(\reg(X),F,\sigma_s,\tau)\\
&  \stackrel{\overline{\partial}^{\sigma_s,\gamma}_{F,m,q,\max}}{\longrightarrow}L^2\Omega^{m,q+1}(\reg(X),F,\gamma,\tau)\stackrel{\overline{\partial}^{\gamma,\gamma}_{F,m,q+1,\max}}{\longrightarrow}...\stackrel{\overline{\partial}^{\gamma,\gamma}_{F,m,m-1,\max}}{\longrightarrow}L^2\Omega^{m,m}(\reg(X),F,\gamma,\tau)
\end{aligned}$$
and let $L_s^{m,q}$ be the corresponding rolled-up operator. By the above discussion it is now clear that $m_f$ preserves the domain of $L_s^{m,q}$ and that $[L_s^{m,q},m_f]=\overline{\partial}f\wedge-(\overline{\partial}f\wedge)^*$ is continuous for each $f\in S_c(X)$. Since $\chi_s^{m,r}$ is a vector bundle isometric endomorphism we have $\chi_s^{m,r}\circ m_f=m_f\circ \chi_s^{m,r}$ and $(\chi_s^{m,r})^{-1}\circ m_f=m_f\circ (\chi_s^{m,r})^{-1}$ for each $s\in [0,1]$, $r=0,...,m$ and $f\in C(X)$. Therefore, using the above arguments, we can also conclude  that for each $f\in S_c(X)$ the operator $m_f$ preserves the domain of $Q_s^{m,q}$ and  $$[Q_s^{m,q},m_f]:L^2\Omega^{m,\bullet}(\reg(X),F,\sigma_1,\tau)\rightarrow L^2\Omega^{m,\bullet}(\reg(X),F,\sigma_1,\tau)$$ is continuous. Furthermore the operator $Q_s^{m,q}$ is unitarily equivalent to the operator defined in \eqref{urolledup} through the isometry $\pi^*:L^2\Omega^{m,\bullet}(\reg(X),F,\sigma_1,\tau)\rightarrow L^2\Omega^{m,\bullet}(A,E,g_1,\rho)$. By Lemma \ref{lemma10} we can thus conclude that $Q^{m,q}_s:L^2\Omega^{m,\bullet}(\reg(X),F,\sigma_1,\tau)\rightarrow L^2\Omega^{m,\bullet}(\reg(X),F,\sigma_1,\tau)$ has entirely discrete spectrum and this is equivalent to the compactness of the resolvent. Finally it is clear that the grading of $L^2\Omega^{m,\bullet}(\reg(X),F,\sigma_1,\tau)$, which is induced by the splitting in $L^2$ $F$-valued $(m,\bullet)$-forms with even/odd anti-holomorphic degree, commutes with $m$ and anti-commutes with $Q_s^{m,q}$. We can therefore conclude that  the triplet $$(L^2\Omega^{m,\bullet}(\reg(X),F,\sigma_1,\tau),m,Q^{m,q}_s)$$ is an even unbounded Fredholm module over $C(X)$. This concludes the proof of the first part. Now we tackle the second part of the proof, and to do that, we use Prop. \ref{sameclass}. Note that for each $s\in (0,1]$ the metrics $\sigma_s$ and $\sigma_1$ are quasi-isometric. Hence the continuity of the map $(0,1]\rightarrow B(L^2\Omega^{m,\bullet}(\reg(X),F,\sigma_1,\tau))$ given by $s\mapsto (Q^{m,q}_s+i)^{-1}$ with respect to the operator norm follows by arguing as in \cite{Hilsum}. As remarked above we have $Q_s^{m,q}=(\pi^*)^{-1}\circ P^{m,q}_s\circ \pi^*$ with $P^{m,q}_s$ defined in \eqref{urolledup} and $\pi^*:L^2\Omega^{m,\bullet}(\reg(X),F,\sigma_1,\tau)\rightarrow L^2\Omega^{m,\bullet}(A,E,g_1,\rho)$ the isometry induced by the resolution map $\pi:M\rightarrow X$. Hence we have 

$$
\begin{aligned}
\|(Q^{m,q}_s+i)^{-1}-(Q^{m,q}_0+i)^{-1}\|_{\mathrm{op}}&=\|\pi^*\circ (P^{m,q}_s+i)^{-1}\circ (\pi^*)^{-1}-\pi^*\circ (P^{m,q}_0+i)^{-1}\circ (\pi^*)^{-1}\|_{\mathrm{op}}\\
& =\|(P^{m,q}_s+i)^{-1}- (P^{m,q}_0+i)^{-1}\|_{\mathrm{op}}.
\end{aligned}
$$
Thanks to Lemma s\ref{tris} and \ref{imago}, we are in a position to apply Th. \ref{resolv}, and hence we obtain $$\lim_{s\rightarrow 0}\|(P^{m,q}_s+i)^{-1}-(P^{m,q}_0+i)^{-1}\|_{\mathrm{op}}=0.$$ We can thus conclude that the map $[0,1]\rightarrow B(L^2\Omega^{m,\bullet}(\reg(X),F,\sigma_1,\tau))$ given by $s\mapsto (Q^{m,q}_s+i)^{-1}$ is continuous with respect to the operator norm. This settles the second requirement of Prop. \ref{sameclass}. We are left to show that for each $f\in S_c(X)$ the map $[0,1]\rightarrow B(L^2\Omega^{m,\bullet}(\reg(X),F,\sigma_1,\tau))$ given by $s\mapsto [Q^{m,q}_s,m_f]$ is continuous with respect to the strong operator topology. To this aim it is enough to show that for any arbitrarily fixed $r=0,...,m$ and $\eta\in L^2\Omega^{m,r}(\reg(X),F, \sigma_1, \tau)$ we have 
\begin{equation}
\label{stronglimit}
\lim_{s\rightarrow 0}[Q^{m,q}_s,m_f]\eta=[Q^{m,q}_0,m_f]\eta\quad \mathrm{in}\ L^2\Omega^{m,r}(\reg(X),F,\sigma_1,\tau).
\end{equation}
 
If $0\leq r\leq q-2$ then $$[Q^{m,q}_s,m_f]\eta=[\overline{\partial}_{F,m,r}^{\sigma_1,\sigma_1}+\overline{\partial}_{F,m,r-1}^{\sigma_1,\sigma_1,t},m_f]\eta=[Q^{m,q}_0,m_f]\eta$$ for each $s\in [0,1]$. Thus \eqref{stronglimit} is obviously satisfied. If $r=q-1$ then 
$$
\begin{aligned}
[Q^{m,q}_s,m_f]\eta&=[D_{F,m,q-1}^{\sigma_1,\sigma_s}+\overline{\partial}_{F,m,q-2}^{\sigma_1,\sigma_1,t},m_f]\eta\\
&=[\chi_{s}^{m,q}\circ \overline{\partial}_{F,m,q-1,\max}^{\sigma_1,\sigma_s}+\overline{\partial}_{F,m,q-2}^{\sigma_1,\sigma_1,t},m_f]\eta\\
&=[\chi_{s}^{m,q}\circ \overline{\partial}_{F,m,q-1,\max}^{\sigma_1,\sigma_s},m_f]\eta+[\overline{\partial}_{F,m,q-2}^{\sigma_1,\sigma_1,t},m_f]\eta\\
&=\chi_{s}^{m,q}\circ[\overline{\partial}_{F,m,q-1,\max}^{\sigma_1,\sigma_s},m_f]\eta+[\overline{\partial}_{F,m,q-2}^{\sigma_1,\sigma_1,t},m_f]\eta\\
& =\chi_{s}^{m,q}(\overline{\partial}f\wedge \eta)+[\overline{\partial}_{F,m,q-2}^{\sigma_1,\sigma_1,t},m_f]\eta.
\end{aligned}
$$
Note that the term $[\overline{\partial}_{F,m,q-2}^{\sigma_1,\sigma_1,t},m_f]\eta$ is independent on $s$ while the equality $$\lim_{s\rightarrow 0}\chi_{s}^{m,q}(\overline{\partial}f\wedge \eta)=\chi_{0}^{m,q}(\overline{\partial}f\wedge \eta)\quad \mathrm{in}\ L^2\Omega^{m,q}(\reg(X),F,\sigma_1,\tau)$$  follows easily by the Lebesgue dominated convergence theorem and the fact that $\overline{\partial}f\in \Omega^{0,1}_c(\reg(X))$ and  $\chi^{m,q}_s\in C(\reg(X)\times [0,1], \mathrm{End}(p^*\Lambda^{m,q}(\reg(X))\otimes p^*F))$. Since $$[Q^{m,q-1}_0,m_f]\eta=\chi_{0}^{m,q}(\overline{\partial}f\wedge \eta)+[\overline{\partial}_{F,m,q-2}^{\sigma_1,\sigma_1,t},m_f]\eta$$ we can conclude that \eqref{stronglimit} also holds true in the case $r=q-1$. If $r=q$ then 
$$
\begin{aligned}
[Q^{m,q}_s,m_f]\eta&=[D_{F,m,q}^{\sigma_s,\gamma}+(D_{F,m,q-1}^{\sigma_1,\sigma_s})^*,m_f]\eta\\
&=[\chi_{0}^{m,q+1}\circ \overline{\partial}_{F,m,q,\max}^{\sigma_s,\gamma}\circ (\chi_s^{m,q})^{-1}+\overline{\partial}_{F,m,q-1,\min}^{\sigma_1,\sigma_s,t}\circ (\chi_s^{m,q})^{-1},m_f]\eta\\
&=[\chi_{0}^{m,q+1}\circ \overline{\partial}_{F,m,q,\max}^{\sigma_s,\gamma}\circ (\chi_s^{m,q})^{-1},m_f]\eta+[\overline{\partial}_{F,m,q-1,\min}^{\sigma_1,\sigma_s,t}\circ (\chi_s^{m,q})^{-1},m_f]\eta\\
&=\chi_{0}^{m,q+1}\circ[ \overline{\partial}_{F,m,q,\max}^{\sigma_s,\gamma},m_f]\circ (\chi_s^{m,q})^{-1}\eta+[\overline{\partial}_{F,m,q-1,\min}^{\sigma_1,\sigma_s,t},m_f]\circ (\chi_s^{m,q})^{-1}\eta\\
& =\chi_{0}^{m,q+1}(\overline{\partial}f\wedge (\chi_s^{m,q})^{-1}\eta)+ i_{(\nabla_1f)^{0,1}} ( U_{s}^{m,q}((\chi_s^{m,q})^{-1}\eta)).
\end{aligned}
$$
Again, by  the fact that $\overline{\partial}f\in \Omega^{0,1}_c(\reg(X))$, $U^{m,q}_s$ and $(\chi^{m,q}_s)^{-1}$ $\in C(\reg(X)\times [0,1], \mathrm{End}(p^*\Lambda^{m,q}(\reg(X))\otimes p^*F))$ and the Lebesgue dominated convergence theorem we have
$$\lim_{s\rightarrow 0}\chi_{0}^{m,q+1}(\overline{\partial}f\wedge (\chi_s^{m,q})^{-1}\eta)=\chi_{0}^{m,q+1}(\overline{\partial}f\wedge (\chi_0^{m,q})^{-1}\eta)\quad \mathrm{in}\ L^2\Omega^{m,q+1}(\reg(X),F,\sigma_1,\tau)$$
and   $$\lim_{s\rightarrow 0 }i_{(\nabla_1f)^{0,1}} ( U_{s}^{m,q}((\chi_s^{m,q})^{-1}\eta))=i_{(\nabla_1f)^{0,1}} ( U_{0}^{m,q}((\chi_0^{m,q})^{-1}\eta))  \quad \mathrm{in}\ L^2\Omega^{m,q-1}(\reg(X),F,\sigma_1,\tau).$$ Since $$[Q^{m,q}_0,m_f]\eta=\chi_{0}^{m,q+1}(\overline{\partial}f\wedge (\chi_0^{m,q})^{-1}\eta)
+ i_{(\nabla_1f)^{0,1}} ( U_{0}^{m,q}((\chi_0^{m,q})^{-1}\eta))$$ we can conclude that \eqref{stronglimit} holds true also in the case $r=q$. If $r=q+1$ we have 
$$
\begin{aligned}
[Q^{m,q}_s,m_f]\eta&=[D_{F,m,q+1}^{\gamma,\gamma}+(D_{F,m,q}^{\sigma_s,\gamma})^*,m_f]\eta\\
&=[\chi_{0}^{m,q+2}\circ \overline{\partial}_{F,m,q+1,\max}^{\gamma,\gamma}\circ (\chi_0^{m,q+1})^{-1}+\chi_s^{m,q}\circ \overline{\partial}_{F,m,q,\min}^{\sigma_s,\gamma,t}\circ (\chi_0^{m,q+1})^{-1},m_f]\eta\\
&=[\chi_{0}^{m,q+2}\circ \overline{\partial}_{F,m,q+1,\max}^{\gamma,\gamma}\circ (\chi_0^{m,q+1})^{-1},m_f]\eta+[\chi_s^{m,q}\circ \overline{\partial}_{F,m,q,\min}^{\sigma_s,\gamma,t}\circ (\chi_0^{m,q+1})^{-1},m_f]\eta\\
&=\chi_{0}^{m,q+2}\circ [\overline{\partial}_{F,m,q+1,\max}^{\gamma,\gamma},m_f]\circ (\chi_0^{m,q+1})^{-1}\eta+\chi_s^{m,q}\circ [\overline{\partial}_{F,m,q,\min}^{\sigma_s,\gamma,t},m_f]\circ (\chi_0^{m,q+1})^{-1}\eta\\
& =\chi_{0}^{m,q+2}(\overline{\partial}f\wedge (\chi_0^{m,q+1})^{-1}\eta)+ \chi_s^{m,q}((U_s^{m,q})^{-1}(i_{(\nabla_1f)^{0,1}}(U_0^{m,q+1}((\chi_0^{m,q+1})^{-1}\eta)))).
\end{aligned}
$$
Note that the first term does not depend on $s$ whereas for the second term we have 
$$\lim_{s\rightarrow 0 }\chi_s^{m,q}((U_s^{m,q})^{-1}(i_{(\nabla_1f)^{0,1}}(U_0^{m,q+1}((\chi_0^{m,q+1})^{-1}\eta))))=\chi_0^{m,q}((U_0^{m,q})^{-1}(i_{(\nabla_1f)^{0,1}}(U_0^{m,q+1}((\chi_0^{m,q+1})^{-1}\eta))))\quad \mathrm{in}$$ in  $L^2\Omega^{m,q}(\reg(X),F,\sigma_1,\tau)$ for the same reasons explained in the previous cases. Since for $r=q+1$ we have $$[Q^{m,q}_0,m_f]\eta=\chi_{0}^{m,q+2}(\overline{\partial}f\wedge (\chi_0^{m,q+1})^{-1}\eta)+\chi_0^{m,q}((U_0^{m,q})^{-1}(i_{(\nabla_1f)^{0,1}}(U_0^{m,q+1}((\chi_0^{m,q+1})^{-1}\eta))))$$
we can conclude that \eqref{stronglimit} holds true also in the case $r=q+1$. Finally if $r\geq q+2$ we have 
$$[Q^{m,q}_s,m_f]\eta=[D_{F,m,r}^{\gamma,\gamma}+(D_{F,m,r-1}^{\gamma,\gamma})^*,m_f]\eta= [Q^{m,q}_0,m_f]\eta $$ for each $s\in [0,1]$. Thus \eqref{stronglimit} is obviously satisfied for $r\geq q+2$, which  completes the proof this lemma.
\end{proof}

\begin{rem}
\label{rem}
The condition $\dim(\sing(X))=0$ allows us to use $S_c(X)$ as a dense $*$-subalgebra of $C(X)$ and thus  the map \eqref{www} is bounded even when $s_1\neq 0$ and $s_2=0$. If 
$\dim(\sing(X))>0$ it is not clear to us how to define a dense $*$-subalgebra of $C(X)$ such that  the map \eqref{www} is bounded  when $s_1\neq 0$ and $s_2=0$. This problem might be overcome if one can replace Theorems \ref{compact1} and \ref{compact2} with a stronger convergence result, namely 
\begin{equation}
\label{might}
G_{\overline{\partial}_{E,m,q}^{g_s,g_s}}\rightarrow G_{\overline{\partial}_{E,m,q,\max}^{h,h}}
\end{equation}
 compactly as $s\rightarrow 0$. Indeed in this case one can simply consider as a dense $*$-subalgebra of $C(X)$ the space of smooth functions on $X$, see  \cite[Def. 1]{BeiPiazza}. Unfortunately, at the moment, we do not know how to prove \eqref{might}.
\end{rem}

We are now in a position to prove Th. \ref{tata}.
\begin{proof} (of Th. \ref{tata}).
We start by pointing out that  $\pi^*:L^2\Omega^{m,\bullet}(\reg(X),F,\sigma_1,\tau)\rightarrow L^2\Omega^{m,\bullet}(A,E,g_1,\rho)$ is an isometry that makes the even bounded Fredholm modules $$(L^2\Omega^{m,\bullet}(\reg(X),F,\sigma_1,\tau),m,Q_1^{m,m}\circ (\mathrm{Id}+(Q_1^{m,m})^2)^{-\frac{1}{2}})$$ and $$(L^2\Omega^{m,\bullet}(M,E,g,\rho),m\circ \pi^*,\overline{\eth}_{E,m}\circ (\mathrm{Id}+(\overline{\eth}_{E,m})^2)^{-\frac{1}{2}})$$ unitarily equivalent. Therefore $\pi_*[\overline{\eth}_{E,m}]=[Q_1^{m,m}]$. Thanks to Lemma \ref{parameter} we know that $[Q_1^{m,m}]=[Q_0^{m,m}]$ and by construction, see \eqref{uurolledup}, we have $[Q_0^{m,m}]=[Q_1^{m,m-1}]$. Again by Lemma \eqref{parameter} we have $[Q_1^{m,m-1}]=[Q_0^{m,m-1}]$ and therefore $[Q_1^{m,m}]=[Q_0^{m,m-1}]$. Applying  this procedure iteratively we can conclude that $[Q_1^{m,m}]=[Q_0^{m,r}]$ for each $r=0,...,m$. In particular we have $[Q_1^{m,m}]=[Q_0^{m,0}]$. Finally note that $\chi_0^{m,\bullet}:L^2\Omega^{m,\bullet}(\reg(X),F,\gamma,\tau)\rightarrow L^2\Omega^{m,\bullet}(\reg(X),F,\sigma_1,\tau)$ is an isometry that turns the even bounded Fredholm modules 
$$(L^2\Omega^{m,\bullet}(\reg(X),F,\sigma_1,\tau),m,Q_0^{m,0}\circ (\mathrm{Id}+(Q_0^{m,0})^2)^{-\frac{1}{2}})$$ and $$(L^2\Omega^{m,\bullet}(\reg(X),F,\gamma,\tau),m,\overline{\eth}_{F,m,\mathrm{abs}}\circ (\mathrm{Id}+(\overline{\eth}_{F,m,\mathrm{abs}})^2)^{-\frac{1}{2}})$$ unitarily equivalent. Therefore, $[Q_0^{m,0}]=[\overline{\eth}_{F,m,\mathrm{abs}}]$ and so, we can thus conclude that $$\pi_*[\overline{\eth}_{E,m}]=[\overline{\eth}_{F,m,\mathrm{abs}}]\quad \mathrm{in}\ KK_0(C(X),\mathbb{C})$$ as desired.
\end{proof}

\end{document}